\documentclass[a4paper,11pt]{amsart}

\usepackage{url}
\usepackage{color}
\definecolor{darkgreen}{rgb}{0,0.5,0}
\usepackage[
        colorlinks, citecolor=darkgreen,
        backref,
        pdfauthor={J. Steffen Mueller, Michael Stoll},
        pdftitle={Canonical Heights on Genus Two Jacobians}
]{hyperref}

\usepackage{amssymb,amsmath, amsthm}
\usepackage{comment}
\usepackage{algorithmic,algorithm}
\usepackage[OT2,OT1]{fontenc}
\usepackage{graphicx}
\usepackage{epsfig}
\usepackage{fancyhdr}
\usepackage{enumerate}
\usepackage{colonequals}

\usepackage[alphabetic,backrefs]{amsrefs} % for bibliography
\usepackage[all]{xy}

\newcommand\cyr{%
  \renewcommand\rmdefault{cmr}%
  \renewcommand\sfdefault{wncyss}%
  \renewcommand\encodingdefault{OT2}%
  \normalfont\selectfont}
\DeclareTextFontCommand{\textcyr}{\cyr}

\definecolor{red}{rgb}{0.9,0,0}
\definecolor{purple}{rgb}{0.8,0,0.6}

\numberwithin{equation}{section}

\newtheorem{thm}{Theorem}[section]

\newtheorem{prop}[thm]{Proposition}
\newtheorem{prob}[thm]{Problem}
\newtheorem{lemma}[thm]{Lemma}
\newtheorem{cor}[thm]{Corollary}
\newtheorem{ques}[thm]{Question}

\theoremstyle{definition}

\newtheorem{defn}[thm]{Definition}

\theoremstyle{remark}

\newtheorem{rk}[thm]{Remark}
\newtheorem{ex}[thm]{Example}

\newcommand\Q{\mathbb{Q}}
\newcommand\C{\mathbb{C}}
\newcommand\F{\mathbb{F}}

\newcommand\A{\mathbb{A}}
\newcommand\Z{\mathbb{Z}}
\newcommand\R{\mathbb{R}}

\newcommand\Gal{\mathop{\rm Gal}\nolimits}

\newcommand\Spec{\mathop{\rm Spec}\nolimits}

\renewcommand\O{\mathcal{O}}
\newcommand\J{\mathcal{J}}

\newcommand\Char{\mathop{\rm char}\nolimits}

\newcommand\supp{\mathop{\rm supp}\nolimits}

\newcommand\tors{\mathop{\rm tors}\nolimits}
\newcommand\isom{\cong}

\newcommand{\disc}{\operatorname{disc}}
\newcommand{\Div}{\operatorname{Div}}
\newcommand{\Pic}{\operatorname{Pic}}

\newcommand{\id}{\operatorname{id}}

\renewcommand{\div}{\operatorname{div}}
\newcommand{\GL}{\operatorname{GL}}
\newcommand\nr{\mathop{\rm nr}\nolimits}
\newcommand{\To}{\longrightarrow}
\newcommand{\BP}{{\mathbb P}}
\newcommand{\eps}{\varepsilon}

\newcommand{\calC}{\mathcal{C}}
\newcommand{\calE}{\mathcal{E}}
\newcommand{\calD}{\mathcal{D}}
\newcommand{\calJ}{\mathcal{J}}
\newcommand{\calK}{\mathcal{K}}
\newcommand{\KS}{\operatorname{KS}}

\newcommand{\frk}{\mathfrak{k}}

\newcommand{\pr}{\operatorname{pr}}
\newcommand{\surj}{\twoheadrightarrow}

\newcommand{\Mult}{\operatorname{\sf M}}
\newcommand{\std}{{\text{\rm std}}}
\newcommand{\const}{\text{\rm const.}}
\newcommand{\sm}{{\text{\rm sm}}}

\begin{document}

\title{Canonical Heights on Genus Two Jacobians}

\author{J. Steffen M\"uller}
\address{Institut f\"ur Mathematik,
          Carl von Ossietzky Universit\"at Oldenburg,
          26111 Oldenburg, Germany}
\email{jan.steffen.mueller@uni-oldenburg.de }

\author{Michael Stoll}
\address{Mathematisches Institut,
          Universit\"at Bayreuth,
          95440 Bayreuth, Germany.}
\email{Michael.Stoll@uni-bayreuth.de}

\date{August 2, 2016}

\dedicatory{{\normalsize Scale New Heights!} \\[2mm]
            {\footnotesize\rm (Motto of International (now Jacobs) University Bremen, \\
                     the place where the first author started his PhD under
                     the supervision of the second author)}
           }

%=========================================================================

\begin{abstract} \setlength{\parskip}{1ex} \setlength{\parindent}{0mm}
  Let $K$ be a number field and let $C/K$ be a curve of genus~$2$ with
  Jacobian variety~$J$. In this paper, we study the canonical height
  $\hat{h} \colon J(K) \to \R$. More specifically, we consider the following
  two problems, which are important in applications:
  \begin{enumerate}[(1)]
    \item for a given $P \in J(K)$, compute $\hat{h}(P)$ efficiently;
    \item for a given bound $B > 0$, find all $P \in J(K)$ with $\hat{h}(P) \le B$.
  \end{enumerate}
  We develop an algorithm running in polynomial time (and fast in practice)
  to deal with the first problem. Regarding the second problem, we show how
  one can tweak the naive height~$h$ that is usually used to obtain
  significantly improved bounds for the difference $h - \hat{h}$,
  which allows a much faster enumeration of the desired set of points.

  Our approach is to use the standard decomposition of $h(P) - \hat{h}(P)$
  as a sum of local `height correction functions'. We study these functions
  carefully, which leads to efficient ways of computing them and to essentially
  optimal bounds. To get our polynomial-time algorithm, we have to avoid
  the factorization step needed to find the finite set of places where the
  correction might be nonzero. The main innovation at this point is to
  replace factorization into primes by factorization into coprimes.

  Most of our results are valid for more general fields with a set of absolute
  values satisfying the product formula.
\end{abstract}

\maketitle

%=========================================================================

\vfill\pagebreak

\tableofcontents

\vfill\pagebreak

%=========================================================================

\section{Introduction}\label{intro}

Let $K$ be a global field and let $C/K$ be a
curve of genus~$2$ with Jacobian variety~$J$. There is a map $\kappa \colon J \to \BP^3$ that
corresponds to the class of twice the theta divisor on~$J$; it identifies a point
on~$J$ with its negative, and its image is the Kummer surface~$\KS$ of~$J$.
Explicit versions of~$\kappa$ can be found in the book~\cite{CasselsFlynn} by
Cassels and Flynn for $C$ given in the form $y^2 = f(x)$ and in the
paper~\cite{MuellerKummer} by the first author for general~$C$ (also in characteristic~2).
Thus $\kappa$ gives rise to a height function $h \colon J(K) \to \R$, which we call
the \emph{naive height} on~$J$. It is defined by
\[ h(P) = \sum_{v \in M_K} \log \max\{|\kappa_1(P)|_v, |\kappa_2(P)|_v,
                                      |\kappa_3(P)|_v, |\kappa_4(P)|_v\}\, ,
\]
where $M_K$ is the set of places of~$K$,
$\kappa(P) = (\kappa_1(P) : \kappa_2(P) : \kappa_3(P) : \kappa_4(P))$,
and $|{\cdot}|_v$ is the $v$-adic absolute value, normalized so that the product
formula
\[ \prod_{v \in M_K} |x|_v = 1 \qquad \text{for all $x \in K^\times$} \]
holds.

By general theory~\cite{HindrySilverman}*{Chapter~B} the limit
\[ \hat{h}(P) = \lim_{n \to \infty} \frac{h(nP)}{n^2} \]
exists; it is called the \emph{canonical height} (or \emph{N\'eron-Tate height}) of $P \in J(K)$.
The difference $h - \hat{h}$ is bounded.
The canonical height induces a positive definite quadratic form on $J(K)/J(K)_{\tors}$
(and on the $\R$-vector space $J(K)\otimes_\Z \R$).

In this paper, we tackle the following two problems:

\begin{prob}\label{prob-comp}
Find an efficient algorithm for the computation of $\hat{h}(P)$ for a given point $P \in J(K)$.
\end{prob}

\begin{prob}\label{prob-enum}
  Find an efficient algorithm for the enumeration of all $P \in J(K)$ which satisfy
    $\hat{h}(P) \le B$, where $B$ is a given real number.
\end{prob}

These problems are important because such algorithms are needed if we want to saturate
a given finite-index subgroup of~$J(K)$ (see the discussion at the end of
Section~\ref{S:enum}). This, in turn, is necessary for the computation of generators of $J(K)$.
Such generators are required, for instance, to carry out the method described in
\cite{BMSST08} for the computation of all integral points on a hyperelliptic curve over~$\Q$.
Furthermore, the regulator of $J(K)$ appearing in the conjecture of Birch and
Swinnerton-Dyer is the Gram determinant of a set of generators of
$J(K)/J(K)_{\tors}$ with respect to the canonical height.
So Problem~\ref{prob-comp} and Problem~\ref{prob-enum} are also important in the context
of gathering numerical evidence for this conjecture as in~\cite{FLSSSW}.

It is a classical fact, going back to work by N\'eron~\cite{Neron}, that $\hat{h}(P)$ and
the difference $h(P) - \hat{h}(P)$ can be decomposed into a finite sum of local terms.
In our situation, this can be done explicitly as follows.
The duplication map $P \mapsto 2P$ on~$J$ induces a morphism $\delta \colon \KS \to \KS$,
given by homogeneous polynomials $(\delta_1, \delta_2, \delta_3, \delta_4)$ of degree~4;
explicit equations can again be found in~\cite{CasselsFlynn} and~\cite{MuellerKummer}.
For a point $Q \in J(K_v)$, where $K_v$ is the completion of $K$ at
a place $v\in M_K$, such that $\kappa(Q) = (x_1 : x_2 : x_3 : x_4) \in \KS(K_v)$,
we set
\[ \tilde\eps_v(Q) = -\log \max\{|\delta_j(x_1,x_2,x_3,x_4)|_v : 1 \le j \le 4\}
                      + 4 \log \max\{|x_j|_v : 1 \le j \le 4\}\, .
\]
Note that this does not depend on the scaling of the coordinates.
We can then write $\hat{h}(P)$ in the following form (compare Lemma~\ref{L:telescope}):
\[ \hat{h}(P) = h(P) - \sum_{v \in M_K} \sum_{n=0}^{\infty} 4^{-(n+1)} \tilde\eps_v(2^n P) \]
We set, for $Q \in J(K_v)$ as above,
\begin{equation} \label{E:tildemu}
  \tilde\mu_v(Q) = \sum_{n=0}^{\infty} 4^{-(n+1)} \tilde\eps_v(2^n Q) \,,
\end{equation}
  and we deduce the decomposition
\begin{equation}\label{E:decomp}
  h(P) - \hat{h}(P) = \sum_{v \in M_K} \tilde\mu_v(P)\,,
\end{equation}
which is valid for all points $P \in J(K)$.
In addition, $\tilde\eps_v = \tilde\mu_v = 0$ for all but finitely many~$v$ (the exceptions
are among the places of bad reduction, the places where the given equation of $C$ is
not integral and the archimedean places).
The maps $\tilde\eps_v \colon J(K_v) \to \R$ are continuous maps (with respect to the $v$-adic
topology) with compact domains, so they are bounded. Therefore $\tilde\mu_v$ is also
bounded.

Let us first discuss Problem~\ref{prob-comp}.
Because of equation~\eqref{E:decomp}, it suffices to compute $h(P)$ (which is easy) and $\sum_{v \in M_K}
\tilde\mu_v(P)$ in order to compute $\hat{h}(P)$ for a point $P \in J(K)$.
Building on earlier work of Flynn and Smart~\cite{FlynnSmart}, the second author introduced an algorithm for the
computation of $\tilde\mu_v(P)$ in~\cite{StollH2}.
One of the main problems with this approach is that we need integer factorization to compute the sum
$\tilde{\mu}^{\text{f}}(P) := \sum_v \tilde\mu_v(P)$, where $v$ runs through the finite
primes $v$ such that $\tilde\mu_v(P) \ne 0$, because we need to find these primes, or at
least a finite set of primes containing them.

We use an idea which was already exploited in~\cite{MuellerStollEll} to obtain a
polynomial-time algorithm for the computation of the canonical height of a point on an
elliptic curves (in fact we first used this technique in genus~2 and only later realized
that it also works, and is actually easier to implement, for elliptic curves).
When $v$ is non-archimedean, then there is a constant $c_v >0$ such that the function
\[
  \mu_v := \tilde\mu_v/c_v
\]
maps $J(K_v)$ to $\Q$.
More precisely, $\tilde{\mu}^{\text{f}}(P)$ is a sum of rational multiples of logarithms of
positive integers.
As in~\cite{MuellerStollEll}, we find a bound on the denominator of~$\mu_v$ that depends
only on the valuation of the discriminant; this allows us to devise
an algorithm that computes $\tilde{\mu}^{\text{f}}(P)$ in quasi-linear time.
We can compute $\tilde\mu_v(P)$ for archimedean $v$ essentially from the definition of
$\tilde\mu_v$.
This leads to a factorization-free algorithm that computes $\hat{h}(P)$ in polynomial
time: % removed "More precisely"

\begin{thm}\label{T:IntroAlgo}
  Let $J$ be the Jacobian of a curve of genus~$2$ defined over~$\Q$, and let $P \in J(\Q)$.
  There is an algorithm that computes $\hat{h}(P)$ in time quasi-linear in the size of
  the coordinates of~$P$ and the coefficients of the given equation of $C$,
  and quasi-quadratic in the desired number of digits of precision.
\end{thm}

See Theorem~\ref{T:PolyAlgo} for a precise statement. We expect a similar result
to be true for any number field~$K$ in place of~$\Q$.

We now move on to Problem~\ref{prob-enum}.
If we have an upper bound $\beta$ for $h - \hat{h}$, then the set of all
points $P \in J(K)$ such that $h(P) \le B + \beta$ contains the set $\{P \in J(K)\,:\,
\hat{h}(P) \le B\}$.
Since the naive height $h$ is a logarithmic height, $\beta$ contributes exponentially to
the size of the box we need to search for the enumeration. Therefore it is crucial to keep $\beta$ as small as
possible.

We write $\tilde\beta_v = \max\{\tilde\mu_v(Q) : Q \in J(K_v)\}$, and we obtain the bound
\[ h(P) - \hat{h}(P) \le \sum_{v \in M_K} \tilde\beta_v\]
from~\eqref{E:decomp}.
If we write $\tilde\gamma_v = \max\{\tilde\eps_v(Q) : Q \in J(K_v)\}$, then
clearly $\tilde{\gamma}_v/4 \le \tilde{\beta}_v \le \tilde{\gamma}_v/3$.
In~\cite{StollH1}, it is shown that for curves
given in the form $y^2 = f(x)$, where $f$ has $v$-adically integral coefficients,
we have
\[
 \tilde{\gamma}_v \le - \log |2^4 \disc(f)|_v  = -\log |2^{-4} \Delta|_v\,,
 \]
with $\disc(f)$ denoting the discriminant of~$f$ considered as a polynomial of degree~6
and $\Delta$ denoting the discriminant of the given equation of $C$.
When $v$ is non-archimedean and the normalized additive valuation of $\Delta$ is~1, then
we can take $\tilde\gamma_v = \tilde\beta_v = 0$~\cite{StollH2}.

The results of the present paper improve on this; they are based on a careful study of the functions
$\tilde\mu_v$.
It turns out that when $v$ is non-archimedean, the set of points where $\mu_v$ (or
equivalently, $\tilde\mu_v$) vanishes forms a group.
Moreover, the function $\mu_v$ factors through the component group of the N\'eron model of $J$ when the
given model of $C/K_v$, which we assume to have $v$-integral coefficients in the
following, has rational singularities; see Theorem~\ref{T:epsfac}.
If the minimal proper regular model of $C$ is semistable, then we can use results of Zhang and
Heinz to give explicit formulas for $\mu_v$
in terms of the resistance function on the reduction graph of $C$ (which is essentially the dual graph of the special
fiber of the minimal proper regular model, suitably metrized).
We use this to find simple explicit formulas for~$\mu_v$ that apply in the most frequent
cases of bad reduction, namely nodal or cuspidal reduction.
These explicit formulas give us the optimal bounds for $\tilde\mu_v$ in these cases.
By reducing to the semistable case and tracking how $\mu_v$ changes as we change the
Weierstrass equation of $C$, we deduce the general upper bound
\begin{equation}\label{upper-bd}
  \tilde\beta_v \le -\frac{1}{4}\log |\Delta|_v
\end{equation}
for non-archimedean $v$; see Theorem~\ref{T:UpperBd}.

When $v$ is archimedean, we also get a new bound for $\tilde\mu_v$ by iterating the bound
obtained by the second author in~\cite{StollH1}, leading to vast improvements for
$\tilde\beta_v$.
Combining the archimedean and non-archimedean bounds, we find a nearly optimal bound
$\beta$ for $h - \hat{h}$.

To get even smaller search spaces for the enumeration, we make use of the observation
that we can replace
the naive height $h$ by any function $h'$ such that $|h' - h|$ is bounded.
Using the results on nearly optimal bounds for $\mu_v$ and such a modified
naive height~$h'$ (which is also better suited than $h$ for the enumeration process itself)
we get a much smaller bound on the difference $h' - \hat{h}$ than what was previously possible.
This makes the enumeration feasible in many cases that were completely out of
reach so far.

As an example, we compute explicit generators for the Mordell-Weil group of the Jacobian of the curve
\begin{align}
  C \colon y^2 &= 82342800 x^6 - 470135160 x^5 + 52485681 x^4 \label{record-curve} \\
              &\qquad{} + 2396040466 x^3 + 567207969 x^2 - 985905640 x +
  247747600\nonumber
\end{align}
over $\Q$, conditional on the Generalized Riemann Hypothesis (which is needed to show that the rank
is~22). See Proposition~\ref{P:record}.
This curve has at least~642 rational points, which is the current record for the largest
number of known rational points on a curve of genus~2, see~\cite{record}.

The paper is divided into four parts.
In Part~I, we first generalize the usual notion of the naive height on projective space
and clarify the relation between these generalized naive heights and suitable canonical heights, all in
Section~\ref{S:genhts}.
We then introduce local height correction functions $\eps$ and $\mu$ ($=\mu_v$ in the
notation introduced above) on the Jacobian of a
genus~2 curve over a non-archimedean local field in Section~\ref{genhts}.
This is followed in Section~\ref{S:lambda_Kummer} by a study of certain canonical local heights constructed in terms of $\mu$.
We close Part~I by introducing and investigating the notion of stably minimal Weierstrass
models of curves of genus~2 in Section~\ref{simpmod} and recalling some well-known results
on Igusa invariants in Section~\ref{igusa}.

Part~II is in some sense the central part of the present paper.
Here we study the local height correction function $\mu$ over a non-archimedean local
field.
Using Picard functors, we show in Section~\ref{kermu} that $\mu$ factors through the component group of the N\'eron
model of the Jacobian when the given model of the curve has rational singularities.
We then relate $\mu$ to the reduction graph of $C$ in Section~\ref{nerong2}.
Building on this, the following sections contain simple explicit formulas for $\mu$ when the reduction
of the curve is nodal (Section~\ref{formulas}), respectively cuspidal (Section~\ref{formulas2}).
A simple argument then gives the improved general upper bound~\eqref{upper-bd} for $\mu$, see Section~\ref{UpperBeta}.

In Part~III we describe our factorization-free algorithm for the computation of
$\hat{h}(P)$ for $P \in J(K)$, where $K$ is a global field.
We start in Section~\ref{algo1} by showing how to compute $\mu_v(P)$ for non-archimedean $v$,
using a bound on its denominator.
The following section deals with archimedean places, before we finally combine these
results  in Section~\ref{S:cch} into an algorithm for the computation of $\hat{h}(P)$
that runs in polynomial time; this proves Theorem~\ref{T:IntroAlgo}. Some examples are
discussed in Section~\ref{Ex:canht}.

In the final Part~IV we turn to Problem~\ref{prob-enum}.
Section~\ref{S:htdiffarch} contains two methods for bounding $\tilde\mu_v$ for
archimedean $v$.
In the following Section~\ref{S:varnaive} we describe a modified naive height $h'$
such that the bound on the difference $h'-\hat{h}$ becomes small.
We use this, the results of Section~\ref{S:htdiffarch}, and our nearly optimal bounds
for the non-archimedean height correction functions from Part~II to give an efficient algorithm for the
enumeration of the set of rational points with bounded canonical height in
Section~\ref{S:enum}.
In the final Section~\ref{Ex:chnaive} we compute generators of the Mordell-Weil group of
the record curve~\eqref{record-curve}.

\subsection*{Acknowledgments}
We would like to thank David Holmes for suggesting the strategy of the proof of
Proposition~\ref{P:mu_vanishes}, Elliot Wells for pointing out an
inaccuracy in the complexity analysis in Propositions~\ref{P:mu_approx}
and~\ref{P:nofact}, and the anonymous referee for some useful remarks and suggestions.

%=====================================================================================

\vfill\pagebreak

\section*{\large Part~I: Generalities on Heights and Genus Two Jacobians}

\section{Generalized naive heights} \label{S:genhts}

Let $K$ be a field with a set~$M_K$ of places~$v$ and associated absolute values~$|{\cdot}|_v$
satisfying the product formula
\[ \prod_{v \in M_K} |x|_v = 1 \qquad \text{for all $x \in K^\times$ .} \]
We write $K_v$ for the completion of~$K$ at~$v$. For a tuple
$x = (x_1, \ldots, x_m) \in K_v^m$ we set $\|x\|_v = \max\{|x_1|_v, \ldots, |x_m|_v\}$.

In the following we will introduce some flexibility into our notion of height on
projective spaces. (This is similar to the framework of `admissible families'
in~\cite{Zarhin}.)

\begin{defn} \strut
  \begin{enumerate}[(1)]\addtolength{\itemsep}{2mm}
    \item Let $v \in M_K$. A \emph{local height function} on~$\BP^m$ at~$v$
          is a map $h_v \colon K_v^{m+1} \setminus \{0\} \to \R$ such that
          \begin{enumerate}[(i)]
            \item $h_v(\lambda x) = \log |\lambda|_v + h_v(x)$
                  for all $x \in K_v^{m+1} \setminus \{0\}$ and all $\lambda \in K_v^\times$, and
            \item $\bigl|h_v(x) - \log \|x\|_v\bigr|$ is bounded.
          \end{enumerate}
    \item A function $h \colon \BP^m(K) \to \R$ is a  \emph{height} on~$\BP^m$ over~$K$ if
          there are local height functions~$h_v$  such that for all $x \in \BP^m(K)$ we have
          \[ h\bigl((x_1 : x_2 : \ldots : x_{m+1})\bigr)
              = \sum_{v \in M_K} h_v(x_1, x_2, \ldots, x_{m+1})
          \]
          and $h_v(x) = \log \|x\|_v$ for all but finitely many places~$v$.
  \end{enumerate}
\end{defn}

Note that property~(i) of local height functions together with the product formula
imply that $h$ is invariant under scaling of the coordinates and hence is well-defined.

One example of such a height is the standard height $h_\std$, which we obtain by
setting $h_v(x) = \log \|x\|_v$ for all~$v$. We then have the following simple fact.

\begin{lemma}
  Let $h$ be any height on~$\BP^m$ over~$K$ and let $h_\std$ be the standard height.
  Then there is a constant $c = c(h)$ such that
  \[ |h(P) - h_\std(P)| \le c \qquad \text{for all $P \in \BP^m(K)$.} \]
\end{lemma}

\begin{proof}
  This follows from property~(ii) of local height functions
  and the requirement that $h_v(x) = \log \|x\|_v$ for all but finitely many~$v$.
\end{proof}

\begin{ex} \label{Ex:nonstdht}
  Other examples of heights can be obtained in the following way.
  For each place~$v$, fix a linear form
  $l_v(x_1, \ldots, x_{m+1}) = a_{v,1} x_1 + \ldots + a_{v,m+1} x_{m+1}$
  with $a_{v,1}, \ldots, a_{v,m+1} \in K_v$ and $a_{v,m+1} \neq 0$, such that
  $l_v(x) = x_{m+1}$ for all but finitely many~$v$. Then
  \[ h\bigl((x_1 : \ldots : x_m : x_{m+1})\bigr)
      = \sum_{v \in M_K} \log \max\{|x_1|_v, \ldots, |x_m|_v, |l_v(x_1, \ldots, x_{m+1})|_v\}
  \]
  is a height on~$\BP^m$.

  More generally, we could consider a family of automorphisms $A_v$ of~$K_v^{m+1}$
  with $A_v$ equal to the identity for all but finitely many~$v$, and take
  \[ h(x) = \sum_{v \in M_K} \log \max \|A_v(x)\|_v \,. \]
\end{ex}

Now consider a projective variety $V \subset \BP^m_{K}$ and an
endomorphism $\varphi \colon V \to V$
of degree~$d$ (i.e., given by homogeneous polynomials of degree~$d$).
Then by general theory (see, e.g.,~\cite{HindrySilverman}*{Thm.~B.2.5})
$|h_\std(\varphi(P)) - d h_\std(P)|$ is bounded on~$V(K)$.
We write $\varphi^{\circ n}$ for the $n$-fold iteration of~$\varphi$.
Then the \emph{canonical height}
\[ \hat{h}(P) = \lim_{n \to \infty} d^{-n} h_\std(\varphi^{\circ n}(P)) \]
exists (and satisfies $\hat{h}(\varphi(P)) = d \hat{h}(P)$)~\cite{HindrySilverman}*{Thm.~B.4.1}.
Let $h$ be any height on~$\BP^m$. Since $|h - h_\std|$ is bounded, we can replace $h_\std$
by~$h$ in the definition of~$\hat{h}$ without changing the result.
We can then play the usual telescoping series trick in our more general setting.

\begin{lemma} \label{L:telescope}
  Let $\varphi\bigl((x_1 : \ldots : x_{m+1})\bigr) = \bigl(\varphi_1(x) : \ldots : \varphi_{m+1}(x)\bigr)$
  with homogeneous polynomials $\varphi_j \in K[x_1, \ldots, x_{m+1}]$ of degree~$d$. We have
  \[ \hat{h}(P) = h(P) - \sum_{v \in M_K} \tilde{\mu}_v(P) \,, \]
  where
  \[ \tilde{\mu}_v(P) = \sum_{n=0}^\infty d^{-(n+1)} \tilde{\eps}_v(\varphi^{\circ n}(P)) \]
  and, when $P = (x_1 : \ldots : x_{m+1})$ and $x = (x_1, \ldots, x_{m+1})$,
  \[ \tilde{\eps}_v(P) = d h_v(x) - h_v\bigl(\varphi_1(x), \ldots, \varphi_{m+1}(x)\bigr) \,. \]
\end{lemma}

\begin{proof}
  Note that $\tilde{\eps}_v$ is well-defined: scaling $x$ by~$\lambda$ adds $|\lambda|_v$
  to~$h_v(x)$ and $d |\lambda|_v$ to $h_v(\varphi_1(x), \ldots, \varphi_{m+1}(x))$.
  Let $x$ be projective coordinates for~$P$ and
  write $x^{(n)}$ for the result of applying $(\varphi_1, \ldots, \varphi_{m+1})$
  $n$~times to~$x = x^{(0)}$. Then
  \begin{align*}
    \hat{h}(P) &= \lim_{n \to \infty} d^{-n} h(\varphi^{\circ n}(P)) \\
               &= h(P) + \sum_{n=0}^\infty d^{-(n+1)} \bigl(h(\varphi^{\circ (n+1)}(P))
                                                              - d h(\varphi^{\circ n}(P))\bigr) \\
               &= h(P) + \sum_{n=0}^\infty d^{-(n+1)}
                            \sum_{v \in M_K} \bigl(h_v(x^{(n+1)}) - d h_v(x^{(n)})\bigr) \\
               &= h(P) - \sum_{v \in M_K} \sum_{n=0}^\infty
                               d^{-(n+1)} \tilde{\eps}_v(\varphi^{\circ n}(P)) \\
               &= h(P) - \sum_{v \in M_K} \tilde{\mu}_v(P) \,. \qedhere
  \end{align*}
\end{proof}

We call the functions $\tilde{\mu}_v \colon \BP^m(K_v) \to \R$ \emph{local height
correction functions}.

Note that when $K_v$ is a discretely valued field such that
$|x|_v = \exp(-c_v v(x))$ for $x \in K^\times$ with a constant $c_v > 0$ (and where
we abuse notation and write $v \colon K_v^\times \surj \Z$ also for the normalized
additive valuation associated to the place~$v$) and $h = h_\std$, then we have
\[ \tilde{\mu}_v(P) = c_v \mu_v(P) \qquad\text{and}\qquad
   \tilde{\eps}_v(P) = c_v \eps_v(P) \,,
\]
where
\[ \mu_v(P) = \sum_{n=0}^\infty d^{-(n+1)} \eps_v(P) \]
and
\[ \eps_v(P) = \min \bigl\{v(\varphi_1(x)), \ldots, v(\varphi_{m+1}(x))\bigr\}
                 - d \min \{v(x_1), \ldots, v(x_{m+1})\} \,,
\]
if $x = (x_1, \ldots, x_{m+1})$ are homogeneous coordinates for~$P$.
This is the situation that we will study in some detail in Part~II of this paper,
for the special case when $V \subset \BP^3$ is the Kummer surface associated to
a curve of genus~$2$ and its Jacobian~$J$ and $\varphi$ is the duplication map
(then $d = 4$).

To deal with Problem~\ref{prob-comp}, we work with the standard height~$h_\std$.
We use our detailed
results on the local height correction functions to deduce a bound on the denominator
of~$\mu_v$ (its values are rational) in terms of the valuation of the discriminant
of the curve. This is the key ingredient that leads to our new factorization-free
and fast algorithm for computing~$\hat{h}$, see Part~III.

To deal with Problem~\ref{prob-enum}, we use the flexibility in choosing the
(naive) height~$h$ and modify the standard height in such a way that the sum
$\sum_{v \in M_K} \sup \tilde{\mu}_v(J(K_v))$ that bounds the difference $h - \hat{h}$
is as small as we can make it. The local height functions we use are as in
Example~\ref{Ex:nonstdht} above, with $l_v(x_1,x_2,x_3,x_4) = x_4/s_v$ for
certain $s_v \in K_v^\times$ in most cases.
Every height function of this type has the property
that for any point \hbox{$P = (x_1 : x_2 : x_3 : x_4) \in \BP^3(K)$}
different from~$(0 : 0 : 0 : 1)$ we have
\[ 0 \le h_\std\bigl((x_1 : x_2 : x_3)\bigr) \le h(P) \,. \]
This is relevant, since we can fairly easily enumerate all points $P$ as above
that are on the Kummer surface and satisfy $h_\std\bigl((x_1 : x_2 : x_3)\bigr) \le B$,
see Part~IV.
Refinements of the standard height constructed using Arakelov theory were also used by
Holmes~\cite{Holmes14} to give an `in principle' algorithm for the enumeration of
points of bounded canonical height on Jacobians of hyperelliptic curves over global
fields.

%=====================================================================================

\section{Local height correction functions for genus 2 Jacobians} \label{genhts}

Until further notice, we let
$k$ be a non-archimedean local field with additive valuation~$v$,
normalized to be surjective onto $\Z$. Let $\O$ denote the valuation
ring of $k$ with residue class field $\frk$ and let $\pi$ be a
uniformizing element of $\O$. We consider a smooth projective
curve $C$ of genus~2 over $k$, given by a Weierstrass equation
\begin{equation}\label{ceq}
  Y^2 + H(X,Z) Y = F(X,Z)
\end{equation}
in weighted projective space $\BP_k(1,3,1)$, with weights 1, 3 and~1 assigned to the
variables $X$, $Y$ and~$Z$, respectively. Here
\[ F(X,Z) = f_0 Z^6 + f_1 X Z^5 + f_2 X^2 Z^4 + f_3 X^3 Z^3 + f_4 X^4 Z^2 + f_5 X^5 Z + f_6 X^6 \]
and
\[ H(X,Z) = h_0 Z^3 + h_1 X Z^2 + h_2 X^2 Z + h_3 X^3 \]
are binary forms of degrees~6 and~3, respectively, such that the discriminant
$\Delta(F, H)$ of the Weierstrass equation~\eqref{ceq} is nonzero.
In characteristic different from~2, this discriminant is defined as
\[ \Delta(F,H) = 2^{-12} \disc(4 F + H^2) \in \Z[h_0,\ldots,h_3, f_0,\ldots,f_6] \,, \]
and in general, we define it by the generic polynomial given by this formula.
The curve defined by the equation is smooth if and only if $\Delta(F,H) \neq 0$.

For the remainder of this section we assume that $F, H \in \O[X,Z]$,
so that equation~\eqref{ceq} defines an \emph{integral Weierstrass model} $\calC$ of the curve
in the terminology of Section~\ref{simpmod} below.
The discriminant of this model is then defined to be $\Delta(\calC) := \Delta(F,H)$.
We may assume that $C$ is given by such an integral equation if $k$ is the completion
at a non-archimedean place of a
number field $K$ and $C$ is obtained by base change from $K$, since we can choose a
globally integral Weierstrass equation for the curve.
But also in general, we can always assume that $C$ is given by an integral equation after applying
a transformation defined over $k$, since we know from Corollary~\ref{C:lambdatau} in the next section
how the local height correction function~$\mu$ defined in Definition~\ref{defepsmu}
below behaves under such transformations.

We now generalize the definition of $\eps$ given in~\cite{StollH2} to our
more general setting (\cite{StollH2} works with Weierstrass equations that have $H = 0$).
As in the introduction,
let $J$ denote the Jacobian of~$C$ and let $\KS$ be its Kummer surface, constructed
explicitly together with an explicit embedding into~$\BP^3$ in~\cite{CasselsFlynn}
in the case $H = 0$ and in~\cite{MuellerKummer} in the general case.
Also let $\kappa \colon J \to \BP^3$ denote the composition of the quotient map from $J$
to~$\KS$
with this embedding; it maps the origin $O \in J(k)$ to the point $(0:0:0:1)$.
A quadruple $x = (x_1,x_2,x_3,x_4) \in k^4$ is called a set of {\em Kummer coordinates}
on~$\KS$
if $x$ is a set of projective coordinates for a point in $\KS(k)$; we denote the set of sets of Kummer
coordinates on~$\KS$ by~$\KS_\A$ (this is the set of $k$-rational points on the pointed affine cone
over~$\KS$).
For $x \in \KS_\A$ we write
$v(x) = \min\{v(x_1),\ldots,v(x_4)\}$, and we say that $x$ is {\em normalized} if $v(x)=0$.
If $P \in J(k)$, we say that $x \in \KS_\A$ is a set of \emph{Kummer coordinates for~$P$}
if $\kappa(P) = (x_1 : x_2 : x_3 : x_4)$.

We let $\delta$ denote the duplication map on~$\KS$, which is given
by homogeneous polynomials $\delta_1, \ldots, \delta_4 \in \O[x_1, \ldots, x_4]$
of degree~$4$ such that $\delta(0,0,0,1) = (0,0,0,1)$.
We recall that there is a symmetric matrix $B = (B_{ij})_{1 \le i,j \le 4}$, where the
$B_{ij} \in \O[x_1, \ldots, x_4,\, y_1, \ldots, y_4]$ are bi-homogeneous of degree~2
in $x_1, \ldots, x_4$ and $y_1, \ldots, y_4$ each and have the following properties,
see~\cite{CasselsFlynn}*{Chapter~3} and~\cite{MuellerKummer}.
\begin{enumerate}[(i)]
  \item \label{propI}
        Let  $x, y \in \KS_\A$ be Kummer coordinates for $P, Q \in J(k)$. Then there are
        Kummer coordinates $w, z \in \KS_\A$ for $P+Q$ and~$P-Q$,
        respectively, such that
        \[ w \ast z \colonequals (w_i z_j + n_{ij} w_j z_i)_{1 \le i,j \le 4} = B(x, y) \]
        and hence $v(w) + v(z) = v\bigl(B(x,y)\bigr)$;
        here $n_{ij} = 1$ if $i \neq j$ and $n_{ij} = 0$ if $i = j$.
  \item If $x \in \KS_\A$, then $B(x, x) = \delta(x) * (0,0,0,1)$.
\end{enumerate}

We specialize the notions introduced in Section~\ref{S:genhts} to our situation:
we consider the Kummer surface $\KS \subset \BP^3$ with the duplication map~$\delta$
of degree $d = 4$. We use the standard local height on~$\BP^3$.
\begin{defn}\label{defepsmu}
  Let $x\in \KS_\A$ be a set of Kummer coordinates on $\KS$. Then we set
  \[ \eps(x) = v(\delta(x)) - 4 v(x) \in \Z \qquad\text{and}\qquad
      \mu(x) = \sum^{\infty}_{n=0} \frac{1}{4^{n+1}} \eps(\delta^{\circ n}(x)) \,,
  \]
\end{defn}
where $\delta^{\circ n}$ denotes the $n$-fold composition $\delta\circ\ldots\circ\delta$.

Because $\delta$ is given by homogeneous polynomials of degree~4, $\eps(x)$ does not
depend on the scaling of~$x$, so it makes sense to define $\eps(P) = \eps(x)$
for points $P \in \KS(k)$, where $x \in \KS_{\A}$ is any set of Kummer coordinates for~$P$,
and to define $\eps(P) = \eps(\kappa(P))$ for points $P \in J(k)$. We likewise
extend the definition of~$\mu$. Then we have
\[ \mu(2P) - 4 \mu(P) = -\eps(P) \qquad \text{for all $P \in J(k)$.} \]
Note that our assumption $F, H \in \O[X,Z]$ implies that $\eps \ge 0$.
If $k$ is a local field (as we assume here), then $\KS(k)$ is compact in the $v$-adic
topology, and $\eps$ is continuous, so $\eps$ is bounded.

\begin{rk} \label{R:eps_DVF}
  More generally,
  if $k$ is a field with a discrete valuation and not of characteristic~2, then the
  arguments in~\cite{StollH1} show that when $H = 0$, $\eps \le v(2^4 \disc(F))$,
  so $\eps$ is bounded also for these more general fields.

  If $k$ is any field with a discrete valuation, then one can still conclude that
  $\eps$ is bounded, by making use of the fact that the duplication map is well-defined
  on~$\KS$, which implies that the ideal generated by the~$\delta_j$ and the polynomial~$\delta_0$
  defining~$\KS$ contains a power of the irrelevant ideal. So for some $N > 0$, one can
  express every $x_j^N$ as a linear combination of $\delta_0(x), \ldots, \delta_4(x)$ with
  coefficients that are homogeneous polynomials of degree~$N-4$ with coefficients in~$k$.
  The negative of the minimum of the valuations of these coefficients then gives a bound for~$\eps$.
\end{rk}

\begin{rk}\label{R:cv}
  If $k$ is the completion of a global field at a place~$v$, then for $\alpha \in k^\times$,
  $v(\alpha)/\log \|\alpha\|_v = -c_v$ is a negative constant. So for $P \in J(k)$ we have
  $\eps(P) = c_v \tilde{\eps}_v(P)$ and $\mu(P) = c_v \tilde{\mu}_v(P)$,
  where $\tilde{\eps}_v$ and $\tilde{\mu}_v$ are as defined in the introduction.
\end{rk}

We will also have occasion to use the following function. Let $x,y \in \KS_\A$ and define
\begin{equation} \label{E:epsB}
  \eps(x,y) = v(B(x,y)) - 2 v(x) - 2 v(y) .
\end{equation}
In the same way as for $\eps(x)$ above, we can extend this to points in~$\KS(k)$ and~$J(k)$.

\begin{lemma}\label{lemtech}
  Let $x,y,w,z \in \KS_\A$ be Kummer coordinates satisfying $w \ast z = B(x,y)$. Then we have
  \[ \delta(w) \ast \delta(z) = B(\delta(x),\delta(y)) \,. \]
\end{lemma}

\begin{proof}
  The proof carries over verbatim from the proof of~\cite{StollH2}*{Lemma~3.2}.
\end{proof}

We deduce the following:

\begin{lemma}\label{epsxy}
  Let $x,y,w,z \in \KS_\A$ be Kummer coordinates satisfying $w \ast z = B(x,y)$. Then we have
  \[ \eps\bigl(\delta(x),\delta(y)\bigr) + 2 \eps(x) + 2 \eps(y) = \eps(w) + \eps(z) + 4
  \eps(x,y)\,. \]
\end{lemma}

\begin{proof}
  Using Lemma~\ref{lemtech}, relation~\eqref{E:epsB}, and property~\eqref{propI} above
  for $\delta(w)$, $\delta(z)$, $\delta(x)$ and~$\delta(y)$, we obtain
  \[ v\bigl(\delta(w)\bigr) + v\bigl(\delta(z)\bigr) = v\bigl(B(\delta(x), \delta(y))\bigr)
        = \eps\bigl(\delta(x), \delta(y)\bigr) + 2 v\bigl(\delta(x)\bigr) + 2
        v\bigl(\delta(y)\bigr) \,.
  \]
  Subtracting four times the corresponding relation for $w$, $z$, $x$ and~$y$, we get
  \[ \eps(w) + \eps(z) = \eps\bigl(\delta(x), \delta(y)\bigr) - 4 \eps(x, y) + 2 \eps(x) +
  2 \eps(y) \,, \]
  which is the claim.
\end{proof}

We state a few general facts on the functions $\eps$ and~$\mu$.

\begin{lemma} \label{L:mu eps rel}
  For points $P, Q \in J(k)$, we have the relation
  \[ \mu(P+Q) + \mu(P-Q) - 2 \mu(P) - 2 \mu(Q) = -\eps(P, Q) \,. \]
\end{lemma}

\begin{proof}
  Let $x$ and~$y$ be Kummer coordinates for $P$ and~$Q$, respectively; then $w$ and~$z$
  as in Lemma~\ref{epsxy} are Kummer coordinates for $P+Q$ and~$P-Q$ (in some order).
  The claim now follows from the formula in Lemma~\ref{epsxy}:
  \begin{align*}
    \mu(P+Q) &+ \mu(P-Q) - 2 \mu(P) - 2 \mu(Q) \\
      &= \sum_{n=0}^\infty 4^{-n-1}
          \bigl(\eps(2^n P + 2^n Q) + \eps(2^n P - 2^n Q) - 2 \eps(2^n P) - 2 \eps(2^n Q)\bigr) \\
      &= \sum_{n=0}^\infty 4^{-n-1}
          \bigl(\eps(\delta^{\circ n}(w)) + \eps(\delta^{\circ n}(z))
                  - 2 \eps(\delta^{\circ n}(x)) - 2 \eps(\delta^{\circ n}(y))\bigr) \\
      &= \sum_{n=0}^\infty 4^{-n-1}
          \bigl(\eps(\delta^{\circ(n+1)}(x),\, \delta^{\circ(n+1)}(y))
                - 4 \eps(\delta^{\circ n}(x),\, \delta^{\circ n}(y))\bigr) \\
      &= -\eps(x, y) = -\eps(P, Q) \,. \qedhere
  \end{align*}
\end{proof}

\begin{lemma} \label{L:mu=0}
  If $P \in J(k)$ satisfies $\mu(P) = 0$, then $\mu(P+Q) = \mu(Q)$ for all $Q \in J(k)$.
\end{lemma}

\begin{proof}
  We apply Lemma~\ref{L:mu eps rel} with $P$ and~$Q$ replaced by $Q+nP$ and~$P$,
  respectively, where $n \in \Z$. Taking into account
  that $\mu(P) = 0$ and writing $a_n$ for~$\mu(Q + nP)$, this gives
  \[ a_{n+1} - 2 a_n + a_{n-1} = -\eps(P, Q+nP) \,. \]
  As $k$ is a non-archimedean local field, the multiples of~$P$ accumulate at the origin
  $O \in J(k)$. Recall that $\eps$ is locally constant. This implies that every value
  $\eps(P, Q+nP)$ occurs for infinitely many $n \in \Z$, since
  $Q+(n+N)P$ will be close to~$Q+nP$ for suitably chosen~$N$. We have for any $m > 0$
  \[ a_{m+1} - a_{m} - a_{-m} + a_{-m-1}
       = \sum_{n=-m}^{m} (a_{n+1} - 2 a_n + a_{n-1})
       = - \sum_{n=-m}^{m} \eps(P, Q+nP) \,.
  \]
  Since $\mu$ is bounded, the left hand side is bounded independently of~$m$. We also know
  that $\eps(P, Q+nP) \ge 0$. But if $\eps(P, Q+nP)$ were nonzero for some~$n$, then
  by the discussion above, the right hand side would be unbounded as $m \to \infty$.
  Therefore it follows that $\eps(P, Q+nP) = 0$ for all $n \in \Z$.
  This in turn implies $a_{n+1} - 2 a_n + a_{n-1} = 0$ for all $n \in \Z$.
  The only bounded solutions of this recurrence are constant sequences. In particular, we have
  \[ \mu(P + Q) = a_1 = a_0 = \mu(Q) \,. \qedhere \]
\end{proof}

\begin{prop} \label{C:finite quotient}
  The subset $U = \{P \in J(k) : \mu(P) = 0\}$ is a subgroup of finite index in~$J(k)$.
  The functions $P \mapsto \eps(P)$ and $P \mapsto \mu(P)$
  factor through the quotient $J(k)/U$.
\end{prop}

\begin{proof}
  Lemma~\ref{L:mu=0} shows that $U$ is a subgroup.
  We have $\eps(P) = 0$ for $P \in J(k)$ sufficiently close to the origin. So taking a
  sufficiently small subgroup neighborhood~$U'$ of the origin in~$J(k)$, we see that
  $\eps(2^n P) = 0$ for all $P \in U'$ and all $n \ge 0$. This implies that $\mu = 0$
  on~$U'$, so $U \supset U'$. Because $k$ is a local field, $U'$ and therefore also~$U$
  have finite index in~$J(k)$. By Lemma~\ref{L:mu=0} again,
  $\mu$ factors through $J(k)/U$, and since $\eps(P) = 4 \mu(P) - \mu(2P)$, the same
  is true for~$\eps$.
\end{proof}

We will now show that we actually have
\[ U = \{P \in J(k) : \eps(P) = 0\} \]
(the inclusion `$\subset$' is clear from the definition and Proposition~\ref{C:finite quotient}.)
This is equivalent to the implication $\eps(x) = 0 \Longrightarrow \eps(\delta(x)) = 0$
and generalizes~\cite{StollH2}*{Thm.~4.1}.
For this we first provide a characteristic~2 analogue of~\cite{StollH2}*{Prop.~3.1(1)}.

We temporarily let $k$ denote an arbitrary field.
Let $C_{F,H}$ be a (not necessarily smooth) curve in the weighted projective plane
with respective weights 1, 3, 1 assigned to the variables $X,Y,Z$ that is given by an equation
\begin{equation}\label{E:cfheq}
Y^2 + H(X,Z) Y = F(X,Z),
\end{equation}
where $F, H \in k[X,Z]$
are binary forms of respective degrees~6 and~3. Let $\KS_{F,H}$ denote
the subscheme of $\BP^3$ given by the vanishing of the equation defining the
Kummer surface of $C_{F,H}$ if $C_{F,H}$ is nonsingular. Then the construction
of $\delta=(\delta_1,\delta_2,\delta_3,\delta_4)$ still makes sense
in this context, but we may now have $\delta_i(x)=0$ for all $1 \le i \le4$
(which we abbreviate by $\delta(x) = 0$)
for a set~$x$ of Kummer coordinates on~$\KS_{F,H}$.
We generalize Proposition~3.1 in~\cite{StollH2} (which assumes $H = 0$) to the case considered here.

Note that two equations~\eqref{E:cfheq} for $C_{F,H}$ are related by a transformation
$\tau$ acting on an affine point $(\xi,\eta)$ by
\begin{equation}\label{tau}
  \tau(\xi,\eta) = \left(\frac{a\xi+b}{c\xi+d}, \frac{e\eta + U(\xi,1)}{(c\xi+d)^3} \right),
\end{equation}
where $A = \left(\begin{smallmatrix} a & b \\ c & d \end{smallmatrix}\right)\in \GL_2(k)$,
$e \in k^\times$ and $U \in k[X,Z]$ is homogeneous of degree~3.
The transformation $\tau$ also acts on the forms $F$ and $H$ by
\begin{align*}
  \tau^*F(X,Z) &= (ad-bc)^{-6} \left(e^2 F^A + (e H^A - U^A)\, U^A\right)\\
  \tau^*H(X,Z) &= (ad-bc)^{-3} \left(e H^A - 2 U^A\right)\,,
\end{align*}
where we write
\[ S^A = S(dX-bZ, -cX+aZ) \]
for a binary form $S \in k[X,Z]$.

\begin{lemma}\label{delBvan}
  Let $x\in \KS_{F,H}(k)$. If $\delta(\delta(x))=0$, then we already have $\delta(x)=0$.
\end{lemma}

\begin{proof}
  If $k$ has characteristic different from~2, we can apply a transformation so that
  the new Weierstrass equation will have $H = 0$; the statement is then~\cite{StollH2}*{Prop.~3.1(1)}.
  So from now on, $k$ has characteristic~2.
  We may assume without loss of generality that $k$ is algebraically closed.
  If the given curve is smooth, then the result is obvious, because the situation
  described in the statement can never occur. If it is not smooth,
  we can act on $F$ and $H$ using transformations of the form~\eqref{tau}, so it is enough
  to consider only one representative of each orbit under such transformations.
  This is analogous to the strategy in the proof of~\cite{StollH2}*{Prop.~3.1}. We can, for
  example, pick the representatives listed in Table~\ref{T:reps}.

  \begin{table}
    \begin{tabular}{|c|c|c|c|} \hline
      type & $H$           & $F$             & conditions          \\ \hline
        1  & 0             & 0               &                     \\
        2  & $Z^3$         & 0               &                     \\
        3  & $Z^3$         & $aXZ^5$         & $a\ne 0$            \\
        4  & $XZ^2$        & $aXZ^5$         & $a\ne 0$            \\
        5  & $XZ^2$        & $bX^3Z^3$       & $b\ne 0$            \\
        6  & $Z^3$         & $aXZ^5+bX^3Z^3$ & $ab \ne0$           \\
        7  & $XZ^2$        & 0               &                     \\
        8  & $X Z (X + Z)$ & 0               &                     \\
        9  & $X Z (X + Z)$ & $bX^3Z^3$       & $b(b+1)\ne 0$       \\
       10  & $X Z (X + Z)$ & $aXZ^5+bX^3Z^3$ & $a(a+b)(a+b+1)\ne0$ \\
       11  & $XZ^2$        & $aXZ^5+bX^3Z^3$ & $ab \ne0$           \\
       12  & 0             & $XZ^5$          &                     \\
       13  & 0             & $X^3Z^3$        &                     \\ \hline
    \end{tabular}
    \medskip
    \caption{Representatives in characteristic~2}
    \label{T:reps}
  \end{table}
  For these representatives, elementary methods as in the
  proof of~\cite{StollH2}*{Prop.~3.1} can be used to check that $\delta(x) = 0$ indeed
  follows from  $\delta(\delta(x))= 0$.
\end{proof}

We can use the above to analyze the group $U$.

\begin{thm}\label{T:mu_U}
  Suppose that~$k$ is a non-archimedean local field and that $J$ is the Jacobian of a
  smooth projective
  curve of genus~2, given by a Weierstrass equation~\eqref{ceq} with integral
  coefficients.
  Then the set $\{P \in J(k) : \eps(P) = 0\}$ equals the subgroup~$U$
  in Proposition~\ref{C:finite quotient}. In particular, $U$ is a subgroup of finite index in $J(k)$
  and $\eps$ and~$\mu$ factor through the quotient $J(k)/U$. Moreover we have that
  $\eps(-P) = \eps(P)$ and $U$ contains the kernel of reduction $J(k)^1$ with respect to the given model of
$J$,  i.e., the subgroup of points whose image in~$\KS(\frk)$ equals that of~$O$.
\end{thm}

\begin{proof}
  The statement in Lemma~\ref{delBvan} implies $\eps(P) = 0 \Longrightarrow \eps(2P) = 0$
  for points $P \in J(k)$, since $\eps(P) = 0$ is equivalent to
  $\delta(\tilde{x}) \neq 0$
  if $x$ are normalized Kummer coordinates for~$P$, with
  reduction~$\tilde{x}$.
  This shows that $\eps(P) = 0$ implies $\mu(P) = 0$ (and conversely), so
  $\{P \in J(k) : \eps(P) = 0\} = \{P \in J(k) : \mu(P) = 0\} = U$.
  The remaining statements now are immediate from Proposition~\ref{C:finite quotient},
  taking into account that for $P$ in the kernel of reduction, we trivially have $\eps(P) = 0$.
\end{proof}

An algorithm for the computation of $\mu(P)$ which is based on Theorem~\ref{T:mu_U}
(for $H=0$) is given in~\cite{StollH2}*{\S6}.
Using the relation in Lemma~\ref{L:mu eps rel}, we obtain the following alternative
procedure for computing~$\mu(P)$.
\begin{enumerate}[1.]
  \item Let $x$ be normalized Kummer coordinates for~$P$. \\
        Set $y_0 = (0,0,0,1)$ and $y_1 = x$.
  \item For $n = 1, 2, \ldots$, do the following.
        \begin{enumerate}[a.]
          \item Using pseudo-addition (see~\cite{FlynnSmart}*{\S4}),
                compute normalized Kummer coordinates $y_{n+1}$
                for~$nP$ from $x$, $y_{n-1}$ and~$y_n$; record $\eps(P, nP)$, which is the
                shift in valuation occurring when normalizing~$y_{n+1}$.
          \item If $\eps(P, nP) = 0$, check whether $v(\delta(y_n)) = 0$
                (by Theorem~\ref{T:mu_U}, this is equivalent to
                $nP \in U$). If yes, let $N=n$ and exit the
                loop.
        \end{enumerate}
  \item Return
        \[ \mu(P) = \frac{1}{2N} \sum_{n=1}^{N-1} \eps(P, nP) \,. \]
\end{enumerate}
To see that this works, note that by Lemma~\ref{L:mu eps rel} we have
\[ \mu\bigl((n+1)P\bigr) - 2\mu(nP) + \mu\bigl((n-1)P\bigr) = 2\mu(P) - \eps(P, nP) \,. \]
The sequence $\bigl(\mu(nP)\bigr)_{n \in \Z}$ is periodic with period~$N$, where
$N$ is the smallest positive integer~$n$ such that $nP \in U$ (which exists according
to Theorem~\ref{T:mu_U}). Taking the sum over one period gives
\[ 2 N \mu(P) = \sum_{n=0}^{N-1} \eps(P, nP) = \sum_{n=1}^{N-1} \eps(P, nP) \,. \]

From the periodicity we can also deduce the possible denominators of $\mu(P)$.
As $\eps$ has integral values, we see that $\mu(P) \in \frac{1}{2N}\Z$ if $N$ is a
period of $\bigl(\mu(nP)\bigr)_{n \in \Z}$.
In fact, we can show a little bit more.

\begin{cor} \label{C:Denominator}
  Let $P \in J(k)$ and $N = \min\{n \in \Z_{>0} : \mu(nP) = 0\}$. Then
  \begin{align*}
      \mu(P) \in \frac{1}{N} \Z &\quad\text{if $N$ is odd, and} \\
      \mu(P) \in \frac{1}{2N} \Z &\quad\text{if $N$ is even.}
  \end{align*}
\end{cor}

\begin{proof}
  The sequence $\bigl(\eps(P,nP)\bigr)_{n \in \Z}$ has period~$N$ and is symmetric.
  So if $N$ is odd, we actually have
  \[ \mu(P) = \frac{1}{2N} \sum_{n=1}^{N-1} \eps(P, nP)
            = \frac{1}{N} \sum_{n=1}^{(N-1)/2} \eps(P, nP) \in \frac{1}{N}\Z \,. \qedhere
  \]
\end{proof}
Analyzing the possible denominators of $\mu(P)$ will play a key role in
Section~\ref{algo1}, where we discuss another algorithm for the computation of
$\mu(P)$.

%==================================================================================

\section{Canonical local heights on Kummer coordinates}
\label{S:lambda_Kummer}

We now define a notion of canonical local height for Kummer coordinates.
We keep the notation of the previous section.

\begin{defn}\label{lhg2gen}
  Let $x\in \KS_\A$ be a set of Kummer coordinates on $\KS$.
  The {\em canonical local height of $x$} is given by
  \[ \hat{\lambda}(x) = -v(x) - \mu(x)\,. \]
\end{defn}

\begin{rk}\label{}
  We can also define the canonical local height on an archimedean local field in an
  analogous way.
  Then, if $K$ is a global field and $x$ is a set of Kummer coordinates for a point
  $J(K)$, we have
  \[
    \hat{h}(P) = \sum_{v \in M_K}\frac{1}{c_v}\hat{\lambda}_v(x)\, ,
  \]
  where $c_v$ is the constant introduced in Remark~\ref{R:cv} for a non-archimedean place
  $v$ and $c_v = [K_v:\R]^{-1}$ if $v$ is archimedean.
\end{rk}

The canonical local height $\hat{\lambda}$ on Kummer coordinates has
somewhat nicer properties
than the canonical local height defined (for instance
in~\cite{FlynnSmart} or, more generally, in~\cite{HindrySilverman}*{\S B.9})
with respect to a divisor on $J$.

\begin{prop}\label{lhg2props}
  Let $x,y,z,w\in \KS_\A$. Then the following hold:
  \begin{enumerate}[\upshape (i)]
    \item $\hat{\lambda}(\delta(x)) = 4\hat{\lambda}(x)$.
    \item If $w \ast z = B(x,y)$, then
          $\hat{\lambda}(z) + \hat{\lambda}(w) = 2 \hat{\lambda}(x) + 2 \hat{\lambda}(y)$.
    \item $\hat{\lambda}(x) = -\lim_{n\to\infty}4^{-n} v\bigl(\delta^{\circ n}(x)\bigr)$.
    \item If $k'/k$ is a finite extension of ramification index~$e$ and $\hat{\lambda'}$
          is the canonical local height over~$k'$, then we have
          $\hat{\lambda'}(x) = e \cdot \hat{\lambda}(x)$.
  \end{enumerate}
\end{prop}

\begin{proof} \strut
  \begin{enumerate}[\upshape (i)]
    \item This follows easily from the two relations
          \[ v\bigl(\delta(x)\bigr) = 4 v(x) + \eps(x) \quad\text{and}\quad
              \mu\bigl(\delta(x)\bigr) = 4 \mu(x) - \eps(x) \,.
          \]
    \item This is similar, using Lemma~\ref{L:mu eps rel}
          and $\eps(x,y) = v(w) + v(z) - 2 v(x) - 2 v(y)$.
    \item This follows from~(i) and the fact that $\mu(x)$ is a bounded function, implying
          \[\hat{\lambda}(x) = 4^{-n} \hat{\lambda}\bigl(\delta^{\circ n}(x)\bigr)
                              = -4^{-n} v\bigl(\delta^{\circ n}(x)\bigr) + O(4^{-n})\,.
          \]
    \item This is obvious from the definition of $\hat{\lambda}$. \qedhere
  \end{enumerate}
\end{proof}

The canonical local height on Kummer coordinates also behaves well under isogenies.

\begin{prop}\label{lhg2isog}
  Let $C$ and~$C'$ be two curves of genus~$2$ over~$k$ given by Weierstrass
  equations, with associated Jacobians
  $J$ and~$J'$, Kummer Surfaces $\KS$ and~$\KS'$ and sets of sets of Kummer coordinates
  $\KS_\A$ and $\KS'_\A$, respectively.
  Let $\alpha \colon J \to J'$ be an isogeny defined over~$k$.
  Then $\alpha$ induces a map $\alpha \colon \KS \to \KS'$; let $d$ denote its degree.
  We also get a well-defined induced map
  $\alpha \colon \KS_\A \to \KS'_{\A}$ if we
  fix $a \in k^\times$ and require $\alpha(0,0,0,1)=(0,0,0,a)$. Then we have
  \[ \hat{\lambda}\bigl(\alpha(x)\bigr) = d \hat\lambda(x) - v(a) \]
  for all $x\in \KS_\A$.
\end{prop}

\begin{proof}
  All assertions except for the last one are obvious.
  By the definition of $\hat{\lambda}$, we can reduce to the case $a = 1$.
  Using part~(iii) of Proposition~\ref{lhg2props} it is then enough to show
  that
  \[ v\bigl(\delta^{\circ n}(\alpha(x))\bigr) = d v(\delta^{\circ n}(x)) + O(1)\,. \]
  However, we have $v(\alpha(x)) - dv(x) = O(1)$ by assumption, so it suffices to show
  that
  \begin{equation}\label{checkisog}
    v\bigl(\delta^{\circ n}(\alpha(x))\bigr) = v\bigl(\alpha(\delta^{\circ n}(x))\bigr)\,.
  \end{equation}
  But since $\alpha \colon J \to J'$ is an isogeny,
  $\delta^{\circ n}(\alpha(x))$ and $\alpha(\delta^{\circ n}(x))$ represent the same point
  on~$\KS'$,
  hence they are projectively equal. Because they also have the same degree, the factor of
  proportionality is independent of~$x$. It therefore suffices to check~\eqref{checkisog} for a
  single $x$; we take $x=(0,0,0,1) \in \KS_\A$. Because we have $\delta(x) = x$ and, by
  assumption, $\alpha(x) = x'$, where $x'=(0,0,0,1) \in \KS'_\A(k)$, we find
  \[ \delta^{\circ n}(\alpha(x)) = x' \quad\text{and}\quad
     \alpha(\delta^{\circ n}(x)) = x'\,,
  \]
  thereby proving~\eqref{checkisog} and hence the proposition.
\end{proof}

\begin{rk}
    Canonical local heights with similar functorial properties were constructed by
    Zarhin~\cite{Zarhin} on total spaces of line bundles (without the zero section).
    See also~\cite{BombieriGubler} for an approach to canonical local heights using
    rigidified metrized line bundles.
\end{rk}

The preceding proposition is particularly useful for analyzing the
behavior of the canonical local height under a change of
Weierstrass equation of the curve.

Recall that two Weierstrass equations for $C$ are related by a transformation $\tau$ as
in~\eqref{tau}, specified by a triple $(A,e,U)$,
where $A = \left(\begin{smallmatrix} a & b \\ c & d \end{smallmatrix}\right)\in \GL_2(k)$,
$e \in k^\times$ and
\[ U = u_0 Z^3 + u_1 X Z^2 + u_2 X^2 Z + u_3 X^3 \in k[X,Z] \]
is homogeneous of degree~3.
Such a transformation induces a map on $\KS_\A$ as follows:
Let $x = (x_1,x_2,x_3,x_4) \in \KS_\A$.
Then $\tau(x)$ is given by the following quadruple:
\begin{align*}
  (ad-bc)^{-1}\Big(& d^2 x_1 + c d x_2 + c^2 x_3, \\
                    & 2 b d x_1 + (ad + bc) x_2 + 2 a cx_3, \\
                    & b^2 x_1 + a b x_2 + a^2 x_3, \\
                    & (ad - bc)^{-2} (e^2 x_4 + l_1 x_1 + l_2 x_2 + l_3 x_3)\Big)\,,
\end{align*}
where $l_1$, $l_2$, $l_3$ do not depend on~$x$.
More precisely, we can write
\[
    l_i = l_{i,1} + l_{i,2} + l_{i,3}\,,
\]
where \begin{align*}
 l_{i,1}&=\frac{e^2}{(ad-bc)^4}l'_{i,1}\quad\text{  with  }l'_{i,1}\in\Z[f_0,\ldots,f_6,a,b,c,d],\\
 l_{i,2}&=\frac{e}{(ad-bc)^4}l'_{i,2}\quad\text{  with  }l'_{i,2}\in\Z[h_0,\ldots,h_3,u_0,\ldots,u_3,a,b,c,d],\\
 l_{i,3}&=\frac{1}{(ad-bc)^4}l'_{i,3}\quad\text{  with  }l'_{i,3}\in\Z[u_0,\ldots,u_3,a,b,c,d]\\
\end{align*}for $i=1,2,3$.
All of the $l'_{i,j}$ are homogeneous of degree $8$ in $a,b,c,d$ and homogeneous in the other variables.

So we see that $\tau$ acts on~$k^4$ as a linear map~$\tau'$ whose determinant has valuation
\[
    v(\tau) \colonequals v(\det(\tau')) = 2 v(e) - 3 v(ad - bc)\,.
\]
In this situation, Proposition~\ref{lhg2isog} implies:
\begin{cor}\label{C:lambdatau}
  Let $\tau=([a,b,c,d],e,U)$ be a transformation~\eqref{tau} between two
  Weierstrass equations $\calC$ and $\calC'$ of a smooth projective
  curve~$C/k$ of genus~2 and let $\KS$ be the model of the Kummer surface associated
  to~$\calC$. Then we have
  \[ \hat{\lambda}(\tau(x)) = \hat\lambda(x) - v(\tau) \]
  for all $x\in \KS_\A$. In particular,
  \[ \mu(x) = \mu(\tau(x)) + v(\tau(x)) - v(x) - v(\tau) \,. \]
\end{cor}

This can be used to construct a canonical local height which does not depend on
the choice of Weierstrass equation.
\begin{defn}
  Let $C/k$ be a smooth projective curve of genus~2 given by a Weierstrass equation~\eqref{ceq}
  with discriminant $\Delta$ and let $\KS$ be the associated Kummer surface. We call the function
  \[ \tilde{\lambda} \colon \KS_\A \To \R\,, \qquad
    x \longmapsto \hat{\lambda}(x)+\frac{1}{10}v(\Delta)
  \]
  the \emph{normalized canonical local height on $\KS_\A$}.
\end{defn}

\begin{cor}\label{normlhg2}
  The normalized canonical local height is independent of the given Weierstrass equation of~$C$,
  in the following sense: if $W$ and~$W'$ are two Weierstrass equations for~$C$,
  with associated sets of sets of Kummer coordinates $\KS_\A$ and~$\KS'_\A$ and canonical
  local heights $\tilde{\lambda}$ and $\tilde{\lambda}'$, respectively,
  and $\tau$ is a transformation~\eqref{tau} between them, then for all $x \in \KS_\A$
  we have $\tilde{\lambda}'(\tau(x)) = \tilde{\lambda}(x)$.
\end{cor}
\begin{proof}
  Let $\Delta$ and~$\Delta'$ be the respective discriminants of $W$ and~$W'$.
  By \cite{Liu2}*{\S2}, we have
  \begin{equation}\label{E:disc_change}
      v(\Delta') = v(\Delta) + 10v(\tau)\,,
  \end{equation}
  so, using Corollary~\ref{C:lambdatau},
  \[ \tilde{\lambda}'(\tau(x)) = \hat{\lambda}'(\tau(x)) + \frac{1}{10} v(\Delta')
                               = \hat{\lambda}(x) - v(\tau) + \frac{1}{10} v(\Delta')
                               = \hat{\lambda}(x) + \frac{1}{10} v(\Delta)
                               = \tilde{\lambda}(x) \,. \qedhere
  \]
\end{proof}

We will not need the normalized canonical local height in the remainder of this paper.

%============================================================================

\section{Stably minimal Weierstrass models} \label{simpmod}

In this section, $k$ continues to denote a non-archimedean local field with
valuation ring~$\O$ and residue field~$\frk$. We build on results established by Liu~\cite{Liu2}
in the more general context of hyperelliptic curves of arbitrary genus.

Recall that an equation of the form~\eqref{ceq} defining
a curve~$C$ over~$k$ of genus~2 is an \emph{integral Weierstrass model} of~$C$
if the polynomials $F$ and~$H$ have coefficients in~$\O$. (Note that this is
slightly different from the notion of an `integral equation' as defined
in~\cite{Liu2}*{D\'efinition~2}, but the difference is irrelevant for our purposes,
since any minimal Weierstrass model is actually given by an integral equation,
see~\cite{Liu2}*{Remarque~4}.) It is a \emph{minimal
Weierstrass model} of~$C$ if it is integral and the valuation of its discriminant
is minimal among all integral Weierstrass models of~$C$ \cite{Liu2}*{D\'efinition~3}.
We introduce the following variant of this notion.

\begin{defn} \label{D:stablyminimal}
  An integral Weierstrass model of a smooth projective curve~$C$ over~$k$ of genus~2 is
  \emph{stably minimal} if it is a minimal Weierstrass model for $C$ over~$k'$ for
  every finite field extension $k'$ of~$k$.
\end{defn}

Stably minimal Weierstrass models can be characterized in terms of the multiplicities
of the points on the special fiber, where the multiplicity is defined as follows:

\begin{defn} \label{multip}
  Only for this definition let $k$ be an arbitrary field,
  and let $C_{F,H}$ be a curve in $\BP_k(1,3,1)$
  given by an equation of the form~\eqref{ceq} over~$k$; we assume that $C_{F,H}$
  is reduced.
  The {\em multiplicity $m(P,C_{F,H})$} of a geometric point $P \in C_{F,H}(\bar{k})$
  is defined as follows:
  \begin{itemize}
    \item If $P$ is a singular point of type $A_n$ (relative to the embedding of
      $C_{F,H}$ into $\BP_k(1,3,1)$), then $m(P,C_{F,H})=n+1$.
    \item If $P$ is fixed by the involution $\iota(X:Y:Z) = (X:-Y-H(X,Z):Z)$ and is nonsingular,
          then $m(P,C_{F,H})=1$.
    \item Otherwise $m(P,C_{F,H})=0$.
  \end{itemize}
\end{defn}

Singularities of type $A_n$ were defined by Arnold over the complex numbers, and hence for
arbitrary fields of characteristic zero, see for instance~\cite{BPV}*{\S II.8}.
For the case of positive characteristic, see~\cite{GreuelKroening}.
Note that if the characteristic of~$k$ is not~$2$, then
$\pi(P)$ is a root of multiplicity~$m(P,C_{F,H})$ of $F^2 + 4 H$, where
$\pi \colon C_{F,H} \to \BP^1$ sends $(X:Y:Z)$ to~$(X:Z)$.

We will use this notion in the context of points on the special fiber of a
Weierstrass model of a curve of genus~$2$ over a complete local field.
In this context, Definition~\ref{multip} is equivalent
to~\cite{Liu2}*{D\'efinition~9} when the curve is reduced, see~\cite{Liu2}*{Remarque~8}.

An algorithm that computes the multiplicity was given by Liu~\cite{Liu2}*{\S6.1}.
Liu defines further multiplicities $\lambda_r(P)$~\cite{Liu2}*{D\'efinition~10}
for points on the special fiber of an integral Weierstrass model (and $r \ge 1$)
that allow to characterize when such a model is minimal.
We note here that $\lambda_r(P)$ gives the value of $\lambda(P) = \lambda_1(P)$
after making a field extension of ramification index~$r$. Also, Lemme~7(e)
of~\cite{Liu2} states for $r$ sufficiently large that $\lambda_r(P) = m(P)$
if the special fiber is reduced and implies that $\lambda_r(P) \ge r$
if the special fiber is non-reduced. In the reduced case, we also have $\lambda(P) \le m(P)$.

Setting $\lambda = \lambda_1$, Corollaire~2 in~\cite{Liu2} states (for $g = 2$)
that the model is minimal if and only if $\lambda(P) \le 3$ and $\lambda'(P) \le 4$ (and is the
unique minimal Weierstrass model up to $\O$-isomorphism, if and only if in
addition $\lambda'(P) \le 3$) for
all $\frk$-points~$P$ on the special fiber, where $\lambda'(P)$ is a number satisfying
$\lambda'(P) \le 2 \lceil \lambda(P)/2 \rceil$, see~\cite{Liu2}*{Lemme~9(c)}.

\begin{lemma} \label{L:charstabmin}
  An integral Weierstrass model of a smooth projective curve~$C$ over~$k$ of genus~2 is stably minimal
  if and only if its special fiber is reduced and the multiplicity of every geometric
  point on the special fiber is at most~$3$.

  If the special fiber is reduced and all multiplicities are at most~$2$, then the
  model is the unique minimal Weierstrass model of~$C$ over any finite extension~$k'$
  of~$k$, up to isomorphism over the valuation ring of~$k'$.
\end{lemma}

\begin{proof}
  First note that the multiplicity of a point is a geometric property; it does not
  change when we replace $k$ by a finite extension. If the special fiber of an integral
  Weierstrass model has the given properties, then it follows from Liu's results
  mentioned above that $\lambda(P) \le m(P) \le 3$ and therefore $\lambda'(P) \le 4$
  for all points~$P$ on the special fiber, even after replacing~$k$ by a finite extension.
  It follows that the model is stably minimal.

  If $m(P) \le 2$ for all~$P$, then $\lambda(P) \le 2$ and $\lambda'(P) \le 2$,
  so by Liu's results, the model is the unique minimal Weierstrass model of~$C$ over~$k'$.

  Conversely, assume that the special fiber does not have the given properties.
  Then either the special fiber is non-reduced, or else there is a point~$P$ on the
  special fiber of multiplicity $m(P) \ge 4$. If the special fiber is non-reduced,
  then after replacing~$k$ by a sufficiently ramified extension~$k'$, there is a point~$P$
  on the special fiber such that $\lambda(P) > 3$ over~$k'$ (ramification index~$4$ is sufficient).
  If the special fiber is reduced and there is a (geometric) point~$P$ on the special
  fiber with $m(P) > 3$, then again after replacing~$k$ by a sufficiently large finite
  extension~$k'$ (such that $P$ is defined over the residue field and the ramification
  index is at least~$m(P)$), we have $\lambda(P) = m(P) > 3$ over~$k'$.
  Liu's results then show that the model is not minimal over~$k'$.
\end{proof}

\begin{lemma} \label{L:allminstab}
  If $C$ is a smooth projective curve over~$k$ of genus~2, then there is a finite extension~$k'$ of~$k$
  such that
  \begin{enumerate}[\upshape (i)]
    \item the minimal proper regular model of~$C$ over the valuation ring of~$k'$
          has semistable reduction, and
    \item each minimal Weierstrass model of~$C$ over~$k'$ is already stably minimal.
  \end{enumerate}
\end{lemma}

\begin{proof}
  That there is a finite extension with the first property is a special case of the
  semistable reduction theorem~\cite{DeligneMumford1969}. After a further
  unramified extension, we can assume that all geometric components
  of the special fiber
  of the minimal proper regular model (which all have multiplicity~1)
  are defined over the residue field
  and that at least one component has a smooth point defined over the residue field.
  This implies by Hensel's Lemma that $C(k') \neq \emptyset$. It then follows
  from~\cite{Liu2}*{Corollaire~5} that every minimal Weierstrass model of~$C$ over~$k'$
  is dominated by the minimal proper regular model. Since the latter has reduced
  special fiber, the same is true for each minimal Weierstrass model.

  Now assume that there exists a stably minimal Weierstrass model of~$C$ over~$k'$.
  Then every minimal Weierstrass model of~$C$ over~$k'$ must already be
  stably minimal, since both models must have the same valuation of the
  discriminant, and the discriminant of the stably minimal model remains
  minimal over any finite field extension of~$k'$. So it is enough to
  show that a stably minimal model exists.

  We now consider the various possibilities for the special fiber of the minimal
  proper regular model. The possible configurations are shown in Figures
  \ref{picc1}, \ref{picc2}, \ref{picc3} and~\ref{picss}
  (on pages \pageref{picc1}, \pageref{picc2},
  \pageref{picc3} and~\pageref{picss}).
  If the reduction type is $[I_{m_1-m_2-m_3}]$ in the notation of~\cite{NamiUeno},
  then the Weierstrass model whose special fiber contains the component(s) that
  are not ($-2$)-curves has the property that all points on the special fiber
  have multiplicity at most~$2$; this is then the unique minimal Weierstrass
  model, and it is stably minimal by Lemma~\ref{L:charstabmin}.
  It remains to consider reduction type $[I_{m_1} - I_{m_2} - l]$.
  We see that the Weierstrass models that correspond to components in the chain
  linking the two polygons and also those coming from the component of one of the polygons
  that is connected to the chain satisfy the conditions of Lemma~\ref{L:charstabmin}
  and are thus stably minimal.
  On the other hand, Weierstrass models whose special fiber does not
  correspond to a component in the chain or to one of its neighbors have a point in the
  special fiber whose multiplicity is at least~$4$ and so cannot be stably minimal.
\end{proof}

%==============================================================================

\section{Igusa invariants} \label{igusa}

In this section we describe how we can easily distinguish between different types of
reduction using certain invariants of genus~2 curves introduced by
Igusa in~\cite{Igusa}. The results of this section are essentially due
to Liu~\cite{Liustable}; see also~\cite{Mestre}.

Let $k$ be an arbitrary field of characteristic not equal to~2 and consider the invariants
$J_2, J_4, J_6, J_8, J_{10}$ defined in~\cite{Igusa}, commonly called {\em Igusa invariants}.
Then $J_{2i}(F)$ is an invariant of degree $2i$ of binary sextics, and if
\[ F(X,Z) = f_0 Z^6 + f_1 X Z^5 + f_2 X^2 Z^4 + f_3 X^3 Z^3 + f_4 X^4 Z^2 + f_5 X^5 Z + f_6 X^6 \]
is a binary sextic, then $J_{2i}(F) \in \Z[\frac{1}{2},f_0,\ldots,f_6]$.
For example, $J_{10}(F) = 2^{-12} \disc(F)$. It is shown in~\cite{Igusa} that the invariants
$J_2, J_4, J_6, J_{10}$ generate the even degree part of the ring of invariants of binary sextics.

Now let $F$ and~$H$ be the generic binary forms over~$\Z$ of degrees 6 and~3, respectively,
  with coefficients $f_0, \ldots, f_6$ and $h_0, \ldots, h_3$ as before.
  It turns out that
  $J_{2i}\left(4F + H^2\right)$ is an element of~$\Z[f_0, \ldots, f_6, h_0, \ldots, h_3]$.

\begin{defn}
  Let $k$ be an arbitrary field and let $H,\, F \in k[X,Z]$ be binary forms of respective degrees~3
  and~6 over~$k$.
  Let $C_{F,H}$ be the curve given by the equation
  $Y^2 + H(X,Z) Y = F(X,Z)$
  in the weighted projective plane $\BP_k(1,3,1)$.
  For $1 \le i \le 5$ we define the {\em Igusa invariant $J_{2i}(C_{F,H})$ of
  $C_{F,H}$} as
  \[ J_{2i}(C_{F,H}) = J_{2i}\left(4F + H^2\right)\,. \]
  Following Liu~\cite{Liustable}, we also define two additional invariants, namely
  \[ I_4(C_{F,H}) = J_2(C_{F,H})^2 - 24 J_4(C_{F,H}) \]
  and
  \begin{align*}
    I_{12}(C_{F,H}) &= -8 J_4(C_{F,H})^3 + 9 J_2(C_{F,H}) J_4(C_{F,H}) J_6(C_{F,H})\\
                    & \qquad {} - 27 J_6(C_{F,H})^2 - J_2(C_{F,H})^2 J_8(C_{F,H})\,.
  \end{align*}
\end{defn}

The following is a consequence of~\cite{Liustable}*{Thm.~1}.
\begin{prop}\label{P:IgusaInv}
Let $k$ be a field and let $C_{F,H}/k$ be the curve given by the equation
  \[
    Y^2+H(X,Z)Y=F(X,Z)
  \]
in $\BP_k(1,3,1)$, where $H,\,F\in k[X,Z]$ are binary forms
of degree~3 and~6, respectively. For $1 \le i \le 5$ and
$j\in\{4,12\}$ we set $J_{2i}=J_{2i}(C_{F,H})$ and $I_j=I_j(C_{F,H})$.
  \begin{enumerate}[\upshape(i)]
  \item $C_{F,H}$ is smooth $\iff J_{10}\ne 0$.
  \item $C_{F,H}$ has a unique node and no point of higher multiplicity \\
        $\iff J_{10}=0$ and $I_{12}\ne 0$.
  \item $C_{F,H}$ has exactly two nodes \\
        $\iff J_{10}=I_{12}=0$, $I_4\ne0$, and $J_4 \neq 0$ or $J_6 \neq 0$.
  \item $C_{F,H}$ has three nodes $\iff J_{10}=I_{12}=J_4 = J_6=0$ and $I_4\ne 0$.
  \item $C_{F,H}$ has a cusp $\iff J_{10}=I_{12}=I_4=0$ and $ J_{2i}\ne0$ for some $i\le 4$.
  \item $C_{F,H}$ is non-reduced or has a point of multiplicity at least~4 $\iff J_{2i}=0$ for all $i$.
  \end{enumerate}
\end{prop}

When $C$ is a curve of genus~2 over a non-archimedean local field, then
Igusa invariants can also be used to obtain information on the reduction type of $C$,
see~\cite{Liustable}*{Thm.~1, Prop.~2}.

\begin{prop}\label{P:IgusaSemistable}
Let $k$ be a non-archimedean local field with normalized additive valuation
$v \colon k^\times \surj \Z$ and valuation ring~$\O$,
and let $C/k$ be a smooth projective genus~2 curve, given by a minimal Weierstrass
model with reduced special fiber.
Suppose that the minimal proper regular model $\calC^{\min}$ of $C$ over $\Spec \O$ is semistable
and has reduction type $\calK$ in the notation of~\cite{NamiUeno}.
We set $J_{2i}=J_{2i}(C)$ for $i \in \{1,\ldots,5\}$ and $I_4=I_4(C)$,
$I_{12} = I_{12}(C)$.
\begin{enumerate}[\upshape (i)]
  \item If $\calK = [I_{m-0-0}]$,
        where $m>0$, then $m = v(J_{10})$.
  \item If $\calK = [I_{m_1-m_2-0}]$, where $0<m_1 \le m_2$, then
        \begin{align*}
          m_1 = \min\left\{v(I_{12}), \tfrac{1}{2}v(J_{10})\right\} \quad\text{and}\quad
          m_2 = v(J_{10}) - m_1\,.
        \end{align*}
  \item If $\calK =  [I_{m_1-m_2-m_3}]$, where $0<m_1 \le m_2\le m_3$, then
        \begin{align*}
          m_1 &= \min\left\{v(J_4), \tfrac{1}{3} v(J_{10}),\tfrac{1}{2} v(I_{12}) \right\},\\
          m_2 &= \min\left\{v(I_{12})-m_1, \tfrac{1}{2}(v(J_{10})-m_1) \right\} \quad\text{and} \\
          m_3 &= v(J_{10}) - m_1 - m_2\,.
        \end{align*}
  \item If $\calK = [I_0-I_0-l]$, then $l = \tfrac{1}{12} v(J_{10})$.
  \item If $\calK = [I_{m_1}-I_0-l]$, where $m_1 >0$, then
        \begin{align*}
          l = \tfrac{1}{12} v(I_{12}) \quad\text{and}\quad
          m_1 = v(J_{10}) - v(I_{12})\,.
        \end{align*}
  \item If $\calK = [I_{m_1}-I_{m_2}-l]$, where $m_2\ge m_1>0$ and $l>0$, then
        \begin{align*}
          l &= \tfrac{1}{4}v(I_{4})\,, \\
          m_1 &= \min\left\{v(I_{12})-3v(I_4)\,,
                            \tfrac{1}{2}(v(J_{10}) - 3 v(I_{4}))\right\} \quad\text{and} \\
          m_2 &= v(J_{10}) - 3 v(I_4) - m_1\,.
        \end{align*}
\end{enumerate}
\end{prop}

%====================================================================================

\vfill\pagebreak

\section*{\large Part~II: Study of Local Height Correction Functions}

In Part~II of the paper, $k$ will always denote a non-archimedean local field with
residue field~$\frk$, valuation ring~$\O$ and normalized additive valuation
$v \colon k^\times \surj \Z$.
We let $C$ be a curve of genus~2 over $k$, given by an integral Weierstrass model
$\calC$, which we consider as a subscheme of the weighted projective plane
$\BP_S(1,3,1)$, where $S = \Spec(\O)$.
In the following five sections we find explicit formulas and bounds for the local
height correction function $\mu$ for the most frequent cases of bad reduction and use
these to deduce a general bound on $\mu$.
We denote the minimal proper regular model of $C$ over $S$ by $\calC^{\min}$.
Let $J$ be the Jacobian of~$C$; we denote its N\'eron model over~$S$ by~$\calJ$.
We write $\calC_v$, $\calC^{\min}_v$ and $\calJ_v$ for the respective special fibers of
$\calC$, $\calC^{\min}$ and $\calJ$.

%==============================================================================

\section{The `kernel' of $\mu$} \label{kermu}

By Theorem~\ref{T:mu_U}, the set
\[ U=\{P \in J(k) : \eps(P) = 0\} \]
is a group and the local height correction function~$\mu$ factors
through the quotient~$J(k)/U$.
In this section we relate $U$ to the N\'eron model of $J$ when $\calC$ has rational singularities.
See~\cite{Artin} for a brief account of the theory of rational singularities on arithmetic surfaces.

For the remainder of this section we assume that $\calC/S$ is normal and reduced.
We let $\calJ^0$ denote the fiberwise-connected component of the identity of~$\calJ$.
Then $\calJ^0$ has generic fiber $\calJ_{k} \isom J$ and special fiber the connected
component of the identity $\calJ^0_v$ of~$\calJ_v$.
If $\calC'\to\calC$ is a desingularization of $\calC$, then the identity
components $\Pic^0_{\calC'/S}$ and $\Pic^0_{\calC/S}$ of the respective relative Picard functors of
$\calC'$ and $\calC$ can both be represented by separated
schemes, see~\cite{BLR}*{Thm.~9.7.1}.
There are canonical $S$-group scheme morphisms
\begin{equation}\label{PiC to J}
        \xymatrix{
        \Pic^0_{\calC/S}\ar[r]&\Pic^0_{\calC'/S}\ar[r]^{\sim}&\calJ^0\,;}
\end{equation}
the latter map is an isomorphism by~\cite{BLR}*{Thm.~9.4.2}.
Let $\alpha \colon \Pic^0_{\calC/S}\to\calJ^0$ denote the composition of the
morphisms from~\eqref{PiC to J}; note that $\alpha$ does not depend on the choice of the
desingularization~$\calC'$.
We will show that if $P \in J(k)$ has reduction on~$\calJ$
in the image of~$\alpha$, then $\eps(P) = \mu(P) = 0$.
The idea is to first show that this is true for points in the image of a certain open subscheme;
we then prove that this suffices for the general case.

Let $\calC_\sm$ be the smooth locus of $\calC$.
Following~\cite{BLR}*{\S9.3}, we define an $S$-subscheme~$W$
of the symmetric square~$\calC_\sm^{(2)}$ of~$\calC_\sm$
consisting of the points $w \in \calC_\sm^{(2)}$
that satisfy the following conditions:
\begin{itemize}
  \item $H^1(\calC,\,\O_{\calC}(D_w))=0$, where $D$ is the universal Cartier divisor
        $D \subset \calC \times_S\, \calC^{(2)}_\sm$
        induced by the canonical map $\calC^{(2)}_\sm \to \Div^2_{\calC/S}$.
  \item If $w =\{w_1,w_2\}$ with $w_1,w_2$ geometric points on the special fiber of~$\calC$,
        then the hyperelliptic involution~$\iota$ maps the component containing~$w_1$ to the
        component containing~$w_2$.
\end{itemize}

Then $W$ has the following properties:
\begin{enumerate}[\upshape (i)]
        \item $W$ is an open subscheme of $\calC_\sm^{(2)}$.
        \item There is a strict $S$-birational group law on $W$,
              induced by the group law on $\Pic_{\calC/S}$.
        \item $\Pic^0_{\calC/S}$ is the $S$-group scheme associated with this
              strict $S$-birational group law.
\end{enumerate}
For (ii) and (iii) see the discussion preceding~\cite{BLR}*{Thm.~9.3.7}.

Let $\Pic^{[2]}_{\calC/S}$ be the open subfunctor of $\Pic_{\calC/S}$ whose elements have total degree~2.
Let  $\rho \colon W \to \Pic^{[2]}_{\calC/S}$ be the canonical map induced by~$D$;
by~\cite{BLR}*{Lemma~9.3.5} it is an open immersion.
Replacing $S$ by the spectrum of the valuation ring of a finite unramified
extension of~$k$, if necessary,
we can find a section $x_0 \in \BP^1_{S}(S)$ such that its pullback~$D_0$
under the covering map
$\calC \to \BP^1_S$ is horizontal and does not intersect the singular locus of~$\calC$.
We denote by~$c_0 $  the class of~$D_0$ in $\Pic^{[2]}_{\calC/S}$.
Let $w = \{P_1, P_2\} \in W$;
using the condition on the action of~$\iota$ on the components $P_1$ and~$P_2$
lie on, we find that
\[
\rho_0(w) \colonequals \rho(w) - c_0 \in \Pic^0_{\calC/S}\,.
\]
In fact, $\rho_0$
defines an open immersion $\rho_0 \colon W \to \Pic^0_{\calC/S}$, see~\cite{BLR}*{Lemma~9.3.6}.

\begin{lemma}\label{L:mu_vanishes1}
  Suppose that the residue characteristic of~$k$ is not~2.
  Let $P \in J(k)$ such that the reduction of~$P$ on~$\calJ_v$ is in~$\alpha(\rho_0(W))$.
  Then $\eps(P) = 0$.
\end{lemma}
\begin{proof}
  We may assume  that $\calC \colon Y^2=F(X,Z)$.
  Let $J_F$ denote the model of~$J$ in~$\BP^{15}$ constructed in
  \cite{CasselsFlynn}*{Chapter~2}
  and let $\calJ_F/S$ denote the model it defines over~$S$.
  Following~\cite{MWsieve}*{\S5}, we denote by~$\calJ_F^0$ the
  fiberwise-connected component of the identity of the smooth locus of~$\calJ_F$, so
  that the generic fiber is~$\calJ_F$ and
  the special fiber~$\calJ^0_{F,v}$ is the connected component of the identity of the smooth locus of
  the special fiber~$\calJ_{F,v}$.
  We have a morphism $\psi \colon \calC_\sm^{(2)}\to\calJ^0_F$, defined using the
  expressions for the coordinates on $J_F$ in~\cite{CasselsFlynn}*{Chapter~2},
  see the proof of~\cite{MWsieve}*{Lemma~5.7}.
  We also denote the restriction of this morphism to~$W$ by~$\psi$.

  The N\'eron mapping property yields a natural map $\varphi \colon \calJ^0_F\to\calJ$.
  In general, its image can be a proper subset of $\calJ^0$.
  Nevertheless, the following diagram of $S$-scheme morphisms is commutative
  by~\cite{LiuBook}*{Prop.~3.3.11}, since $W$ is reduced, $\calJ^0$ is separated and
  the diagram is commutative when restricted to generic fibers:
  \begin{equation}\label{diag1}
  \xymatrix{
  W\ar[d]_{\rho_0}\ar[r]^{\psi}&\calJ^0_F\ar[d]^{\varphi}\\
  \Pic^0_{\calC/S}\ar[r]^{\alpha}&\calJ^0
  }
  \end{equation}
  It follows from~\cite{MWsieve}*{Prop.~5.10} that a point $P \in J(k)$ satisfies $\eps(P)=0$ if and only if $P$
  reduces to~$\calJ^0_{F,v}(\frk)$.
  So if $P$ has reduction in~$\alpha(\rho_0(W))$, then the commutativity of the diagram~\eqref{diag1} shows that
  $\eps(P)=0$.
\end{proof}

If the residue characteristic is~2, then no explicit analogue of the group scheme $\calJ_{F}$ is known.
Instead, we have to work with explicit expressions to prove a result analogous
to Lemma~\ref{L:mu_vanishes1}.

Let $\tilde{F}$ and~$\tilde{H}$ be the reductions of $F$ and~$H$, respectively.
In analogy with~\cite{MWsieve}*{Definition~5.1}, we define the subscheme
$\tilde{\calD}$ of $\A^3_\frk\times\A^4_\frk\times\A^5_\frk$
consisting of all triples
\[ (A, B, C) = \bigl((a_0, a_1, a_2),\,(b_0,b_1,b_2,b_3),\,(c_0,c_1,c_2,c_3,c_4)\bigr)
                 \in \A^3_\frk\times\A^4_\frk\times\A^5_\frk
\]
such that
\[
  AC = \tilde{F} - B^2 - B\tilde{H}\,,
\]
where
\begin{eqnarray*}
    A &=& a_0Z^2 + a_1 XZ + a_2 X^2,\\
    B &=& b_0Z^3+ b_1 XZ^2 + b_2 X^2Z + b_3 X^3,\\
    C &=& c_0Z^4 + c_1 XZ^2 + c_2 X^2Z^2 + c_3X^3Z + c_4X^4.
\end{eqnarray*}
Moreover, we set $\calD \colonequals (\pi_2 \times \id)\bigl(\pr_{12}(\tilde{\calD})\bigr)$,
where $\pr_{12}$ is the projection onto the first two factors
and $\pi_2$ is the canonical map $\A^3_\frk\setminus\{(0,0,0) \} \to \BP^2_\frk$.

Note that if the curve~$\calC_v $ defined by
$Y^2+\tilde{H}(X,Z)Y=\tilde{F}(X,Z)$
in~$\BP_{\frk}(1,3,1)$ is nonsingular, then $\calD(\frk)$ is in bijective
correspondence with the possible Mumford representations of effective
divisors of degree~2 on~$\calC_v$.

In general, this correspondence still holds for the subset~$\calD'$
of all $({A}, {B}) \in \calD$
such that ${A}$ does not vanish at the image in~$\BP^1$ of a singular point
of~$\calC_v$, and those effective divisors with support in the smooth locus
of~$\calC_v$. More precisely, we get a map
$\zeta \colon \calD' \to C^{(2)}_v$
such that if $\zeta(({A}, {B})) = \{\tilde{P_1}, \tilde{P_2}\}$,
then there are representatives $(X_i, \, Y_i, Z_i)$ of $\tilde{P_i}$ ($i=1,\,2$)
satisfying
\begin{enumerate}[\upshape (i)]
    \item $A(X,Z)= (Z_1 X -X_1 Z)(Z_2 X- X_2 Z)$;
    \item $Y_i = B(X_i, Z_i)$ for $i=1,2$.
\end{enumerate}

If $\calC_v$ is nonsingular, and $(A,B) \in \calD$, then we can compose the natural
surjection $\calD \to \mathrm{Jac}(\calC_v)\setminus \{O\}$ with the quotient map
$\mathrm{Jac}(\calC_v) \to \KS_{\tilde{F}, \tilde{H}}$.
In the general case one can also define a surjection
$\omega \colon \calD \to \KS_{\tilde{F},\tilde{H}}\setminus\{(0:0:0:1)\}$
with the following property:
If $P = [(P_1)-(P_2)] \in J(k)$ is such that the reductions $\tilde{P_1}$ and~$\tilde{P_2}$
are both smooth points on~$\calC_v$, and if $(A,B) \in \calD'$ is such that
$\zeta((A,B)) = \{\tilde{P_1},\tilde{\iota(P_2)}\}$, then the reduction of~$\kappa(P)$
on~$\KS_{\tilde{F},\tilde{H}}$ is~$\omega((A,B))$.
The image of a pair $(A,B) \in \calD$ under~$\omega$ is of the form $(a_0:-a_1:a_2:x_4)$.

\begin{lemma}\label{L:mu_vanishes2}
  Suppose that the residue characteristic of~$k$ is~2.
  Let $P \in J(k)$ such that the reduction of~$P$ on~$\calJ$ is
  in~$\alpha(\rho_0(W))$.
  Then $\eps(P) = 0$.
\end{lemma}

\begin{proof}
  \begin{table}
    \begin{tabular}{|c|c|c|c|c|c|} \hline
      type   & condition & additional    & $m(\infty)$ & $m(0)$ & $m(1)$ \\ \hline
      1      & $x_4 = 0$ &               &   &   &   \\
      2      & $x_4 = 0$ &               & 6 &   &   \\
      3      & $x_4 = 0$ & $x_1 = 0$     & 5 &   &   \\
      4      & $x_4 = 0$ & $x_1 = 0$     & 4 &   &   \\
      5      & $x_1 x_3 = x_4 = 0$ &     & 3 & 2 &   \\
      6      & $x_1 = x_4 = 0$ &         & 3 &   &   \\
      7      & $x_4 = 0$ &               & 4 & 2 &   \\
      8      & $x_4 = 0$ &               & 2 & 2 & 2 \\
      9      & $x_1 x_3 = x_4 = 0$ &     & 2 & 2 &   \\
      10     & $x_1 = x_4 = 0$ &         & 2 &   &   \\
      11     & $x_1 = x_4 = 0$ &         & 3 &   &   \\
      12     & $x_4 = 0$ & $x_1 = 0$     & 5 &   &   \\
      13     & $x_4 = 0$ & $x_1 x_3 = 0$ & 3 & 3 &   \\ \hline
    \end{tabular}
    \medskip

    \caption{Conditions for the vanishing of $\delta(x)$}
    \label{condmult}
  \end{table}

  Let $(A,B) \in D'_{\tilde{F}, \tilde{H}}$ such that
  $\zeta((A,B)) = \{\tilde{P_1},\tilde{P_2}\} \in W$.
  By the discussion preceding the lemma, it suffices to
  show that we have $\delta(x) \ne 0$ for $x = \omega((A,B))$.

  Changing the given model, if necessary, we can assume that $\tilde{H}$ and $\tilde{F}$ are
  as in the list of representatives 1--13 in Table~\ref{T:reps}.
  Table~\ref{condmult} contains conditions on $x$ which are equivalent
  to the vanishing of $\delta(x)$ for each representative and
  additional conditions which a point $x=(x_1:x_2:x_3:x_4)\in\BP^3$ satisfying
  $\delta(x) = 0$ must satisfy in order to lie on $\KS_{\tilde{F},\tilde{H}}$.
  Finally, we have listed the
  multiplicities $m(\infty)$, $m(0)$, $m(1)$
  that $\calC_v$ has at the points with $(X:Z)=(1:0)$, $(X:Z)=(0:1)$ and
  $(X:Z)=(1:1)$, respectively, in case the multiplicities there are greater than~1.
  Note that we do not have to treat type~1, as $\calC_v$ is assumed to be
  reduced.

  Since ${A}(X,Z)$ does not vanish at the image in~$\BP^1$ of a singular point, we get
  ${x_1} \ne 0$ and, if $(0,0)$ is a singular point, also ${x_3} \ne 0$.
  Using Table~\ref{condmult}, this already implies that $\delta(x) \ne 0$
  whenever $\calC_v$ is irreducible.
  In the reducible cases 2, 7 and~8, $\calC_v$ has two irreducible components, and
  one checks easily that ${x_4}$ does not vanish
  because, by definition of $W$, $\iota$ maps the component containing $\tilde{P_1}$ to the
  component containing $\tilde{P_2}$.
  Hence $\delta(x) \neq 0$ by Table~\ref{condmult}.
\end{proof}

The next proposition follows from Lemmas~\ref{L:mu_vanishes1} and~\ref{L:mu_vanishes2}.
\begin{prop} \label{P:mu_vanishes}
  Let $\alpha \colon \Pic^0_{\calC/S} \to \calJ^0$ be the canonical homomorphism.
  If the reduction of~$P \in J(k)$ on~$\calJ_{v}$
  is in the image of~$\alpha$, then $\eps(P) = \mu(P) = 0$.
\end{prop}

\begin{proof}
  If $T$ is an $S$-scheme and $x\in \Pic^0_{\calC/S}(T)$, then
  by property (ii) and~(iii) of~$W$, there is an \'etale cover~$T'/T$ and
  $w_1,\ldots,w_n\in W(T')$ such that
  \[
    x = \rho_0(w_1) + \ldots + \rho_0(w_n)\,,
  \]
  where the sum is taken with respect to the group law on~$\Pic^0_{\calC/S}$.
  In fact we can take $n = 2$; this follows from~\cite{BLR}*{Lemma~5.1.4} and the discussion
  following~\cite{BLR}*{Lemma~5.2.4}.
  Using this and Theorem~\ref{T:mu_U}, it suffices to show that $\eps(P) = 0$ when the
  reduction of~$P$ on~$\calJ_{v}$ is in~$\alpha(\rho_0(W))$. Hence the result follows from
  Lemmas~\ref{L:mu_vanishes1} and~\ref{L:mu_vanishes2}.
\end{proof}

Let $J_0(k)$ denote the subgroup of~$J(k)$ consisting of points whose
image on the special fiber of $\J$ is in $\J^0(\frk)$.
By~\cite{BoschLiu}*{Lemma~2.1} the group~$\Phi(\frk)$ of $\frk$-rational points
in the component group~$\Phi$ of~$J$ satisfies
\[ \Phi(\frk)\cong J(k)/J_0(k)\,. \]
We can now give a criterion for when $\eps$ and~$\mu$ factor through~$\Phi(\frk)$.

\begin{thm}\label{T:epsfac}
  Let $C$ be a smooth projective curve of genus~2 defined
  over a non-archimedean local field $k$, given by an integral Weierstrass model~$\calC$
  with  rational singularities.
  Then $\eps$ and $\mu$ factor through~$\Phi(\frk)$.
\end{thm}

\begin{proof}
  First note that if $\calC$ has rational singularities, then $\calC$ is
  normal and reduced.
  Moreover, according to~\cite{BLR}*{Thm.~9.7.1}, the homomorphism~$\alpha$
  is an isomorphism if and only if $\calC$ has rational singularities.
  This implies that the image of~$\alpha$, restricted to the generic fiber, is~$J_0(k)$.
  By Proposition~\ref{P:mu_vanishes}, we have $\eps(P) = \mu(P) = 0$ for $P$ in the
  image of~$\alpha$.
  Theorem~\ref{T:mu_U} implies that $\mu$ and $\eps$ factor
  through~$\Phi(\frk)$.
\end{proof}

\begin{rk} \label{R:comp_ell}
  A non-minimal Weierstrass model cannot have rational singularities.
  Moreover, there are minimal (even stably minimal) Weierstrass models
  of curves of genus~$2$ that have non-rational singularities.
  See Example~\ref{CountEx} for a stably minimal Weierstrass model having
  $\mu(P) \ne 0$ for some points $P \in J_0(k)$.

  This behavior cannot occur for elliptic curves; here $\mu$ always factors through $\Phi(\frk)$,
  provided the given Weierstrass model is minimal, see~\cite{SilvermanHeights}. This
  is crucial for the usual algorithms to compute canonical heights on elliptic curves.
  Note that a Weierstrass model of an elliptic curve is minimal if and only if it has
  rational singularities by~\cite{Conrad}*{Corollary~8.4}.
\end{rk}

%==============================================================================

\section{N\'eron functions and reduction graphs} \label{nerong2}

Our next goal is to derive a formula for~$\mu(P)$ in the case when
the minimal proper regular model of $C$ is semistable and
$\mu$ factors through $\Phi(\frk)$.
To this end, we need the notion of N\'eron functions.
The following result is due to N\'eron; see~\cite{LangFund}*{\S11.1}.

\begin{prop}\label{NF}
  Let $A$ be an abelian variety defined over a local field $k$.
  Then we can associate to any divisor $D \in \Div_A(\bar{k})$
  a function $\lambda_D \colon A(\bar{k}) \setminus \supp(D) \to \R$ such that the following
  conditions are satisfied, where we write $\lambda \equiv \lambda' \bmod \const$
  to indicate that the functions $\lambda$ and~$\lambda'$ differ by a constant.
  \begin{enumerate}[\upshape (1)]
  \item If $D,E\in\Div_A(\bar{k})$, then
        $\lambda_{D+E} \equiv \lambda_{D} + \lambda_{E} \bmod \const$
  \item If $D=\div(f)\in\Div_A(\bar{k})$ is principal, then
        $\lambda_{D} \equiv \bar{v} \circ f \bmod \const$,
        where $\bar{v}$ is the extension of~$v$ to~$\bar{k}$.
  \item If $D \in \Div_A(\bar{k})$ and $T_P \colon A \to A$
        is the translation map by a point $P \in A({\bar{k}})$,
        then we have $\lambda_{T_P^*D} \equiv \lambda_{D} \circ T_P \bmod \const$
  \end{enumerate}
  Also, $\lambda_D$ is uniquely determined up to adding a constant.
\end{prop}

We call a function $\lambda_{D}$ as in Proposition~\ref{NF} a {\em N\'eron function associated with $D$}.

We can use local heights on Kummer coordinates to construct N\'eron functions on
the Jacobian $J$ of our genus~2 curve $C$.
If $P_0 \in C(\bar{k})$, then we have an embedding $C_{\bar{k}} \to J_{\bar{k}}$
(defined over $\bar{k}$) that maps %a point
$P \in C(\bar{k})$ to the divisor class $[(P)-(P_0)] \in \Pic^0_C(\bar{k}) = J(\bar{k})$.
Its image is the theta divisor~$\Theta_{P_0}$.
We set $\Theta_{P_0}^\pm = \Theta_{P_0} + \Theta_{\iota(P_0)}$; then $\Theta_{P_0}^\pm$ is
symmetric and in the linear equivalence class of~$2\Theta$ (where $\Theta$ is a theta
divisor coming from taking a Weierstrass point as base-point).
For the following, fix a point $\infty \in C(\bar{k})$ at infinity.
For $i \in \{1,\ldots,4\}$, we set
\[ D_i = \Theta^\pm_{\infty} + \div\left(\frac{\kappa_i}{\kappa_1}\right) \]
and we define a function
$\hat{\lambda}_i \colon J(k)\setminus \supp(D_i) \to \R$
by
\[ \hat{\lambda}_i(P) = \hat{\lambda}\left(\frac{\kappa(P)}{\kappa_i(P)}\right)\,. \]

\begin{lemma}\label{L:Uchida}
  Let $\infty \in C(\bar{k})$ be a point at infinity as above and let $i \in \{1,\ldots,4\}$.
  Then $D_i$ is defined over~$k$ and the function~$\hat{\lambda}_i$ is a N\'eron function
  associated with~$D_i$.
\end{lemma}

\begin{proof}
  If $\infty \notin C(k)$, then we have $\infty \in C(k')$ for some
  quadratic extension $k'$ of~$k$ and the nontrivial element of the Galois group~$\Gal(k'/k)$
  maps $\infty$ to~$\iota(\infty)$, proving the first assertion.
  For a proof of the second assertion, see \cite{UchiCano}*{Thm.~5.3}.
\end{proof}

\begin{defn} \label{D:redgraph}
  Assume that $C$ has semistable reduction over~$k$.
  Let $C' = \calC^{\min}_{v,\bar{\frk}}$ denote the special fiber of
  the minimal proper regular model~$\calC^{\min}$ of~$C$, considered over the algebraic
  closure of the residue field~$\frk$.
  The \emph{reduction graph} $R(C)$ of~$C$
  is a graph with vertex set the set of irreducible components of~$C'$;
  two vertices $\Gamma_1$ and~$\Gamma_2$ are connected by $n$~edges, where $n$ is the
  number of intersection points of $\Gamma_1$ and~$\Gamma_2$ if $\Gamma_1 \neq \Gamma_2$,
  and $n$ is the number of nodes of~$\Gamma_1$ if $\Gamma_1 = \Gamma_2$.
  The Galois group of~$\frk$ acts on~$R(C)$ in a natural way.

  We consider $R(C)$ as a metric graph by giving each edge length~$1$. For two
  vertices $\Gamma_1$ and~$\Gamma_2$, we define $r(\Gamma_1, \Gamma_2)$ as the resistance
  between the vertices, when $R(C)$ is considered as an electric network with unit
  resistance along every edge.
\end{defn}

\begin{rk} \label{R:comp r}
  We can compute $r(\Gamma_1, \Gamma_2)$ as follows. Order the vertices of~$R(C)$
  in some way and let $M$ be the intersection matrix with respect
  to this ordering. Since all components of the special fiber have multiplicity one,
  the kernel of~$M$ is spanned by the `all-ones' vector and the image of~$M$ consists
  of the vectors whose entries sum to zero.
  Let $v$ be the vector with entries zero except that the entry
  corresponding to~$\Gamma_1$ is~$1$ and the entry corresponding to~$\Gamma_2$ is~$-1$.
  Then there is a vector $g$ with rational entries such that $M g = v$, and
  \[ r(\Gamma_1, \Gamma_2) = -g \cdot v \]
  is, up to sign, the standard inner product of the two vectors. (Note that $g$ is not unique,
  but adding a vector in the kernel of~$M$ to it will not change the result.)
  See for instance~\cite{Cinkir}*{Lemma~6.1}.

  Note that the linear map given by~$M$ on the space of functions on the vertices
  can be interpreted as the discrete Laplace operator on the graph~$R(C)$. It is then
  easy to see that $g$, viewed as a function on the vertices, is piecewise linear
  along sequences of edges not containing $\Gamma_1$, $\Gamma_2$ or a vertex of degree
  at least~3. This makes it quite easy to find~$g$ and to compute $r(\Gamma_1, \Gamma_2)$.
\end{rk}

The reduction graph is unchanged when we replace $k$ by an unramified extension.
If we base-change to a ramified extension~$k'$ of~$k$ with ramification index~$e$,
then the new reduction graph is obtained by subdividing the edges of~$R(C)$
into $e$ new edges. We can give these new edges length~$1/e$; then the underlying
metric space remains the same. In particular, $r(\Gamma_1, \Gamma_2)$ does not depend
on~$k'$. This allows us to replace $k$ by a finite extension if necessary.
The scaling of the length corresponds to extending the valuation $v \colon k^\times \surj \Z$
to $\bar{k}^\times \to \Q$ instead of considering the normalized valuation on~$k'$.
All notions defined in terms of the valuation (for example, intersection numbers)
are then scaled accordingly.

\begin{prop}\label{P:muresist}
  We assume that $\calC^{\min}$ is semistable. Let
  $P = [(P_1)-(P_2)] \in J(k)$, with $P_1, P_2 \in C(k)$ mapping to components
  $\Gamma_1$ and~$\Gamma_2$, respectively, of the special fiber of~$\calC^{\min}$.
  We make the following further assumptions.
  \begin{enumerate}[\upshape (i)]
    \item \label{Pmr:1}
          If $Q_1, Q_2 \in C(k)$ map to $\Gamma_1$ and~$\Gamma_2$, respectively,
          then $\mu(P) = \mu([(Q_1)-(Q_2)])$.
    \item \label{Pmr:2}
          There is a constant $\mu_1 \in \Q$ such that $\mu([(Q_1)-(Q'_1)]) = \mu_1$ for all
          $Q_1, Q'_1 \in C(k)$ mapping to~$\Gamma_1$ such that the images
          of $Q_1$ and~$Q'_1$ on the special fiber of~$\calC^{\min}$ are distinct.
    \item \label{Pmr:3}
          There is a constant $\mu_2 \in \Q$ such that $\mu([(Q_2)-(Q'_2)]) = \mu_2$ for all
          $Q_2, Q'_2 \in C(k)$ mapping to~$\Gamma_2$ such that the images
          of $Q_2$ and~$Q'_2$ on the special fiber of~$\calC^{\min}$ are distinct.
  \end{enumerate}
  Then we have
  \[ \mu(P) = r(\Gamma_1, \Gamma_2) + \frac{\mu_1 + \mu_2}{2}\, . \]
\end{prop}

\begin{proof}
  By the discussion preceding the statement of the theorem, we can assume that $k$
  is sufficiently large for $C(k)$ to contain all points we might be interested in.

  Let $P_0 \in C(k)$. The embedding with respect to $P_0$ is obtained from the
  `difference map' $\psi \colon C \times C \to J$
  that sends a pair of points $(P_1, P_2)$ to $[(P_1)-(P_2)]$ by specializing the second
  argument to~$P_0$. One easily checks that
  \[ \psi^* \Theta_{P_0} = \Delta_C + (\{\iota(P_0)\} \times C) + (C \times \{P_0\})\,, \]
  where $\Delta_C$ denotes the diagonal and $\iota$ is the hyperelliptic involution on~$C$.
  We then have
  \[ \psi^* \Theta_{P_0}^\pm = 2 \Delta_C + \pr_1^* D_0 + \pr_2^* D_0\,, \]
  where $D_0 = (P_0) + (\iota(P_0))$. By the results in~\cite{Heinz}
  this implies that, taking $\lambda_0$ to be
  a N\'eron function associated to~$\Theta_{P_0}^\pm$,
  \[ \lambda_0\bigl([(P_1)-(P_2)]\bigr)
    = 2 \langle P_1, P_2 \rangle + \langle P_1 + P_2, P_0 + \iota(P_0) \rangle + c
  \]
  for all points $P_1, P_2 \in C(k^{\nr})$ with $P_1 \neq P_2$ and
  $\{P_1, P_2\} \cap \{P_0, \iota(P_0)\} = \emptyset$,
  where $\langle \cdot, \cdot \rangle$ is the pairing in~\cite{Heinz}*{Thm.~4.4} and $c
  \in \R$ is a constant.

  If $\calC^{\min}$ has semistable reduction,
  then, by \cite{Heinz}*{Remark~4.6}, the pairing $\langle \cdot, \cdot \rangle$
  coincides with Zhang's admissible pairing $(\cdot, \cdot)_a$ defined in~\cite{Zhang}
  in terms of harmonic analysis on the reduction graph~$R(C)$.
  In these terms, we have for $Q, Q' \in C(k^{\nr})$:
  \[ \langle Q, Q' \rangle = (Q, Q')_a
                           = i(\overline{Q}, \overline{Q}') + g_\nu(\Gamma, \Gamma')\,,
  \]
  where $i(\overline{Q}, \overline{Q}')$ is the intersection multiplicity
  of the sections $\overline{Q}, \overline{Q}'\in \calC^{\min}(\O^{\nr})$
  induced by $Q$ and~$Q'$, respectively,
  and $g_\nu(\Gamma, \Gamma')$ is the Green's function associated to a certain measure~$\nu$
  on~$R(C)$, with $\Gamma$ and~$\Gamma'$ being the respective components of the special
  fiber of~$\calC^{\min}$ that $Q$ and~$Q'$ reduce to. See~\cite{Zhang}*{\S4}.
  We extend $g_\nu$ to a bilinear map on the free abelian group generated by the vertices of~$R(C)$.

  Lemma~\ref{L:Uchida} gives, for $P_0 = \infty$ and $P = [(P_1)-(P_2)]$ with normalized
  Kummer coordinates $x(P)=(x_1(P),\ldots,x_4(P))$,
  \begin{align*}
    \mu(P) &= v(x_1(P)) - \hat{\lambda}_1(P) \\
           &= v(x_1(P)) - 2 i(\overline{{P}}_1, \overline{{P}}_2)
                        - i(\overline{{P}}_1 + \overline{{P}}_2, \overline{P}_0 + \overline{\iota(P_0)}) \\
           &\hphantom{{}= v(x_1(P))}{}
                        - 2 g_\nu(\Gamma_1, \Gamma_2) - g_\nu(\Gamma_1+\Gamma_2,
                        \Gamma_0+\Gamma'_0) - c \,,
  \end{align*}
  where $\Gamma_1$ and~$\Gamma_2$ are the respective components that $P_1$ and~$P_2$
  reduce to, and $\Gamma_0$ and~$\Gamma'_0$ are the respective components
  that $P_0$ and $\iota(P_0)$ reduce to.
  We assume for a moment that the images of $P_1$ and~$P_2$ on the special fiber of the
  original model~$\calC$ are distinct from the images of the points at infinity.
  By assumption~\eqref{Pmr:1}, $\mu(P)$ is unchanged
  when we replace the points $P_1$ and~$P_2$ by other points still mapping to $\Gamma_1$
  and~$\Gamma_2$, respectively.
  We can therefore assume that the images of $P_1$ and~$P_2$ on the special
  fiber of~$\calC^{\min}$ are distinct from each other and also from the images
  of $P_0$ and~$\iota(P_0)$. This implies that $v(x_1(P)) = 0$ and that
  the intersection numbers in the formula above are zero. We can choose further points
  $Q_1$ and~$Q_2$ that also reduce to $\Gamma_1$ and~$\Gamma_2$ with reductions on the
  special fiber of~$\calC$ distinct from those of $P_0$ and~$\iota(P_0)$ and such that
  $P_1$, $P_2$, $Q_1$ and~$Q_2$ all reduce to distinct points on the special fiber of~$\calC^{\min}$.
  Using assumptions \eqref{Pmr:2} and~\eqref{Pmr:3}, we obtain the following relations.
  \begin{align*}
    -\tfrac{1}{2} \mu_1
       = -\tfrac{1}{2}\mu\bigl([(P_1)-(Q_1)]\bigr)
      &= g_\nu(\Gamma_1, \Gamma_1) + g_\nu(\Gamma_1, \Gamma_0+\Gamma'_0) + \tfrac{1}{2} c \\
    \mu(P) = \mu\bigl([(P_1)-(P_2)]\bigr)
      &= -2 g_\nu(\Gamma_1, \Gamma_2) - g_\nu(\Gamma_1+\Gamma_2, \Gamma_0+\Gamma'_0) - c \\
    -\tfrac{1}{2} \mu_2
       = -\tfrac{1}{2}\mu\bigl([(P_2)-(Q_2)]\bigr)
      &= g_\nu(\Gamma_2, \Gamma_2) + g_\nu(\Gamma_2, \Gamma_0+\Gamma'_0) + \tfrac{1}{2} c
  \end{align*}
  Adding them together gives
  \[ \mu(P) - \tfrac{1}{2}(\mu_1 + \mu_2)
      = g_\nu(\Gamma_1 - \Gamma_2, \Gamma_1 - \Gamma_2) = r(\Gamma_1, \Gamma_2)\,,
  \]
  as desired. See~\cite{Zhang}*{\S3} for the last equality.

  If our assumption that the images of $P_1$ and~$P_2$ on the special fiber of the
  original model~$\calC$ are distinct from the images of the points at infinity is not satisfied,
  then we choose another point $P_0$ for which the assumption is satisfied. We can then
  perform a change of coordinates~$\tau$ over~$\O$ that moves $P_0$ to infinity and apply the
  result above. By Corollary~\ref{C:lambdatau} (note that $v(\tau) = 0$ in this case) and the
  fact that $v(\tau(x)) = v(x)$, $\mu(P)$ is unchanged by~$\tau$.
\end{proof}

\begin{rk}
  We see from the proof that for two points $Q, Q'$ both having image on a component~$\Gamma$,
  but with distinct reductions that are also distinct from those of $P_0$ and~$\iota(P_0)$,
  we always have
  \[ \mu([(Q)-(Q')]) = -2g_\nu(\Gamma, \Gamma) - 2g_\nu(\Gamma, \Gamma_0+\Gamma'_0) - c \,. \]
  So the assumption that this value does not depend on the choice of $Q$ and~$Q'$ is not
  really necessary.
\end{rk}

\begin{thm} \label{muresist}
  Let $C$ be a smooth projective curve of genus~2 defined
  over a non-archimedean local field $k$, given by an integral Weierstrass model.
  Let $J$ be the Jacobian of~$C$ and $\calJ$ its N\'eron model over~$S = \Spec \O$.
  Assume that the minimal proper regular model~$\calC^{\min}$ of~$C$ over~$S$
  is semistable and that $\mu$ factors through the component group~$\Phi(\frk)$
  of~$\calJ$.
  Let $P \in J(k)$ be such that its image in~$\Phi(\frk)$ is $[\Gamma_1 - \Gamma_2]$,
  where $\Gamma_1$ and~$\Gamma_2$ are components of the special fiber of~$\calC^{\min}$.
  Then we have
  \[ \mu(P) = r(\Gamma_1, \Gamma_2)\,. \]
\end{thm}

\begin{proof}
  Since $\mu$ factors through~$\Phi(\frk)$, it follows that $\mu([(P_1)-(P_2)])$
  vanishes when $P_1$ and~$P_2$ map to the same component on the special fiber
  of~$\calC^{\min}$ and in general depends only on the components $P_1$ and~$P_2$
  map to. This shows that assumptions \eqref{Pmr:1} to~\eqref{Pmr:3} in
  Proposition~\ref{P:muresist} are satisfied with $\mu_1 = \mu_2 = 0$. The claim follows.
\end{proof}

%==============================================================================

\section{Formulas and bounds for $\mu(P)$ in the nodal reduction case} \label{formulas}

In this section and the next, we will deduce explicit formulas for~$\mu(P)$
when we have a stably minimal Weierstrass model~$\calC$.
Recall that $\calC^{\min}$ denotes the minimal proper regular model of~$\calC$.
In the following, when we speak of components, points, and so on, of the special
fiber of $\calC$ or~$\calC^{\min}$, we always mean \emph{geometric} components, points, and so on.

In this section we shall use Theorem~\ref{muresist} and Remark~\ref{R:comp r} to find
explicit formulas for $\mu(P)$ whenever $C/k$ has nodal reduction, i.e.,
the special fiber $\calC_v$ of $\calC$ is reduced and all multiplicities are at most~$2$.
In this case $\calC$ is semistable and therefore it has rational singularities.
Let $\Delta = \Delta(\calC)$ denote the discriminant of $\calC$; we assume that
there is at least one node, so that $v(\Delta) >0$.

Since there are at most three nodes in the special fiber of $\calC$,
we have to consider three different cases.

First suppose that there is a unique node in the special fiber of $\calC$ and set $m=v(\Delta)$.
In the notation of Namikawa and Ueno~\cite{NamiUeno} this is reduction type $[I_{m-0-0}]$.
If $m=1$, then $\calC$ is regular over $S$.
In general, there is a unique component, which we denote by~$A$, of genus~$1$ in the special fiber of $\mathcal{C}^{\min}$.
As in the case of multiplicative reduction of elliptic curves (see for example~\cite{ATAEC}),
the singular point on the special fiber is replaced by a string of $m-1$ components
of $\mathcal{C}^{\min}$, all of genus~$0$ and multiplicity~$1$.
We choose one of the two components intersecting $A$ and call it $B_1$ and number the other components $B_2,\ldots,B_{m-1}$ consecutively as in Figure~\ref{picc1}.

Using~\cite{BLR}*{Thm.~9.6.1}, it is easy to see that the geometric component group
$\Phi(\bar{\frk})$ of the N\'eron model is generated by $[B_1-A]$ and is
isomorphic to $\Z/m\Z$. We have $[B_j-A]=j\cdot[B_1-A]$ in $\Phi(\bar{\frk})$.

We set $B_0 \colonequals B_m \colonequals A$. Then we have the following result.

\begin{figure}
  \begin{center}
    \includegraphics[width=\textwidth]{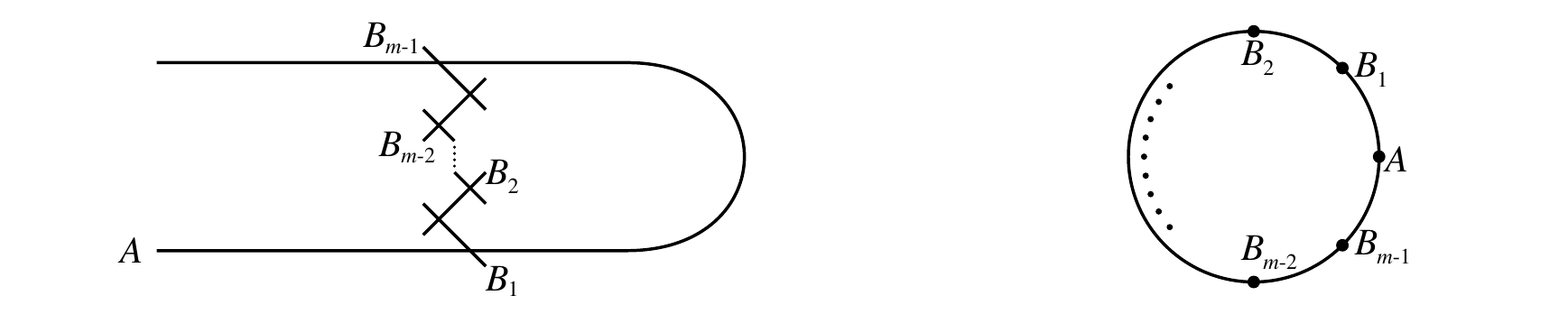}
    \caption{The special fiber of reduction type $[I_{m-0-0}]$ and its reduction graph}\label{picc1}
  \end{center}
\end{figure}

\begin{prop}\label{muchi}
  Suppose that there is a unique node in the special fiber of $\calC$;
  let $m$ and the notation for the components of the special fiber of~$\calC^{\min}$
be as above. If $P \in J(k)$ maps to $[B_i-A]$ in the component group, then we have
\[ \mu(P)=\frac{i (m - i)}{m}\, . \]
\end{prop}

\begin{proof}
  Since the given model is semistable, we can use Theorem~\ref{muresist} and Remark~\ref{R:comp r}.
  One choice of $g$ as in Remark~\ref{R:comp r} is given by
  \[ g(B_j) = \begin{cases}
                -\dfrac{j (m - i)}{m} & \text{if $0 \le j \le i$,} \\[2mm]
                -\dfrac{i (m - j)}{m} & \text{if $i \le j \le m$.}
              \end{cases}
  \]
  Then
  \[ \mu(P) = r(B_i, A)
            = -\bigl(g(B_i) - g(A)\bigr)
            = \frac{i (m - i)}{m}\,,
  \]
  as claimed.
\end{proof}

\begin{rk}
Proposition~\ref{muchi} resembles the formula for the canonical local height on an
elliptic curve with split multiplicative reduction given, for instance, in~\cite{SilvermanHeights}.
\end{rk}

\begin{figure}
  \begin{center}
    \includegraphics[width=\textwidth]{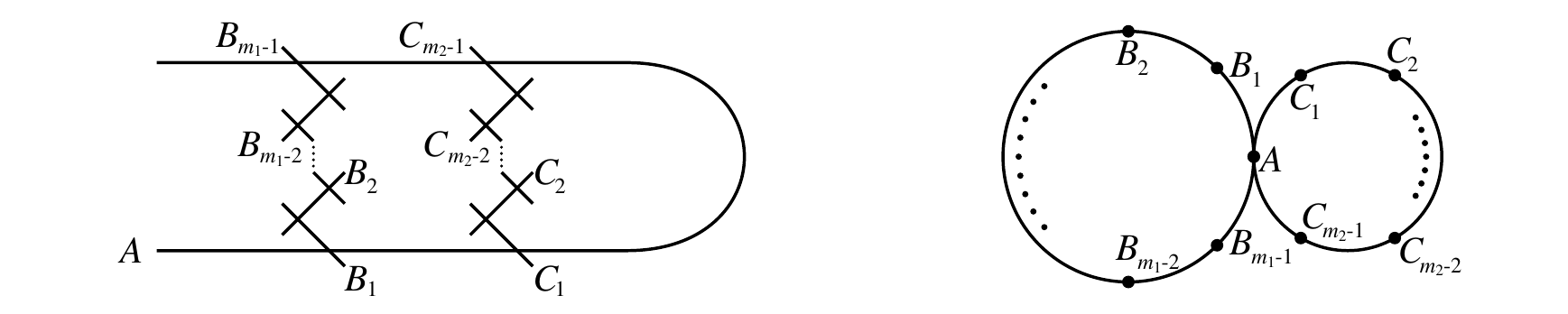}
    \caption{The special fiber of reduction type $[I_{m_1-m_2-0}]$ and its reduction graph}\label{picc2}
  \end{center}
\end{figure}

Now suppose that there are precisely two nodes in the special fiber of~$\calC$.
The reduction type is $[I_{m_1-m_2-0}]$ in the notation of~\cite{NamiUeno}, where $m_1,m_2\ge 1$
and $m_1+m_2=v(\Delta)$.
The special fiber of $\mathcal{C}^{\min}$ is
obtained by blowing up the two singular points of the special fiber of $\calC$
repeatedly and replacing them with a chain of $m_1-1$ and $m_2-1$ curves of
genus~0, respectively.
We call these components $B_1,\ldots,B_{m_1-1},C_1,\ldots,C_{m_2-1}$, numbered
as in Figure~\ref{picc2}, where $A$ contains all images of points reducing to a
nonsingular point and we pick components $B_1$ and $C_1$ intersecting $A$ as in
the case of a unique node.
The component group $\Phi(\bar{\frk})$ is isomorphic to $\Z/m_1\Z\times\Z/m_2\Z$
and is generated by $[B_1-A]$ and $[C_1-A]$; this follows again
using~\cite{BLR}*{Thm.~9.6.1}.
If we have $m_1=1$ or $m_2=1$, then the corresponding singular point on the special
fiber of $\calC$  is regular and is therefore not blown up.

We set $B_0 \colonequals B_{m_1} \colonequals C_0 \colonequals C_{m_2} \colonequals A$.
Then every element of the component
group has a representative of the form $[B_i - C_j]$ with $0 \le i \le m_1$ and
$0 \le j \le m_2$. The following result expresses $\mu(P)$ in terms of this representative.

\begin{prop}\label{muchi2}
  Suppose that there are exactly two nodes in the special fiber of $\calC$;
  let $m_1$ and $m_2$ and the notation for the components of the special fiber of~$\calC^{\min}$
be as above. If $P \in J(k)$ maps to $[B_i-C_j]$ in the component group, then we have
\[ \mu(P) = \frac{i (m_1 - i)}{m_1} + \frac{j (m_2 - j)}{m_2}\,. \]\
\end{prop}

\begin{proof}
This is an easy computation along the same lines as in the proof of Proposition~\ref{muchi}.
\end{proof}

The final case that we have to consider is the case of
three nodes in the special fiber of $\calC$, which then has two components.
We call these components $A$ and $E$.
The special fiber of the minimal proper regular model is obtained using a
sequence of blow-ups of the singular points; they are replaced by a chain of
$m_i-1$ curves of genus~0 and multiplicity~$1$, respectively, where
$v(\Delta)=m_1+m_2+m_3$.
Hence the special fiber of $\mathcal{C}^{\min}$ contains the two components $A$ and $E$, connected by three chains of curves of genus~0 that we call $B_1,\ldots,B_{m_1-1},$ $C_1,\ldots,C_{m_2-1}$ and $D_1,\ldots,D_{m_3-1}$, respectively, where $B_1,C_1$ and $D_1$ intersect $A$, as shown in Figure~\ref{picc3}. The reduction type is $[I_{m_1-m_2-m_3}]$.

\begin{figure}
  \begin{center}
    \includegraphics[width=\textwidth]{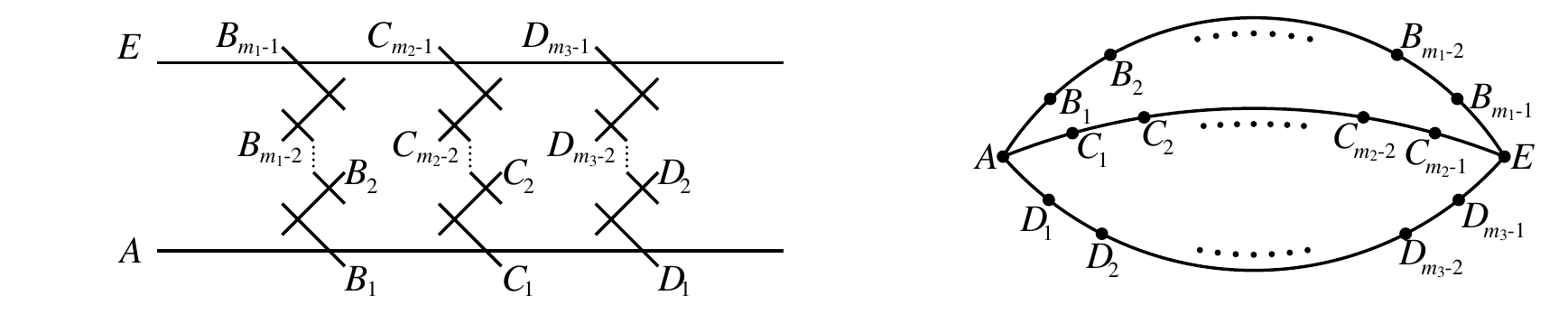}
    \caption{The special fiber of reduction type $[I_{m_1-m_2-m_3}]$ and its reduction graph}
    \label{picc3}
  \end{center}
\end{figure}

By~\cite{BLR}*{Prop.~9.6.10}, the group $\Phi(\bar{\frk})$ is
isomorphic to $\Z/d\Z\times\Z/n\Z$, where
\[ d = \gcd(m_1,m_2,m_3) \quad\text{and}\quad n = \frac{m_1 m_2 + m_1 m_3 + m_2 m_3}{d} \,. \]

We set $B_0 \colonequals C_0 \colonequals D_0 \colonequals A$
and $B_{m_1} \colonequals C_{m_2} \colonequals D_{m_3} \colonequals E$.
Then it is not hard to see that each element of~$\Phi(\bar{\frk})$ can be written in one of the forms
\[ [B_i - C_j], \qquad [C_j - D_l] \qquad \text{or} \qquad [D_l - B_i] \]
with $0 \le i \le m_1$, $0 \le j \le m_2$, $0 \le l \le m_3$.
The following result allows us to express $\mu(P)$ for any $P \in J(k)$ in terms of the component
$P$ maps to.

\begin{prop}\label{muchi3}
  Suppose that there are three nodes in the special fiber of~$\calC$;
  let $m_1$, $m_2$, $m_3$ and the notation for the components of the special fiber of~$\calC^{\min}$
  be as above. If $P$ maps to $[B_i-C_j]$ in the component group
  for some $0\le i\le m_1$ and $0\le j\le m_2$, then we have
  \[ \mu(P) = \frac{m_2 i (m_1 - i) + m_3 (i + j)(m_1 - i + m_2 - j) + m_1 j (m_2 - j)}%
                   {m_1 m_2 + m_1 m_3 + m_2 m_3}\,. \]
  The formulas for $[C_j-D_l]$ and $[D_l-B_i]$ are analogous.
\end{prop}

\begin{proof}
The proof is analogous to those of Propositions~\ref{muchi} and~\ref{muchi2}.
To find $g$, use that it is piecewise linear on the segments $A B_1 \ldots B_i$,
$B_i \ldots B_{m_1-1} E$, $A C_1 \ldots C_j$, $C_j \ldots C_{m_2-1} E$ and
$A D_1 \ldots D_{m_3-1} E$ and the relations at the vertices $A$, $E$, $B_i$ and~$C_j$.
\end{proof}

\begin{rk} \label{eps3}
Using the relation $\eps(P) = 4 \mu(P) - \mu(2P)$, one can show by a somewhat tedious
computation involving a number of different cases that if the image of~$P$ in~$\Phi(\frk)$
is $[\Gamma_1 - \Gamma_2]$, where $\Gamma_1$ and~$\Gamma_2$ are components of the special
fiber of~$\calC^{\min}$, then $\eps(P)$ is the `distance' between $\Gamma_1$ and~$\Gamma_2$
in the reduction graph, where the `length' of the path between $B_i$ and~$B_j$ (say,
analogously for $C_i$, $C_j$ and $D_i$, $D_j$) is $\min\{2|i-j|, m_1\}$ and otherwise,
`lengths' are additive. In particular, if $\Phi(\frk) = \Phi(\bar{\frk})$, then
\[ \gamma = \max\{\eps(P) : P \in J(k)\} = \max\{m_i + m_j - \delta_{ij} : 1 \le i < j \le
3\}\,, \]
where $\delta_{ij} = 0$ if both $m_i$ and~$m_j$ are even, and $\delta_{ij} = 1$ otherwise.
\end{rk}

\begin{rk}\label{FindComp}
In order to use the results of this section to actually compute $\mu(P)$ for a given point
$P\in J(k)$, we need to be able to find the component of $\mathcal{J}_v$ that $P$
reduces to.
One approach is to find $P_1$ and $P_2\in C$ such that $P=[(P_1)-(P_2)]$ and find the
reductions of $P_1$ and $P_2$ to $\mathcal{C}_v^{\min}$.
Another approach is to use a transformation (possibly defined over an unramified extension of
$k$) to move the singular points to $\infty, (0,0)$ and $(1,0)$, respectively.
Then we can (possibly after applying another transformation) read off the component that $P$ maps
to directly from the Kummer coordinates of $P$.
\end{rk}

The discussion of this section shows that we get the following results on the
local height constant $\beta = \max\{\mu(P) : P \in J(k)\}$. Recall that
$\gamma = \max\{\eps(P) : P \in J(k)\}$ and that $\gamma/4 \le \beta \le \gamma/3$.
We will see that in many cases the lower bound is attained.

Let $P$ be a node on $\calC_v$; it is defined over a finite extension of~$\frk$.
We say that the node~$P$ is \emph{split} if the two tangent directions of the branches
at~$P$ are defined over every extension that $P$ is defined over, otherwise $P$ is \emph{non-split}.
We say that $P$ is \emph{even} if its contribution~$m_i$ to the valuation of the discriminant
is even, and \emph{odd} otherwise.

\begin{cor}\label{Case1Bds}
Suppose that $C/k$ is a smooth projective curve of genus~2 given by an integral
Weierstrass model $\calC$ such that there is a unique node
in the special fiber of~$\calC$ and let $m=v(\Delta)$. Then we have
\[ \beta = \frac{1}{2m}\left\lfloor\frac{m^2}{2}\right\rfloor
        \le \frac{v(\Delta)}{4} .
\]
if the node is split or even, and $\beta = 0$ otherwise.
\end{cor}

\begin{proof}
This follows from Proposition~\ref{muchi}, taking into account that
if $m$ is odd and the node is non-split, then the group $\Phi(\frk)$ is trivial.
\end{proof}

\begin{rk} \label{R:gamma4}
Using the relation $\eps(P) = 4 \mu(P) - \mu(2P)$, one can check that
\[ \eps(P) = 2 \min\{i, m-i\} \qquad \text{if $P$ maps to $[B_i - A]$ in~$\Phi(\frk)$.} \]
If $m$ is even (and $\beta > 0$), then $\beta = m/4 = \gamma/4$.
If $m$ is odd, then $\beta = (m - 1/m)/4$ and $\gamma = m-1$, so $\beta/\gamma = (1 + 1/m)/4$
approaches~$1/4$ as $m \to \infty$, but for $m = 3$ (the worst case), we have
$\beta = \gamma/3$.
\end{rk}

\begin{cor}\label{Case2Bds}
Suppose that $C/k$ is a smooth projective curve of genus~2
given by an integral
Weierstrass model $\calC$ such that there are exactly two nodes
in the special fiber of $\calC$.
Let $v(\Delta) = m_1 + m_2$ as above. Then we have
\[ \beta = \frac{1}{2m_1} \left\lfloor\frac{m_1^2}{2}\right\rfloor
              + \frac{1}{2m_2} \left\lfloor\frac{m_2^2}{2}\right\rfloor
        \le \frac{v(\Delta)}{4}
\]
if each of the nodes is split or even,
\[ \beta = \frac{1}{2m_i}\left\lfloor\frac{m_i^2}{2}\right\rfloor \]
if the node corresponding to~$m_i$ is split or even and the other node is non-split and odd,
and $\beta = 0$ if both nodes are non-split and odd.
\end{cor}

\begin{proof}
This follows from Proposition~\ref{muchi2}, taking into account the action of Frobenius
on~$\Phi(\bar{\frk})$.
\end{proof}

If we have three nodes,
then it helps to take the field of definition of the nodes into account.

\begin{cor}\label{Case3Bds}
Suppose that $C/k$ is a smooth projective curve of genus~2
given by an integral Weierstrass model $\calC$ such that there are three nodes
in the special fiber of $\calC$.
We say that $\calC$ is split if the two components $A$ and~$E$
of the special fiber of~$\calC^{\min}$ are defined over~$\frk$, otherwise $\calC$ is non-split.
Let $v(\Delta) = m_1+m_2+m_3$ as above and set $M = m_1 m_2 + m_1 m_3 + m_2 m_3$.
\begin{enumerate}[\upshape (a)]
  \item If all nodes are $\frk$-rational, $\calC$ is split,
      and we have $m_1\ge m_3$ and $m_2\ge m_3$, then
      \[ \beta = \frac{1}{2M} \left(m_2\left\lfloor\frac{m_1^2}{2}\right\rfloor
                                      + m_3\left\lfloor\frac{(m_1+m_2)^2}{2}\right\rfloor
                                      + m_1\left\lfloor\frac{m_2^2}{2}\right\rfloor\right)
              \le \frac{m_1+m_2}{4} < \frac{v(\Delta)}{4} \,.
      \]
    \item If all nodes are $\frk$-rational, but $\calC$ is non-split, then
      \[ \beta = \max\{0\}
                    \cup \Bigl\{\frac{m_i+m_j}{4}
                                  : \text{$1 \le i < j \le 3$, $m_i$ and $m_j$
                                even}\Bigr\}\,.
      \]
\item If two of the nodes lie in a quadratic extension of $\frk$ and are
      conjugate over $\frk$ and one is $\frk$-rational, then
      \[ \beta = \begin{cases}
                    \dfrac{m_1}{M}\max\left\{
                      \left\lfloor\dfrac{m^2_1}{2}\right\rfloor + m_1m_3,
                      \left\lfloor\dfrac{m^2_3}{2}\right\rfloor
                        + m_1\left\lfloor\dfrac{m_3}{2}\right\rfloor\right\},
                        & \text{if $\calC$ is split,} \\[1.5ex]
                    \dfrac{m_1}{2},
                    & \text{if $\calC$ is non-split and $m_1$ is even,} \\[1.5ex]
                    0,  & \text{otherwise.}
                  \end{cases}
      \]
      where $m_3$ corresponds to the rational node (and $m_1 = m_2$).
\item If all nodes are defined over a cubic extension of $\frk$ and are conjugate
  over~$\frk$, then $m_1 = m_2 = m_3 = v(\Delta)/3$ and
      \[ \beta = \begin{cases}
          \dfrac{v(\Delta)}{9}, & \text{if $\calC$ is split,} \\[1.5ex]
                  0,            & \text{otherwise.}
                 \end{cases}
      \]
\end{enumerate}
\end{cor}

\begin{proof}
The proof of (a) follows easily from Proposition~\ref{muchi3}.

For the other cases, note that in the non-split case, some power of Frobenius acts as negation
on the component group~$\Phi(\bar{\frk})$, so the only elements of~$\Phi(\frk)$ are
elements of order~2 in~$\Phi(\bar{\frk})$, which correspond to $[B_{m_1/2} - C_{m_2/2}]$
if $m_1$ and~$m_2$ are even (where $\mu$ takes the value $(m_1+m_2)/4$), and similarly with
the obvious cyclic permutations.

In the situation of~(c), we must have $m_1 = m_2$. If $P=[(P_1)-(P_2)] \in J(k)$
and $P_1 \in C(\bar{k})$ maps to one
of the conjugate nodes, then $P_2$ must map to the other, so all $P\in J(k)$ must map to
a component of the form $[B_i-C_j]$ or $[D_i-D_j]$.
Now the result in the split case follows from a case distinction depending on whether $m_1 \le m_3$
or not. In the non-split case, the only element of order~2 that is defined over~$\frk$
is $[B_{m_1/2} - C_{m_1/2}]$ if it exists.

In the situation of~(d), the group $\Phi(\frk)$ is of order~3 (generated by $[E - A]$) in
the split case and trivial in the non-split case.
\end{proof}

Extending the valuation $v \colon k^\times \surj \Z$ to $\bar{v} \colon \bar{k}^\times \to \Q$,
we get extensions of $\eps$ and~$\mu$ to $J(\bar{k})$. Denote $\max \{\mu(P) : P \in J(\bar{k})\}$
by~$\bar{\beta}$ and $\max \{\eps(P) : P \in J(\bar{k})\}$ by~$\bar{\gamma}$.
Then by the discussion at the beginning of Section~\ref{nerong2} and the results above,
we find that
\[ \bar{\beta} = \frac{\bar{\gamma}}{4} = \frac{v(\Delta)}{4}\,, \]
when there are one or two nodes, and
\[ \frac{v(\Delta)}{6} \le \bar{\beta} = \frac{\bar{\gamma}}{4}
                       = \frac{v(\Delta) - \min\{m_1,m_2,m_3\}}{4} <
                       \frac{v(\Delta)}{4}\,,
\]
when there are three nodes. (Equality is achieved as soon as the Galois action
on~$R(C)$ is trivial and the ramification index is even.)

%====================================================================================

\section{Formulas and bounds for $\mu(P)$ in the cuspidal reduction case} \label{formulas2}

In this section we consider the case of a stably minimal Weierstrass model~$\calC$ such that
there are (one or two) points of multiplicity~$3$ on the special fiber.
These points are either both $\frk$-rational or they are defined over a quadratic extension of $\frk$ and
are conjugate over $\frk$.

In the notation of Namikawa and Ueno~\cite{NamiUeno}, the reduction type is of the form
$[\calK_1-\calK_2-l]$, where $l \ge 0$ and $\calK_j$ is an elliptic Kodaira type
for $j \in \{1,2\}$.
We can compute $\calK_1$, $\calK_2$ and $l$ as in~\cite{liuminimaux}*{\S6.1}.
By~\cite{liuminimaux}*{\S7}, we have
\[
  \Phi(\bar{\frk}) \cong \Phi_1(\bar{\frk}) \times \Phi_2(\bar{\frk})\,,
\]
where $\Phi_j$ is the component group of an elliptic curve with Kodaira type~$\calK_j$.
As in the previous section, we write $\Delta = \Delta(\calC)$ for the discriminant
of the model~$\calC$.

\begin{figure}[htb]
 \begin{center}
   \includegraphics[width=0.7\textwidth]{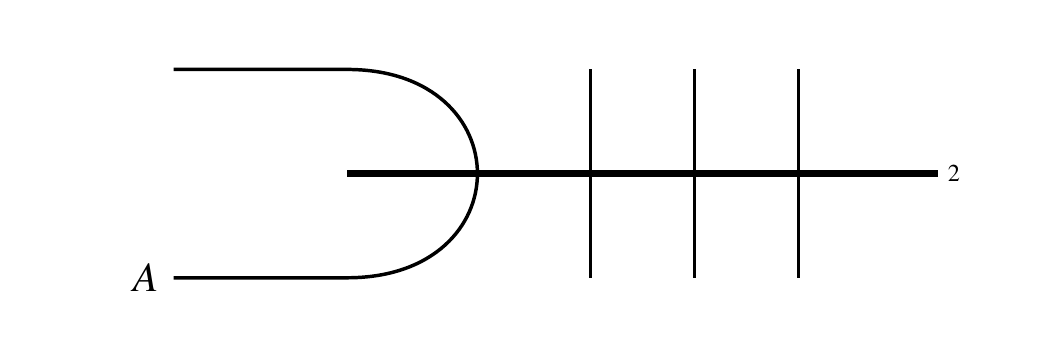}
   \caption{The special fiber of reduction type $[I_0-I^*_0-0]$}\label{picc41}
 \end{center}
\end{figure}

If $\calC$ is not regular, then we can compute
the minimal proper regular model $\calC^{\min}$ of $C$ from $\calC$
by a sequence of blow-ups in the singular point(s) of $\calC$, so the corresponding morphism
$\zeta:\calC^{\min}\to\calC$ is the minimal desingularization of $\calC$.

Suppose that $l>0$. Then the special fiber of~$\calC^{\min}$
consists of Kodaira types $\calK_1$ and~$\calK_2$, connected by a
chain of $l-1$ rational curves. See for example Figure~\ref{picss}.
The desingularization $\zeta$ contracts $\calK_2$ to one of the singular points; in this
case we say that this point {\em corresponds to $\calK_2$}.
If there is another singular point in $\calC_v(\bar{\frk})$, then it corresponds to
$\calK_1$, otherwise we must have $\calK_1 = I_0$.

Suppose now that $l=0$.
If both $\calK_1$ and~$\calK_2$ are good or
multiplicative, then we are in the situation $[I_{{m_1}-{m_2}-0}]$ for
some $m_1,m_2 \ge 0$, which we have discussed in the previous section.
So we may assume that at least one of the~$\calK_j$ is additive, say~$\calK_2$.
Then $\calC^{\min}_v$ looks like  Kodaira type~$\calK_2$, but with
one of the rational curves replaced by (see~\cite{NamiUeno})
\begin{itemize}
  \item a curve $A$ of genus~1 if $\calK_1 = I_0$ (see Figure~\ref{picc41} for the case
        $\calK_2 = I^*_{0}$);
  \item one of the rational components of~$\calK_1$, otherwise; the remainder of
        $\calK_1$ is then attached to this component.
\end{itemize}
We say that a singularity corresponds to one of the Kodaira types $\calK_1$ or
$\calK_2$ similarly to the case $l>0$.

\begin{lemma}\label{L:e1_e2}
  Suppose that the residue characteristic of~$k$ is not~2.
  Let $C$ be given by a stably minimal Weierstrass model with reduction type $[\calK_1-\calK_2-l]$.
  Then after at most a quadratic unramified extension of~$k$
  there is a stably minimal Weierstrass model
  \[ \calC \colon Y^2 = F(X,Z) = f_6 X^6 + f_5 X^5 Z + f_4 X^4 Z^2
                                    + X^3 Z^3 + f_2 X^2 Z^4 + f_1 X Z^5 + f_0 Z^6
  \]
  of~$C$, isomorphic to the given model of $C$, such that the elliptic curve with Weierstrass model
  \[ \calE_1 \colon Y^2 Z = X^3 + f_2 X^2 Z + f_1 X Z^2 + f_0 Z^3 \]
  has Kodaira type~$\calK_1$ and the elliptic curve with Weierstrass model
  \[ \calE_2 \colon Y^2 Z  = X^3 + f_4 X^2 Z + f_5 X Z^2 + f_6 Z^3 \]
  has Kodaira type~$\calK_2$.
\end{lemma}

\begin{proof}
  After possibly making a quadratic unramified extension and applying a transformation,
  we can assume that there is a unique point $\infty \in\calC_v(\frk)$ at infinity
  on the special fiber and that it is a cusp, corresponding to $\calK_2$,
  see the discussion preceding the lemma.
  Moreover, we can assume that if
  there is another singular point in $\calC_v(\bar{\frk})$, then this point
  is~$P=(0,0)\in\calC_v(\frk)$ (in which case it must correspond to $\calK_1$).

  Because the residue characteristic is not~2, we may assume that $\calC$ has $H=0$
  and that $f_3$ is a unit.
  By Hensel's Lemma there is a factorization $F=F_1F_2$, where $F_2$ is a cubic form
  reducing to $Z^3$.
  Similarly, we may assume that $F_1$ reduces to
  $X^3$ if there is a cusp at $P$ and to $X^2 (X + aZ)$ with $a \neq 0$
  if there is a node at $P$; otherwise $F_1$ is squarefree.
  Consider the elliptic curves given by the Weierstrass models
  \[ \calD_1 \colon Y^2Z = F_1(X,Z)\quad\text{ and }\quad
     \calD_2 \colon Y^2Z = F_2(Z,X)\,. \]
  We first show that $\calD_1$ has Kodaira type~$\calK_1$ and $\calD_2$ has Kodaira
  type~$\calK_2$.

  If $\calD_1$ is not minimal, then we can apply a transformation to~$\calC$ which
  makes $\calD_1$ minimal.
  This decreases the valuation of the discriminant~$\Delta(\calD_1)$,
  but increases the valuation of~$\Delta(\calD_2)$ by the same amount.
  The resulting model is still stably minimal and the resulting $F_2$
  still reduces to~$Z^3$. Hence we may assume that $\calD_1$ is minimal.

  Let $Q=(0,0)\in\calD_{1,v}(\frk)$; then $\calD_1$ is smooth outside~$Q$.
  Note that $F_2$ is a unit in~$\O_{\calC,P}$, so that
  $P$ is a smooth point if and only if $Q$ is a smooth point,
  in which case $\calD_1$ has reduction type $I_0=\calK_1$.
  More generally, $\calC$ is regular at~$P$ if and only if $\calD_1$ is regular at~$Q$,
  and $P$ is a node (resp., a cusp) if and only if $Q$ is a node (resp., a cusp).
  Recall that $P$ corresponds to~$\calK_1$, so that $\calD_1$ has reduction type $I_1$
  (resp., $I\!I$) if and only if $\calK_1 = I_1$ (resp., $\calK_1=I\!I$).

  Now suppose that $\calC$ is not regular at~$P$ and $\calD_1$ is not regular at~$Q$.
  The minimal desingularization $\xi \colon \calC'\to\calC$ in~$P$ can be computed by a sequence
  of blow-ups, starting with the blow-up of~$\calC$ in~$P$.
  The preimage of~$P$ under the latter map is contained in the chart~$\calC^1$
  obtained by dividing the $x$- and $y$-coordinates
  by the uniformizing element~$\pi$.
  Similarly, in order to compute the minimal desingularization $\xi_1 \colon \calD_1'\to \calD_1$
  in~$Q$, we first blow up~$\calD_1$ in~$Q$; then the chart~$\calD_1^1$ obtained
  by dividing the $x$- and $y$-coordinates by~$\pi$ contains the preimage of~$Q$.
  But because $F_2$ reduces to~$Z^3$, the special fibers of $\calC^1$ and~$\calD_1^1$ are
  identical. This continues to hold after further blow-ups (if any are necessary), so
  we have $\xi^{-1}(P) = \xi_1^{-1}(Q)$. There are no exceptional components in these
  preimages, since we assumed that $\calD_1$ is minimal. Therefore $\calD_1'$ is in fact
  the minimal proper regular model of the elliptic curve defined by~$\calD_1$.
  Since the minimal desingularization of~$\calC'$ in the point~$\infty\in\calC'_v(\frk)$
  leads to~$\calC^{\min}$, and since $P$ corresponds to~$\calK_1$,
  we deduce that $\calD_1$ has Kodaira type~$\calK_1$.

  A similar argument (for which we first apply a transformation to make $\calD_2$ minimal)
  shows that $\calD_2$ has Kodaira type~$\calK_2$.
  To complete the proof of the lemma, we therefore only need to make sure that
  $\calE_i$ has the same reduction type as~$\calD_i$ for $i=1,2$.
  This is certainly satisfied if the coefficients of~$\calE_i$ and~$\calD_i$
  agree modulo~$\pi^{N_i+1}$, where $N_i$ is the number of blow-ups needed to construct the minimal
  desingularization of~$\calD_i$.
  Suppose that
  $F_1 = a_0 Z^3 + a_1 X Z^2 + a_2 X^2 Z + a_3 X^3$ and
  $F_2 = b_3 Z^3 + b_2 X Z^2 + b_1 X^2 Z + b_0 X^3$.
  Writing out the coefficients of~$F$ in terms of the coefficients of~$F_1$ and~$F_2$,
  we see that it suffices to have
  \[
    v(a_0b_2)>v(a_1),\;v(b_0a_2)>v(b_1),\;v(a_0b_1+a_2b_2)>v(a_2),\;v(a_1b_0+a_2b_1)>v(b_2)\,.
  \]
  If this is not satisfied, it can be achieved by acting on the given stably minimal Weierstrass model
  via a suitable element of~$\GL_2(\O)$ as in~\S\ref{S:lambda_Kummer}.
  Finally, we scale the variables to get $f_3=1$.
\end{proof}

\begin{rk}\label{R:e1_e2_c2}
  If the residue characteristic is~2, then it is not hard to see that one can also construct a stably minimal
  Weierstrass model $\calC$ and corresponding elliptic Weierstrass models $\calE_1$ and $\calE_2$ as in the
  lemma in a similar way. The construction is more cumbersome, since we cannot assume $H=0$.
\end{rk}

In view of Theorem~\ref{T:epsfac} we want a condition for $\calC$ to have rational singularities.

\begin{lemma}\label{C4GeoMin}
  The model $\calC$ has rational singularities if and only if $l = 0$.
\end{lemma}

\begin{proof}
  We may assume that $\calC$ is as in Lemma~\ref{L:e1_e2} or Remark~\ref{R:e1_e2_c2}.
  Then all points in $\calC_v(\bar{\frk})\setminus\{\infty, P\}$ are non-singular,
  where $\infty\in\calC_v(\frk)$ is the unique point at infinity, and $P=(0,0)\in \calC_v(\frk)$.
  If $\calC$ is regular in $P$, then $P$ is a rational singularity.
  If not, then, by~\cite{Artin2}*{Thm.~3}, $P$ is a rational singularity  if and only if the fundamental cycle of
  $\xi^{-1}(P)$ has arithmetic genus~0, where $\xi$ is any desingularization of $P$.
  In particular, the assertion that $P$ is a rational singularity depends only on the configuration of
  $\xi^{-1}(P)$, where $\xi \colon \calC' \to \calC$ is the minimal desingularization of $P$.
  Now let $\calE_{1}$ be as in Lemma~\ref{L:e1_e2} or Remark~\ref{R:e1_e2_c2}, and
  let $\xi_1 \colon \calE'_1 \to \calE_1$ denote the minimal desingularization of the singular
  point~$Q=(0,0)\in \calE_{1,v}(\frk)$; then
  the assertion that $Q$ is a rational singularity depends only on the configuration of
  $\xi_1^{-1}(Q)$.
  We have $\xi^{-1}(P) =\xi_1^{-1}(Q)$ as in the proof of Lemma~\ref{L:e1_e2}
  (this also works when $\Char{\frk}=2$ and does not require minimality of $\calE_1$).
  In particular, $P$ is a rational singularity if and only if $Q$ is a rational singularity.

  A similar argument proves the corresponding statement for~$\calE_2$.
  Hence $\calC$ has rational singularities if and only if both $\calE_1$ and $\calE_2$
  have rational singularities.
  By~\cite{Conrad}*{Corollary~8.4} a Weierstrass model of an
  elliptic curve has rational singularities if and only if it is minimal.
  But it is easy to see that $\calE_1$ and $\calE_2$ are
  both minimal if and only if~$l=0$.
\end{proof}

According to Lemma~\ref{C4GeoMin}, not all singularities of
the given stably minimal Weierstrass model $\calC$ are rational when $l> 0$.
The following example shows that in this situation $\eps(P) \ne 0$, and hence
$\mu(P) \ne 0$, can indeed occur for $P\in J_0(k)$.

\begin{ex}\label{CountEx}
  Let $p$ be an odd prime and let $C/\Q_p$ be given by
  \[ Y^2 = Z (X^2 + Z^2) (X^3 + p^5 X Z^2 + p^8 Z^3)\, . \]
  Let $P_1 = (0,p^4) \in C(\Q_p)$ and $P_2 = \iota(P_1)$. The
  reduction type is $[I_0-I\!I\!I-1]$ and hence $\#\Phi(\bar{\frk}) = 2$.
  It turns out that both $P_1$ and $P_2$
  map to the same component and so we have $P = [(P_1) - (P_2)]\in J_0(k)$.
  The image of~$P$ on the Kummer surface is of the form $(x_1:0:0:x_4)$, where
  $v(x_4)-v(x_1)=2$. We get $\eps(P)=\eps(2P)=6$ and $\mu(P)=\mu(2P)=2$.
\end{ex}

\begin{figure}[tb]
  \begin{center}
    \includegraphics[width=\textwidth]{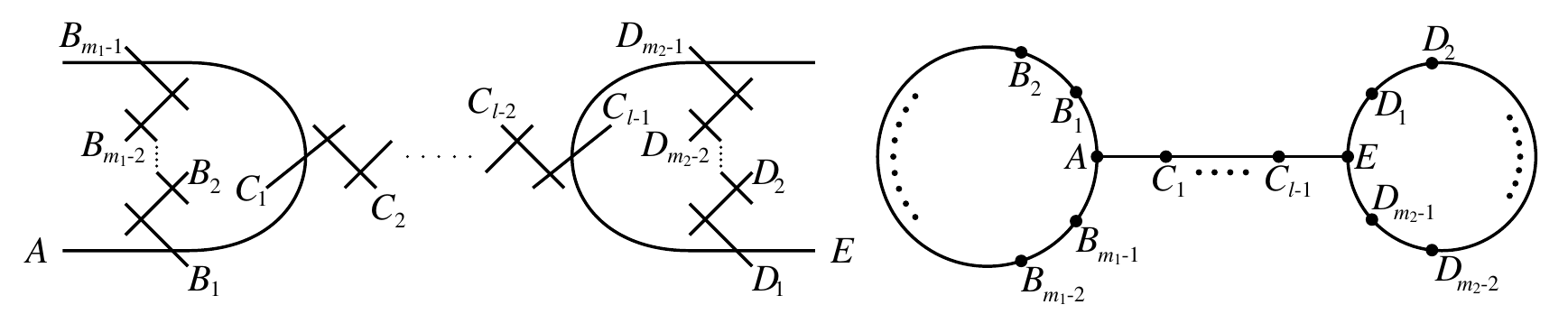}
    \caption{The special fiber of reduction type $[I_{m_1}-I_{m_2}-l]$
             and its reduction graph}\label{picss}
  \end{center}
\end{figure}

The case of semistable reduction, corresponding to reduction type $[I_{m_1}-I_{m_2}-l]$,
see Figure~\ref{picss}, deserves special attention. Here $l \ge 1$, by the discussion above.
Note that $m_1 = 0$ (or $m_2 = 0$) is possible; in that case $A$ (or~$E$) is a
curve of genus~$1$ and there are no components~$B_i$ (or~$D_i)$. If $m_1 = 1$ (or $m_2 = 1$),
then $A$ (or~$E$) is a nodal curve (and again there are no $B_i$ or~$D_i$).
After perhaps
an unramified quadratic extension, we can assume that all components in the `chain'
that connects the two polygons in the special fiber of~$\calC^{\min}$ are defined over~$\frk$.
There are then $l+1$ different (meaning pairwise non-isomorphic over~$\O$) minimal Weierstrass
models of the curve, compare the proof of Lemma~\ref{L:allminstab}.
Explicitly, these models can be taken to have the form
\begin{align} \nonumber
  \calC_j \colon Y^2 + (h_0 \pi^{3j} Z^3
       &+ h_1 \pi^j Z^2 X + h_2 \pi^{l-j} Z X^2 + h_3 \pi^{3(l-j)} X^3) Y \\
    &= f_0 \pi^{6j} Z^6 + f_1 \pi^{4j} X Z^5 + f_2 \pi^{2j} X^2 Z^4 + X^3 Z^3  \label{E:EEred} \\
    &\qquad{} + f_4 \pi^{2(l-j)} X^4 Z^2 + f_5 \pi^{4(l-j)} X^5 Z + f_6 \pi^{6(l-j)} X^6 \nonumber
\end{align}
for $j = 0,1,\ldots,l$, where
\begin{align*}
 y^2 + h_1 x y + h_0 y &= x^3 + f_2 x^2 + f_1 x + f_0  \qquad \text{and} \\
 y^2 + h_2 x y + h_3 y &= x^3 + f_4 x^2 + f_5 x + f_6
\end{align*}
are minimal Weierstrass equations of elliptic curves of reduction types $I_{m_1}$
and $I_{m_2}$, respectively.
Such a model corresponds to the vertex~$C_j$
of the reduction graph (where we set $C_0 = A$ and $C_l = E$); the corresponding
component of the special fiber of~$\calC^{\min}$ is the one that is visible in the
special fiber of~$\calC_j$. The valuation of the discriminant of~$\calC_j$ is
$m_1 + m_2 + 12 l$ and does not depend on~$j$.

A \emph{simple path} in~$R(C)$ is a subgraph that is a tree without vertices
of valency~$\ge 3$.
Let $P_1, P_2 \in C(k)$ reduce to components $\Gamma_1$ and~$\Gamma_2$ of the special
fiber of~$\calC^{\min}$, respectively. Consider the model~$\calC_j$ of~$C$. If
there is a simple path from $\Gamma_1$ to~$\Gamma_2$ in the reduction graph that
passes through~$C_j$, then we say that $\calC_j$ \emph{lies between} $P_1$ and~$P_2$.
We denote the $\mu$-function computed with respect to~$\calC_j$ by $\mu_j$.

\begin{prop} \label{P:EEred}
  Assume that $C$ has semistable reduction of type $[I_{m_1}-I_{m_2}-l]$.
  Let $P_1, P_2 \in C(k)$ be points reducing to components $\Gamma_1$ and~$\Gamma_2$ of the special
  fiber of~$\calC^{\min}$ and let $j \in \{0,1,\ldots,l\}$. Define $j_{\min}$
  and~$j_{\max}$ to be the smallest, respectively largest, $j' \in \{0,1,\ldots,l\}$
  such that $\calC_{j'}$ lies between $P_1$ and~$P_2$.
  Let $P = [(P_1)-(P_2)] \in J(k)$. Then
  \[ r(\Gamma_1, \Gamma_2) + j_{\max} - j_{\min}
       \le \mu_j(P)
       \le r(\Gamma_1, \Gamma_2) + |j - j_{\max}| + |j - j_{\min}| \,.
  \]
  If $\calC_j$ lies between $P_1$ and~$P_2$, then the inequalities are equalities.
\end{prop}

\begin{proof}
  First note that the last statement follows from the first, since
  $j_{\min} \le j \le j_{\max}$ implies $j_{\max} - j_{\min} = |j - j_{\max}| + |j - j_{\min}|$.

  Let $B_0 = B_{m_1} = A$ and $D_0 = D_{m_2} = E$.  We prove a number of lemmas.

  \begin{lemma} \label{L:boundary}
    If $j = j_{\max} = j_{\min} \in \{0, l\}$, then $\mu_j(P) = r(\Gamma_1, \Gamma_2)$.
  \end{lemma}

  \begin{proof}
    We assume that $j = j_{\max} = j_{\min} = l$; the other case is analogous.
    Then $\Gamma_1$ and~$\Gamma_2$ are both of the form~$D_i$, and we consider the model~$\calC_l$.
    We first claim that $\mu(P) = 0$ if $\Gamma_1 = \Gamma_2$, but the images
    of $P_1$ and~$P_2$ on~$\Gamma_1$ are distinct. This is clear if $\Gamma_1 = D_0 = E$,
    since in this case $P$ is in the image of~$\alpha$, compare
    Lemmas~\ref{L:mu_vanishes1} and~\ref{L:mu_vanishes2}. Otherwise, we note that
    the points with nonzero multiplicity on the special fiber of~$\calC_l$
    have multiplicities $1$, $2$ and~$3$.
    Transforming the equation over~$\O$ if necessary, we can assume that its reduction
    is case~7 in Table~1 of~\cite{StollH2} or (if the residue characteristic is~$2$)
    case~5 in Table~\ref{condmult} here.

    Recall that $\Gamma_1 = \Gamma_2 = D_i$, where we can assume $0 < i \le m_2/2$.
    Applying a transformation, we may assume that the points $P_1=(\xi_1:\eta_1:1)$ and
    $P_2=(\xi_2:\eta_2:1)$ both reduce to
    $(0:0:1)$ modulo $\pi$ and that $m_2 = \min\{v(f_0), 2v(f_1)\}$.
    First suppose that $i<m_2/2$.
    We then have $v(\xi_1) = v(\xi_2) = v(\xi_1-\xi_2) = i$.
    Normalizing the Kummer coordinates $x$ of $P$ so that $x_1=1$, we can check that $v(x_2)$ and
    $v(x_3)$ are positive, but that $v(x_4) = 0$.
    This follows because $\Gamma_1 = D_i = \Gamma_2$ implies that
    $v(f_2\xi_1\xi_2+2\eta_1\eta_2) = 2i$ if $\Char(\frk) \ne 2$ and $H=0$ and that
    $v(\xi_1\eta_2+\xi_2\eta_1)= 2i$ if $\Char(\frk) =2$.
    By a similar argument, the reduction of the
    image of~$P$ on the Kummer surface has non-vanishing last coordinate if $m_2$
    is even and $i=m_2/2$.
    According to the tables, this implies that $\eps(P) = 0$ and therefore
    also $\mu(P) = 0$.

    Now consider the case that $\Gamma_1$ and~$\Gamma_2$ do not necessarily coincide.
    The considerations above imply that the assumptions
    of Proposition~\ref{P:muresist} are satisfied with $\mu_1 = \mu_2 = 0$ (where we
    use Lemma~\ref{L:mu=0} for the first assumption); the
    proposition then establishes the claim.
  \end{proof}

  \begin{lemma} \label{L:mu=0_same_comp}
    Assume that $\Gamma_1 = \Gamma_2 = C_j$ with $0 < j < l$. Then $\mu_j(P) = 0$.
  \end{lemma}

  \begin{proof}
    In this case, $P$ is in the image of~$\alpha$, so the claim follows
    by Proposition~\ref{P:mu_vanishes}.
  \end{proof}

  Note that Lemmas~\ref{L:boundary} and~\ref{L:mu=0_same_comp} establish the claim
  of Proposition~\ref{P:EEred} in all cases such that $j = j_{\min} = j_{\max}$.

  \begin{lemma} \label{L:mu_constant}
    Assume that both $\calC_j$ and $\calC_{j+1}$ lie between $P_1$ and~$P_2$, where
    $0 \le j < l$. Then $\mu_j(P) = \mu_{j+1}(P)$.
  \end{lemma}

  \begin{proof}
    Let $\tau \colon (\xi : \eta : \zeta) \mapsto (\pi \xi : \eta : \pi^{-1} \zeta)$; then $\tau$ gives
    an isomorphism from the generic fiber of~$\calC_j$ to that of~$\calC_{j+1}$.
    The induced map on Kummer coordinates is
    \[ (x_1, x_2, x_3, x_4) \longmapsto (\pi^{-2} x_1, x_2, \pi^2 x_3, x_4)\,; \]
    we have $v(\tau) = 0$. Since both $\calC_j$ and~$\calC_{j+1}$ lie between
    $P_1$ and~$P_2$, assuming that $\Gamma_1$ is to the left and $\Gamma_2$ to the right
    of $C_j$ and~$C_{j+1}$, we must have that the $x$-coordinate of~$P_1$ on~$\calC_j$
    does not reduce to infinity, whereas that of~$P_2$ does. For normalized Kummer
  coordinates $x = (x_1, x_2, x_3, x_4)$ of~$P$ on the Kummer surface associated
    to~$\calC_j$, this implies $v(x_2) = 0$ (the point is not in the kernel of
    reduction, so $v(x_4) \ge \min\{v(x_1), v(x_2), v(x_3)\}$) and $v(x_1) > 0$.
    Comparing valuations in the equation of~$\calC_j$, we see that $P_2 = (1 : \eta : \zeta)$
    must have $v(\zeta) \ge 2$, which implies $v(x_1) \ge 2$. It follows that
    $v(\tau(x)) = 0 = v(x)$. By Corollary~\ref{C:lambdatau} we also have
    $\hat{\lambda}(\tau(x)) = \hat{\lambda}(x)$ (recall that $v(\tau) = 0$).
    Since
    \[ -v(x) - \mu_j(P) = \hat{\lambda}(x) = \hat{\lambda}(\tau(x))
                        = -v(\tau(x)) - \mu_{j+1}(P)\,,
    \]
    the claim follows.
  \end{proof}

  \begin{lemma} \label{L:mu_factors}
    If $\calC_j$ lies between $P_1$ and~$P_2$, then $\mu_j(P)$ depends
    only on $\Gamma_1$ and~$\Gamma_2$.
  \end{lemma}

  \begin{proof}
    Let $P_1', P_2' \in C(k)$ be points also mapping to $\Gamma_1$ and~$\Gamma_2$, respectively.
    We assume without loss of generality
    that $\Gamma_1$ is to the left of~$\Gamma_2$. By Lemmas~\ref{L:boundary}
    or~\ref{L:mu=0_same_comp}, we have that $\mu_{j_{\min}}([(P_1)-(P_1')]) = 0$
    and $\mu_{j_{\max}}([(P_2)-(P_2')]) = 0$. Using Lemmas~\ref{L:mu_constant} and~\ref{L:mu=0}, we obtain
    \begin{align*}
      \mu_j\bigl([(P_1')-(P_2')]\bigr)
         &= \mu_{j_{\min}}\bigl([(P_1')-(P_2')]\bigr) = \mu_{j_{\min}}\bigl([(P_1)-(P_2')]\bigr) \\
         &= \mu_{j_{\max}}\bigl([(P_1)-(P_2')]\bigr) = \mu_{j_{\max}}\bigl([(P_1)-(P_2)]\bigr)
          = \mu_j(P) \,.
       \qedhere
    \end{align*}
  \end{proof}

  \begin{lemma} \label{L:mu_changes}
    Let $P_1',\, P_2' \in C(k)$ be points mapping to distinct points on the same component
    of the special fiber of~$\calC^{\min}$ and let $P' = [(P'_1) - (P'_2)] \in J(k)$.
    Let $j_0$ be the unique index such that
    $\calC_{j_0}$ lies between $P_1'$ and~$P_2'$. Then $\mu_j(P') = 2 |j-j_0|$.
  \end{lemma}

  \begin{proof}
    By Lemmas~\ref{L:boundary}
    and~\ref{L:mu=0_same_comp}, we have $\mu_{j_0}(P') = 0$. Since
    the images of $P_1'$ and~$P_2'$ on the special fiber of~$\calC^{\min}$ are distinct,
    $P'$ is not in the kernel of reduction with respect to~$\calC_{j_0}$. If
    \[ x^{(j_0)} = (x_1^{(j_0)}, x_2^{(j_0)}, x_3^{(j_0)}, x_4^{(j_0)}) \]
    are normalized Kummer coordinates for $P'$ on the Kummer surface associated
    to~$\calC_{j_0}$, we therefore have
    \[ 0 = v(x^{(j_0)}) = \min\{v(x_1^{(j_0)}), v(x_2^{(j_0)}), v(x_3^{(j_0)})\}\,. \]
    Applying a suitable power of~$\tau$ (see the proof of Lemma~\ref{L:mu_constant}), we find that
    \[ x^{(j)} = (\pi^{2(j_0-j)} x_1^{(j_0)}, x_2^{(j_0)}, \pi^{2(j-j_0)} x_3^{(j_0)}, x_4^{(j_0)}) \]
    are (not necessarily normalized) Kummer coordinates for~$P'$ on the Kummer
    surface associated to~$\calC_j$. For definiteness, assume that $j > j_0$, the case
    $j = j_0$ being clear. Similarly to the proof of Lemma~\ref{L:mu_constant},
    we find that $0 = v(x^{(j_0)}) = v(x_1^{(j_0)})$, which implies that
    $v(x^{(j)}) = -2(j-j_0)$. In the same way as in the proof of Lemma~\ref{L:mu_constant},
    we deduce $\mu_j(P') = 2(j-j_0) = 2 |j-j_0|$.
  \end{proof}

  To continue the proof of the proposition, we now first consider the case that $\calC_j$
  lies between $P_1$ and~$P_2$. In this case, Lemmas \ref{L:mu_factors} and~\ref{L:mu_changes}
  show that the assumptions in Proposition~\ref{P:muresist} hold with
  $\mu_1 = 2|j - j_{\min}|$ and $\mu_2 = 2|j - j_{\max}|$ or conversely.
  So the statement follows from Proposition~\ref{P:muresist} and
  $|j - j_{\max}| + |j - j_{\min}| = j_{\max} - j_{\min}$.

  Now assume that $\calC_j$ does not lie between $P_1$ and~$P_2$.
  We assume for definiteness that $j > j_{\max}$.
  For normalized Kummer coordinates $x^{(j_{\max})}$ for~$P = [(P_1)-(P_2)]$
  on the Kummer surface associated to~$\calC_{j_{\max}}$,
  we have $v(x^{(j_{\max})}_2) \le \min \{v(x^{(j_{\max})}_1), v(x^{(j_{\max})}_3)\}$,
  compare the proof of Lemma~\ref{L:mu_constant} above.
  Then $x^{(j)} = \tau^{j-j_{\max}}(x^{(j_{\max})})$ are Kummer
  coordinates for~$[(P_1)-(P_2)]$ on the Kummer surface associated to~$\calC_j$, and we have
  \[ v(x^{(j_{\max})}) - 2 (j - j_{\max}) \le v(x^{(j)}) \le v(x^{(j_{\max})})\, . \]
  It follows that
  \begin{align*}
    \mu_j(P) &- \mu_{j_{\max}}(P) \\
       &= \bigl(-\hat{\lambda}(x^{(j)}) - v(x^{(j)})\bigr)
           - \bigl(-\hat{\lambda}(x^{(j_{\max})}) - v(x^{(j_{\max})})\bigr) \\
       &= v(x^{(j_{\max})}) - v(x^{(j)}) \in \{0, 1, \ldots, 2 (j - j_{\max})\} .
  \end{align*}
  As $\mu_{j_{\max}}(P) = r(\Gamma_1, \Gamma_2) + j_{\max} - j_{\min}$
  by the case already discussed, the result follows, and the proof
  of Proposition~\ref{P:EEred} is finished.
\end{proof}

\begin{cor} \label{C:EEbound}
  Let $\calC$ be a stably minimal Weierstrass model of~$C$ with discriminant~$\Delta$;
  assume that $C$ has reduction type $[I_{m_1} - I_{m_2} - l]$ with $l > 0$.
  As usual, let
  \[ \beta(\calC) = \max \{\mu(P) : P \in J(k)\} \qquad\text{and}\qquad
     \bar{\beta}(\calC) = \max \{\mu(P) : P \in J(\bar{k})\}\,,
  \]
  where $\mu$ is computed with respect to~$\calC$.
  Then we have
  \[ \beta(\calC) \le \bar{\beta}(\calC) = \frac{m_1 + m_2}{4} + 2l < \frac{v(\Delta)}{4}
     \qquad\text{and}\qquad \bar{\beta} \ge \frac{v(\Delta)}{6} \,.
  \]
\end{cor}

\begin{proof}
  The assumption on the reduction type implies that the model is equivalent
  to one of the form~\eqref{E:EEred}. Proposition~\ref{P:EEred} then gives upper
  bounds for $\mu([(P_1)-(P_2)])$, with $P_1, P_2 \in C(\bar{k})$,
  depending on the images $\Gamma_1$ and~$\Gamma_2$ of $P_1$ and~$P_2$ in the reduction graph.
  The maximizing case occurs for $\Gamma_1 = B_{m_1/2}$
  and $\Gamma_2 = D_{m_2/2}$, giving
  \[ \mu([(P_1)-(P_2)]) = r(B_{m_1/2}, D_{m_2/2}) + l
                    = \tfrac{1}{4} m_1 + l + \tfrac{1}{4} m_2 + l\, .
  \]
  For the remaining inequalities, recall that $v(\Delta) = m_1 + m_2 + 12 l$ and that $l > 0$.
\end{proof}

We state a technical lemma, which will be needed for the proof of Theorem~\ref{T:33bound} below.

\begin{lemma} \label{L:CE}
  Suppose that the residue characteristic of~$k$ is not~2.
  Consider a degenerate Weierstrass equation of the form
  \[ \calC \colon Y^2  = f_0 Z^6 + f_1 X Z^5 + f_2 X^2 Z^4 + X^3 Z^3 \]
  and let
  \[ \calE \colon y^2  = f_0 + f_1 x + f_2 x^2 + x^3 \]
  be an elliptic Weierstrass equation. If $Q_1 = (x_1, y_1)$ and $Q_2 = (x_2, y_2)$
  are points in~$\calE(k)$, then $P_1 = (x_1 : y_1 : 1)$ and $P_2 = (x_2 : y_2 : 1)$
  are points in~$\calC(k)$, and if $x_1, x_2 \in \O$, then
  $\mu_{\calC}([(P_1) - (P_2)]) \le \mu_{\calE}(Q_1 - Q_2)$.
\end{lemma}

Here $\mu_{\calE}$ is the height correction function for the elliptic curve~$\calE$
and $\mu_{\calC}$ denotes the height correction function defined in the same way as~$\mu$
in the smooth case in terms of the equation~$\calC$.

\begin{proof}
  Let $\underline{\delta}_{\calC}
      = (\delta_{\calC,1}, \delta_{\calC,2}, \delta_{\calC,3}, \delta_{\calC,4})$
  be the duplication polynomials on the Kummer surface associated to~$\calC$,
  and let $\underline{\delta}_{\calE} = (\delta_{\calE,1}, \delta_{\calE_2})$
  be the duplication polynomials
  for the numerator and denominator of the $x$-coordinate associated to~$\calE$.
  Then a generic computation shows that, if $(\xi_1 : \xi_2 : \xi_3 : \xi_4)$
  is the image of $[(P_1)-(P_2)]$ on the Kummer surface, we have $(\xi_4 : \xi_1) = x(Q_1-Q_2)$.
  In addition, we find that (as polynomials in the~$\xi_j$)
  $\delta_{\calC,1}(\xi_1,\xi_2,\xi_3,\xi_4) = \delta_{\calE,2}(\xi_4,\xi_1)$
  and $\delta_{\calC,4}(\xi_1,\xi_2,\xi_3,\xi_4) = \delta_{\calE,1}(\xi_4,\xi_1)$.

  That $P_1, P_2 \in \calC(k)$ is obvious from the equations. For the last statement,
  we observe that $\min\{v(\xi_1), v(\xi_2), v(\xi_3), v(\xi_4)\} = \min\{v(\xi_1), v(\xi_4)\}$
  (this is where we use that $x_1$ and~$x_2$ are integral), which implies
  \begin{align*}
    \mu_{\calC}([(P_1)-(P_2)])
      &= \lim_{n \to \infty} 4^{-n} v\bigl(\underline{\delta}_{\calC}^{\circ n}(\underline{\xi})\bigr)
           - v(\underline{\xi}) \\
      &\le \lim_{n \to \infty} 4^{-n} v\bigl(\underline{\delta}_{\calE}^{\circ n}(\xi_4, \xi_1)\bigr)
           - \min\{v(\xi_1), v(\xi_4)\} \\
      &= \mu_{\calE}(Q_1 - Q_2) \,. \qedhere
  \end{align*}
\end{proof}

The following consequence is useful for practical purposes.
For simplicity, we state it for the case of  residue characteristic $\ne 2$,
but we expect that the statement remains true for residue characteristic~$2$.

\begin{thm} \label{T:33bound}
  Suppose that the residue characteristic of~$k$ is not~2.
  Let $\calC$ be a stably minimal Weierstrass model of~$C$ such that $C$ has reduction
  type $[\calK_1 - \calK_2 - l]$. Then
  \[ \beta(\calC) \le \beta(\calK_1) + \beta(\calK_2) + 2l \,, \]
  where $\beta(\calK)$ denotes the maximum of~$\mu$ for an elliptic curve of
  reduction type~$\calK$ (taking the action of Frobenius into account), see Table~1
  in~\cite{CrePriSik}.
\end{thm}

\begin{proof}
  We may assume that the point(s) of multiplicity~$3$ on the special fiber are
  defined over~$\frk$, at the cost of an at most quadratic unramified extension of~$k$.
  Then we can move these points to have $x$-coordinates $0$ and~$\infty$, respectively, and so we can
  assume that our model~$\calC$ is as in Lemma~\ref{L:e1_e2}.
  Let $P \in J(k)$; we write $P = [(P_1)-(P_2)]$ with points $P_1, P_2 \in C(k')$
  for a finite extension~$k'$ of~$k$ such that the reduction of~$C$ over~$k'$ is semistable.
  We can find $\calC_0$, $\calC = \calC_j$ and $\calC_l$ as vertices in the reduction graph
  of the minimal proper regular model of~$C$ over~$k'$. Then the part of the graph
  to the left of~$\calC_0$ corresponds to the reduction graph of~$\calE_1$ over~$k'$,
  in the sense that we consider a semistable model that dominates~$\calE_1$ (and is minimal
  with that property); the graph then is either a line segment (potentially good reduction)
  or a line segment joined to a circle (potentially multiplicative reduction), with
  $\calE_1$ corresponding to the end of the line segment joined to the remaining graph
  of~$\calC$. Similarly, the part of the graph to the right of~$\calC_l$
  corresponds to the reduction graph of~$\calE_2$ over~$k'$.

  Now assume that both $P_1$ and~$P_2$ map (strictly) to the left of~$\calC_0$ in the reduction graph.
  This means that the $x$-coordinates of the points have positive valuation.
  We can then find points $P'_1$ and~$P'_2$ in~$\calE_1(k')$ with the same $x$-coordinates
  as $P_1$ and~$P_2$ and nearby $y$-coordinates. Then $P'_1 - P'_2$ is in~$\calE_1(k)$
  and $P'_1$ and~$P'_2$ have the same images as $P_1$ and~$P_2$ in the reduction graph.
  By our previous results for the semistable case, the value of (or at least the upper bound
  given in Proposition~\ref{P:EEred} for) $\mu_0(P)$ depends only
  on the part of the graph to the left of~$\calC_0$. We can therefore let $l$ tend
  to infinity; then Lemma~\ref{L:CE}
  and the discussion preceding Lemma~\ref{C4GeoMin} show
  that $\mu_0(P)$ is bounded by the value
  of~$\mu_{\calE_1}$ on the difference $P'_1 - P'_2$.
  By the arguments in the proof of Proposition~\ref{P:EEred}, we have that
  \[ \mu_{\calC}(P) = \mu_j(P) \le \mu_0(P) + 2 j \le \beta(\calK_1) + 2 l\, . \]
  The case that $P_1$ and~$P_2$ both map to the right of~$\calC_l$ is similar.

  If (say) $P_1$ maps to the left of~$\calC_0$ and $P_2$ maps to the right of~$\calC_0$,
  but not to the right of~$\calC_l$, then by the formula of Proposition~\ref{P:EEred},
  we can bound $\mu_{\calC}(P)$ by $\mu_1 + 2l$, where $\mu_1$ comes from the part
  of the graph between $P_1$ and~$\calC_0$. By an argument similar to the one used
  in the previous paragraph, $\mu_1$ can be bounded by $\mu_{\calE_1}(P'_1)$, where
  $P'_1$ is the point on~$\calE_1$ corresponding to~$P_1$ and we take the second point
  to be on the component visible in~$\calC_0$. If $P_2$ maps to the right of~$\calC_l$,
  then we similarly obtain a bound of the form
  $\mu_1 + \mu_2 + 2l \le \beta(\calK_1) + \beta(\calK_2) + 2l$.
  The remaining cases are similar or follow directly from Proposition~\ref{P:EEred}.
\end{proof}

The example in Section~\ref{Ex:hdiff} demonstrates the effect of the improved bounds
on~$\beta$ as given in the preceding section. For other examples the bounds
established in this section will be similarly useful.

%=======================================================================================

\section{General upper and lower bounds for $\bar{\beta}$} \label{UpperBeta}

In this section we derive an upper bound for the geometric height constant
$\bar{\beta}(\calC)$ in the general case by reducing to the semistable situation.
We also give a lower bound of the same order of magnitude.
We note the following consequence of the results obtained so far,
see the discussion at the end of Section~\ref{formulas} and Corollary~\ref{C:EEbound}.

\begin{cor} \label{C:stabminbound}
  Assume that $\calC$ is a stably minimal Weierstrass model of~$C$ over~$k$
  and that the minimal proper regular model~$\calC^{\min}$ of~$C$ over~$k$
  has semistable reduction. Denoting the discriminant of~$\calC$ by~$\Delta$
  and writing $\bar{\beta}(\calC) = \max \{\mu_{\calC}(P) : P \in J(\bar{k})\}$,
  where $\mu_{\calC}$ denotes~$\mu$ with respect to the model~$\calC$ and
  $J$ is the Jacobian of~$C$, we have
  \[ \frac{v(\Delta)}{6} \le \bar{\beta}(\calC) \le \frac{v(\Delta)}{4}\, . \]
\end{cor}

When $\calC^{\min}$ does not have semistable reduction, the idea is to pass to a
suitable field extension $k'/k$ and apply Corollary~\ref{C:stabminbound} over $k'$.
In order to compare the corresponding geometric height constants~$\bar{\beta}$,
we need to analyze how $\mu$ changes under minimization.
We first prove the following key lemma:

\begin{lemma}\label{L:Minimization}
  There exists a transformation $\tau \colon \calC \to \calC'$, defined over $k$, such
  that $\calC'$ is a minimal Weierstrass model and
  \[
    v(\tau(x)) + v(\tau) \le v(x)\textrm{ for all }x \in \KS_\A\,.
  \]
\end{lemma}

\begin{proof}
    If $\calC$ is already minimal, then there is nothing to prove.
    Otherwise,~\cite{Liu2}*{Remarque~11} implies that we can compute a minimal Weierstrass
    model by going through the following steps for finitely many points $P$ on the
    special fiber of $\calC$.
    \begin{enumerate}[\upshape (a)]
      \item \label{I:Translate} Move $P$ to $(0,0)$.
      \item \label{I:Divide} Scale $x$ by $1/\pi$.
      \item \label{I:Normalize} Replace $\calC$ by the normalization of the resulting model.
    \end{enumerate}
    As transformations of the form~\eqref{I:Translate}  do not change $v(x)$ and have determinant of
    valuation~0, it suffices to prove
    \[
      v(\tau(x)) + v(\tau) \le v(x)\textrm{ for all }x \in \KS_\A
    \]
    for a transformation $\tau=\sigma\circ\rho$, where $\rho$ is as in~\eqref{I:Divide}
    and $\sigma$ is as in~\eqref{I:Normalize}.
    Note that such a transformation decreases the valuation of the discriminant,
    cf.~\cite{Liu2}*{Lemme~9} and~\cite{Liu2}*{Corollaire~2}.
    By the discussion following Proposition~\ref{lhg2isog}, the
    transformation $\rho$ maps $x \in \KS_\A$ to $(\pi x_1, x_2,
    \pi^{-1} x_3, \pi^3 x_4)$.

    Suppose $v(2) = 0$ and, without loss of generality, $H=0$.
    According to~\cite{Liu2}*{Remarque~2}, the normalization can be computed using the
    transformation $\sigma$ mapping an affine point
    $(\xi, \eta)$ to $\sigma(\xi,\eta) = (\xi, \eta \pi^{-s})$ for some nonnegative integer $s$.
    As $v(\tau) = 3-2s$, we must have $s \ge 2$, since otherwise $\tau$ would
    increase the valuation of the discriminant.
    Because $\tau(x) = (\pi x_1, x_2, \pi^{-1}x_3,
    \pi^{3-2s}x_4)$ for $x \in \KS_\A$,
    we find that $v(\tau(x)) \le v(x) + 1$, implying
    \[
        v(\tau(x)) + v(\tau) - v(x) \le -2s + 4 \le 0\,.
    \]
    The case $v(2) > 0$ is slightly more complicated.
    Here one computes the normalization by repeatedly applying
    transformations
    \begin{equation}\label{E:c2_normalize}
    (\xi, \eta)\mapsto\left(\xi,\, \frac{\eta+R(\xi,1)}{\pi}\right),
    \end{equation}
    where $R \in \O[X,Z]$ is a
    certain cubic form, until the minimum of the valuations of the
    coefficients of $F + RH -R^2$ is equal to~1. See~\cite{Liu2}*{Remarque~2}.
    Such a transformation maps Kummer coordinates $x = (x_1,x_2,x_3,x_4)$ to
    \[
         \left(x_1, x_2, x_3, \pi^{-2}x_4+l_1x_1+l_2x_2 + l_3x_3 \right)
    \]
    and the expressions for the $l_i$ given in
    Section~\ref{S:lambda_Kummer} show that $v(l_i) \ge -2$ for all
    $i$.
    As the determinant of a transformation~\eqref{E:c2_normalize} has valuation $-2$,
    we need to apply at least two such transformations,
    because otherwise the valuation of the discriminant would increase.
    In other words, $\sigma = \sigma_s\circ\cdots\circ\sigma_1$
    where $s \ge 2$ and every $\sigma_i$ is of the form~\eqref{E:c2_normalize}.

    By the properties of the transformations~\eqref{E:c2_normalize}, it suffices to
    show the desired inequality for the case $s = 2$, since further
    applications of transformations~$\sigma_i$ will only make the left
    hand side of the desired inequality smaller and will not change the right hand side.
    So suppose that
    $\sigma = \sigma_2\circ\sigma_1$; then
     $\tau = \sigma\circ \rho$ maps $x \in \KS_\A$ to
    \[
        \tau(x) = \left(\pi x_1,\, x_2, \,\pi^{-1}x_3, \,\pi^{-1}x_4+\pi
        l_1x_1+\pi l_2x_2 + \pi l_3x_3 + \pi l'_1x_1 + l'_2x_2
        +\pi^{-1}l'_3x_3\right)\,,
    \]
    where the $l_i$ arise from $\sigma_1$ and the $l'_i$ arise from
    $\sigma_2$.
    As $v(\tau) = -1$, it clearly suffices to prove that
    \begin{equation}\label{E:char2_ineq}
        v(\tau(x)) \le v(x) +1.
    \end{equation}
    But if~\eqref{E:char2_ineq} is false, then $v(x) = v(x_4) < \min\{v(x_1), v(x_2)+1, v(x_3)+2\}$.
    In this situation it follows from the lower bounds $v(l_i) \ge -2$ and
    $v(l'_i) \ge -2$ that we get
    \[
        v\bigl(\pi l_1x_1+\pi l_2x_2 + \pi l_3x_3 + \pi l'_1x_1 + l'_2x_2
        +\pi^{-1}l'_3x_3\bigr)
         > v(x_4) -1\,.
     \]
    This implies~\eqref{E:char2_ineq} and therefore finishes the proof of the lemma.
\end{proof}

\begin{thm} \label{T:UpperBd}
    Let $C$ be a smooth projective curve of genus~2 defined
    over a non-archimedean local field $k$, given by an integral Weierstrass
    model~$\calC$. Then we have
    \[
        \bar{\beta}(\calC) \le \frac{v(\Delta(\calC))}{4}\,.
    \]
\end{thm}

\begin{proof}
  By Lemma~\ref{L:allminstab} there is a finite extension $k'/k$ such that
  the minimal proper regular model of $C$ over~$k'$ is semistable
  and such that all minimal Weierstrass models of~$C$ over~$k'$ are stably minimal.
  By Corollary~\ref{C:stabminbound}, the claim therefore holds for any minimal
  Weierstrass model of~$C$ over~$k'$.

  It follows from Lemma~\ref{L:Minimization} that there is a transformation
  $\tau \colon \calC\to \calC'$ defined over $k'$ such that $\calC'$ is
  a minimal (and hence stably minimal) Weierstrass model over~$k'$ and such that
  \begin{equation}\label{E:tau_bd}
      v(\tau(x)) + v(\tau) \le v(x)
  \end{equation}
  for all $x \in \KS_\A$.

  Then by the above we have
  \[
      \mu(\tau(x)) \le \frac{v(\Delta(\calC'))}{4}\, .
  \]
  Now using Corollary~\ref{C:lambdatau} and the relation~\eqref{E:disc_change}, we find
  \begin{align*}
    \mu(x) &= \mu(\tau(x)) -v(x) + v(\tau(x)) - v(\tau) \\
           &\le \frac{v(\Delta(\calC'))}{4}-v(x) + v(\tau(x)) - v(\tau) \\
           &= \frac{v(\Delta(\calC))}{4} - v(x) + v(\tau(x)) + \frac{3}{2}v(\tau) \\
           &\le \frac{v(\Delta(\calC))}{4} \,,
  \end{align*}
  where we have used~\eqref{E:tau_bd} and $v(\tau) \le 0$.
\end{proof}

\begin{rk}
  When the residue characteristic is not~2, then we can easily show that
  $\bar{\beta}(\calC)$ is indeed always comparable to $v(\Delta(\calC))$.
  We can assume that $H = 0$ and write $F = c F_0$ with $F_0$ primitive.
  We consider the points of order~$2$ on~$J$. Such a point~$P$ is given by
  a factorization $F_0 = G_1 G_2$ with $G_1$ and~$G_2$
  primitive of degrees $2$ and~$4$, respectively. An explicit computation shows that
  \[ \eps(P) = 4 v(c) + 2 v(R(P))\,, \]
  where $R(P)$ denotes the resultant of $G_1$ and~$G_2$,
  and we have $4\mu(P) = \eps(P)$.
  Since $v(\Delta(\calC)) = v(\disc(F)) = 10 v(c) + v(\disc(F_0))$
  and $4 v(\disc(F_0))$ is the sum of the valuations of the $15$ resultants $R(P)$,
  we find that
  \begin{align*}
    \bar{\beta}(\calC) &\ge \frac{1}{4} \max_{O \neq P \in J[2]} \bigl(4 v(c) + 2 v(R(P))\bigr)
                        \ge v(c) + \frac{1}{30} \sum_{O \neq P \in J[2]} v(R(P)) \\
                       &= v(c) + \frac{2}{15} v(\disc(F_0))
                        \ge \frac{1}{10} v(\Delta(\calC)) \,.
  \end{align*}
  A similar statement should be true when the residue characteristic is~$2$.
\end{rk}

Recall that we denote
$\max \{\eps(P) : P \in J(\bar{k})\}$ by~$\bar{\gamma}(\calC)$.
\begin{cor} \label{C:EpsBound}
    Let $C$ be a smooth projective curve of genus~2 defined
    over a non-archimedean local field $k$, given by an integral Weierstrass
    model~$\calC$. Then we have
    \[
        \bar{\gamma}(\calC) \le v(\Delta(\calC))\,.
    \]
    If $H=0$ and $\Char(k) \ne 2$, then this can be improved to
    \[
    \bar{\gamma}(\calC) \le v(2^{-4}\Delta(\calC))\,.
    \]
\end{cor}
\begin{proof}
  The first inequality follows from~\ref{T:UpperBd} and from $\eps(P) = 4\mu(P) -
  \mu(2P)$.
  The second inequality is Theorem~6.1 of~\cite{StollH1}.
\end{proof}

\begin{ques}
If $\calC$ is a minimal Weierstrass model, does $\bar{\beta}(\calC)$ only depend
on the special fiber of~$\calC^{\min}$?
\end{ques}

Note that the corresponding statement holds for elliptic
curves~\cite{CrePriSik}. In our situation, however, there may be
several non-isomorphic minimal Weierstrass models, which complicates
the picture.

%=======================================================================================

\vfill\pagebreak

\section*{\large Part~III: Efficient Computation of Canonical Heights}

In this part we show how to compute the canonical height~$\hat{h}(P)$
efficiently for a point~$P$ over a number field, global function field
or more general field with a system of absolute values as in Section~\ref{S:genhts}.
We first explain how to compute the local
height correction functions. We use $\Mult(d)$ to denote the time needed to
multiply two $d$-bit integers.

\section{Computing $\mu$ at non-archimedean places} \label{algo1}

In this section, $k$ is a non-archimedean local field again, with valuation ring~$\O$,
uniformizer $\pi$, normalized valuation $v$ and residue class field $\frk$.
Let $\calC$ be an integral Weierstrass model for a genus~2 curve~$C$ over~$k$.
We make no assumptions on the reduction type of~$C$.
We already discussed a method for the computation of~$\mu(P)$ for a given
point~$P \in J(k)$ in Section~\ref{genhts}.
In this section, we provide an alternative fast algorithm and show that its running
time is $\ll (\log v(\Delta)) \Mult\bigl((\log v(\Delta)) v(\Delta) (\log \#\frk)\bigr)$,
where $\Delta = \Delta(\calC)$.

\begin{lemma} \label{L:fast algo}
 Assume that $M$ is a positive integer such that  $M\mu(P) \in \Z$.
  Further assume that $\max \{\eps(P) : P \in J(k)\} \le B$.
  Then
  \[ \mu(P) = \frac{1}{M} \Bigl\lceil M
                  \sum_{n=0}^{\lfloor \log(BM/3)/\log(4) \rfloor} 4^{-n-1} \eps(2^n
                  P)\Bigr\rceil \,.
  \]
\end{lemma}

\begin{proof}
  This follows from $M \mu(P) \in \Z$ and from
  \[ 0 \le M \sum_{n \ge m} 4^{-n-1} \eps(2^n P) \le \frac{B M}{3 \cdot 4^m}\, . \qedhere \]
\end{proof}

If we know that the reduction is nodal, then we get an upper bound~$B$ for~$\eps(P)$ and all possible
denominators of~$\mu(P)$ from the results of Section~\ref{formulas}. More generally, if we
know the smallest positive period~$N$ of the sequence~$(\mu(nP))_n$, then we can take
$M=N$ (respectively, $M=2N$) if $N$ is odd (respectively, even) by Corollary~\ref{C:Denominator}.
Also note that we can always take $B=v(\Delta)$ (or even $B = v(2^{-4} \Delta)$
if $\Char(k) \ne 2$ and the equation of the curve has $H = 0$), see Corollary~\ref{C:EpsBound}.

If we only know an upper bound for the denominator of~$\mu(P)$, then the following alternative
approach can be used. This is analogous to~\cite{MuellerStollEll}*{Lemma~4.2}.

\begin{lemma} \label{L:fast algo 2}
  Assume that $M \ge 2$ is an integer such that $M'\mu(P) \in \Z$ for some $0<M'\le M$.
  Assume in addition that $\max \{\eps(P) : P \in J(k)\} \le B$, and set
  \[ m = \Bigl\lfloor \frac{\log(BM^2/3)}{\log 4} \Bigr\rfloor \,. \]
  Then $\mu(P)$ is the unique fraction with denominator~$\le M$
  in the interval $[\mu_0, \mu_0 + 1/M^2]$, where
  \[ \mu_0 = \sum_{n=0}^{m} 4^{-n-1} \eps(2^n P) \,. \]
\end{lemma}

\begin{proof}
  Note that
  \[ \mu_0 \le \mu(P) \le \mu_0 + \sum_{n>m} 4^{-n-1} B < \mu_0 + 1/M^2 \,. \]
  But since $M \ge 2$, the interval $[\mu_0, \mu_0 + 1/M^2]$ contains at most
  one fraction with denominator bounded by~$M$; by assumption, $\mu(P)$ is such a fraction.
\end{proof}

In order to apply~Lemma~\ref{L:fast algo 2}, we now find a general upper bound~$M$
on the possible denominators of $\mu$.
Let $\calJ$ denote the N\'eron model of~$J$ over~$S=\Spec(\O)$
and write $\Phi$ for the component group of~$\calJ$.

\begin{prop}\label{P:MBound}
  Let $N$ denote the exponent of $\Phi(\bar{\frk})$ and let $P \in J(k)$. Then we have
 \[
   \mu(P) \in \frac{1}{2N} \Z\,.
 \]
 If $N$ is odd or if $C$ has a $k^{\nr}$-rational Weierstrass point, then we have
 \[
   \mu(P) \in \frac{1}{N} \Z\,.
 \]
\end{prop}

\begin{proof}
  Let $i \in\{1,\ldots,4\}$ be such that $\kappa_i(P) \ne 0$.
  Recall from Lemma~\ref{L:Uchida} that the function
  $\hat{\lambda}_i= \hat{\lambda}\circ\frac{\kappa}{\kappa_i}$
  is a N\'eron function with respect to the divisor~$D_i$.
  As $P \notin \supp D_i$, we find
  \[ \mu(P) \equiv\hat{\lambda}(x) \equiv \hat{\lambda}_i(P) \pmod{\Z} \]
  for any set of Kummer coordinates~$x$ for~$P$.
  It follows from the results of~\cite{Neron} and~\cite{LangFund}*{\S11.5} that
  \[
  \hat{\lambda}_i(P) \equiv j(D_i, (P) - (O)) \pmod{\Z}\,,
  \]
  where $j(\;,\;)$ denotes N\'eron's bilinear $j$-pairing, defined in~\cite{Neron}*{\S III.3}.

  By~\cite{Neron}*{Prop.~III.2}, the values of the $j$-pairing  lie in
  $\frac{1}{2N'}\Z$, where  $N' = \#\Phi(\bar{\frk})$
  It is easy to see that we can replace $N'$ by the exponent~$N$ in the proof
  of~\cite{Neron}*{Prop.~III.2}, so the first statement of the proposition follows.

  For the second statement, note that the $j$-pairing takes values in $\frac{1}{N}\Z$ if
  $N$ is odd, again by~\cite{Neron}*{Prop.~III.2} and its proof.
  If $C$ has a $k^{\nr}$-rational Weierstrass point
  $P_0$, then the divisor $D_i$ is linearly equivalent over $k^{\nr}$ to
  $2\Theta_{P_0}$, where $\Theta_{P_0}$ is the theta divisor with respect to $P_0$.
  The N\'eron model does not change under unramified extensions, and $\mu(P)
  \bmod{\Z}$ does not depend on the Weierstrass model of $C$ by Corollary~\ref{C:lambdatau}.
  Hence we can assume that $i=1$ and $D_1 = 2\Theta_{P_0}$, so the linearity of the $j$-pairing
  in the first variable proves the claim.
\end{proof}

\begin{rk}\label{R:exc_types}
  In the notation of Namikawa-Ueno~\cite{NamiUeno}, the only reduction types for which
  Proposition~\ref{P:MBound} does not show that $\mu(P) \in 1/N\Z$ (where $N$ is the
  exponent of~$\Phi(\bar{\frk})$), are
  $[2I\!I\!I-l]$ and $[2I\!I\!I^*-l]$ for $l\ge0$; $[2I_n^*-l]$ for $n,l \ge 0$; and
  $[2I_n-l]$ for $n>0$ even and $l\ge 0$.
  We have not found an example where $\mu(P) \notin 1/N\Z$.
\end{rk}

We can compute the group~$\Phi(\bar{\frk})$ in practice
using~\cite{BLR}*{\S9.6}.
For this we need to know the intersection matrix of the special fiber of a
regular model of~$C$ over~$S$.
This is implemented in~{\sf Magma}, but can be rather slow.
If the residue characteristic is not~2, then we can apply Liu's
algorithm~\cite{liuminimaux} to compute the reduction type and read off~$\Phi(\bar{\frk})$.

In general, an upper bound for the exponent of~$\Phi(\bar{\frk})$
suffices to apply Lemma~\ref{L:fast algo 2}.
We give a bound which only depends on the valuation of the discriminant
$\Delta = \Delta(\calC)$.

\begin{lemma}\label{L:tam_bound}
  The exponent of $\Phi(\bar{\frk})$ is bounded from above by
  \[
    M := \max\left\{2, \left\lfloor\frac{v(\Delta)^2}{3}\right\rfloor\right\}\,.
  \]
  Moreover, the denominator of $\mu(P)$ is bounded from above by $M$ for all $P \in J(k)$.
\end{lemma}

\begin{proof}
  This follows from a case-by-case analysis, using the
  list of groups $\Phi(\bar{\frk})$ from~\cite{liuminimaux}*{\S8} for all reduction types
  in~\cite{NamiUeno}, and Proposition~\ref{P:MBound}.
\end{proof}

\begin{rk} \label{R:den_bound}
  By going through all reduction types, it is possible to obtain better upper bounds
  for the denominator $M'$ of $\mu(P)$ from the Igusa invariants discussed in Section~\ref{igusa}.
  First note that if the special fiber of $\calC$ is non-reduced, then we have
  \begin{enumerate}[(i)]
    \item $M' \le 4$ if $v(\Delta) \le 12$;
    \item $M' \le \max\{12, v(\Delta) -15\}$ otherwise.
  \end{enumerate}
  Suppose that $\calC$ is reduced; then, by Proposition~\ref{P:IgusaInv},
  we can use the Igusa invariants of the special fiber
  to distinguish between the multiplicities of its singularities.
  \begin{enumerate}[\upshape (i)]
    \item If all points on the special fiber of $\mathcal{C}$ have multiplicity at most~2, then we can
    bound $M'$ using Proposition~\ref{P:IgusaSemistable} (i--iii) and
  Propositions~\ref{muchi},~\ref{muchi2},~\ref{muchi3}.
  \item If there is a point of multiplicity~3 on the special fiber, then we have
    \begin{enumerate}[$\bullet$]
      \item $M' \le \min\{6, v(\Delta)+1\}$ if $v(\Delta) \le 10$;
      \item $M' \le 12$, if $v(\Delta) \le 20$;
      \item $M' \le \left\lfloor \frac{(v(\Delta) - 12)^2}{4}\right\rfloor$ otherwise.
    \end{enumerate}
      \item If there is a point of multiplicity $\ge 4$ on the special fiber, then we have
    \begin{enumerate}[$\bullet$]
      \item $M' \le 3v(\Delta) - 10$ if $v(\Delta) \le 10$;
      \item $M'\le4v(\Delta) - 20$ if $v(\Delta) > 10$ and the model is minimal;
      \item $M'\le \left\lfloor\frac{(v(\Delta)-10)^2}{3}\right\rfloor$ if the model is not
        minimal.
    \end{enumerate}
  \end{enumerate}
\end{rk}

The results of this section lead to an efficient algorithm for the computation
of~$\mu(P)$, which is analogous to Algorithm~4.4 of~\cite{MuellerStollEll}.
We assume that the coefficients of $F$ and~$H$ and the coordinates of~$P$ are
given to sufficient $v$-adic precision (in practice, they will be given exactly
as elements of a number field or function field).

\begin{enumerate}[1.]
  \item If $\Char(k) \ne 2$ and $H=0$, set $B \colonequals v(2^{-4} \Delta)$.
        Otherwise, set $B \colonequals v(\Delta)$.
  \item Set $M \colonequals \max\left\{2, \lfloor v(\Delta)^2/3\rfloor\right\}$.
  \item Set $m \colonequals \lfloor \log(BM^2/3)/\log(4) \rfloor$.
  \item Set $\mu_0 \colonequals 0$. Let $x$ be normalized
        Kummer coordinates for~$P$ with $(m+1)B+1$ $v$-adic digits of precision.
  \item For $n \colonequals 0$ to~$m$ do:
        \begin{enumerate}[a.]
          \item Compute $x' \colonequals \delta(x)$ (to $(m+1)B+1$ $v$-adic digits of precision).
          \item If $v(x') = 0$, then return $\mu_0$.
          \item Set $\mu_0 \colonequals \mu_0 + 4^{-n-1} v(x')$.
          \item Set $x \colonequals \pi^{-v(x')} x'$
        \end{enumerate}
  \item Return the unique fraction with denominator at most~$M$
        in the interval between $\mu_0$ and~$\mu_0 + 1/M^2$.
\end{enumerate}
The fraction in the final step can be computed easily, for instance using continued fractions.

For the  complexity analysis in the following proposition, we assume that elements of~$\O$
are represented as truncated power series in~$\pi$, whose coefficients
are taken from a complete set of representatives for the residue classes.
Operations on these coefficients can be performed in time $\ll \Mult(\log \#\frk)$.

\begin{prop} \label{P:fast algo}
  The algorithm above computes $\mu(P)$. Its running time is
  \[ \ll (\log v(\Delta)) \Mult\bigl((\log v(\Delta)) v(\Delta) (\log \#\frk)\bigr) \]
  as $v(\Delta) \to \infty$, with an absolute implied constant.
\end{prop}

\begin{proof}
  The following proof is analogous to the proof of~\cite{MuellerStollEll}*{Prop.~4.5}.
  Corollary~\ref{C:EpsBound} shows that $B$ is a suitable upper bound for~$\eps$
  and Lemma~\ref{L:tam_bound} shows that $M$ is an upper bound for the denominator of~$\mu$.
  Because $M\ge 2$, the loop in step~5 computes the sum in Lemma~\ref{L:fast algo 2}.
  Note that when $v(x')=0$ in step~5b, then $\mu(P) = \mu_0$ by Theorem~\ref{T:mu_U}.
  At each duplication step, the precision loss is~$\eps(2^n P) \le B$, so that with our
  choice of starting precision, after the
  $m+1$~steps in the loop the resulting~$x$ still has at least one digit
  of precision. This proves the correctness of the algorithm.

  Clearly the running time of the algorithm is dominated by the running time of the loop in step~5.
  Step~5a consists of a fixed number of additions and multiplications of elements of $\O$
  which are given to a precision of $(m+1)B + 1$ digits.
  Because steps~5b--5d take negligible time compared to step~5a, each pass through the loop
  takes \[ \ll \Mult\bigl(((m+1)B + 1) (\log \#\frk)\bigr) \]
  operations, leading to a total running time that is
  \begin{align*}
    &\ll (m+1) \Mult\bigl(((m+1)B + 1) (\log \#\frk)\bigr) \\
    &\ll m \Mult(m B (\log \#\frk)) \\
    &\ll (\log v(\Delta)) \Mult\bigl((\log v(\Delta)) v(\Delta) (\log \#\frk)\bigr)
  \end{align*}
  as $v(\Delta) \to \infty$.
  Here we use that $B \ll v(\Delta)$ and $M \ll v(\Delta)^2$, so that $m \ll \log v(\Delta)$.
\end{proof}

\begin{rk}\label{R:better_bounds}
  In step~2, we can use Remark~\ref{R:den_bound} to compute a sharper upper bound
  for the denominator of~$\mu$.
  See also the discussion following Remark~\ref{R:exc_types}.
  Of course, if we want to find $\mu(P)$ for several points $P$, the quantities $M,\,
  B$ and $m$ only have to be computed once.
\end{rk}

\begin{rk}\label{R:DVF}
  We can compute $\mu(P)$ using the algorithm above in more general situations.
  Suppose that $k$ is any discretely valued field with valuation ring~$\O$
  and uniformizer~$\pi$.
  In that case, the sequence $(\mu(nP))_n$ might not have a finite period, so the method
  for the computation of~$\mu(P)$ discussed in Section~\ref{genhts} might not be
  applicable.
  However, Lemma~\ref{L:fast algo}, Lemma~\ref{L:fast algo 2},
  Proposition~\ref{P:MBound} and Lemma~\ref{L:tam_bound} remain valid.
  If $\Char(k) \ne 2$ and if $H=0$, then we have the upper bound
  $\eps(P) \le v(2^{-4}\Delta)$ (cf.~Remark~\ref{R:eps_DVF}), so the algorithm
  above can be used and Proposition~\ref{P:fast algo} remains valid as well,
  in the sense that the computation can be done using $\ll \log v(\Delta)$
  operations with elements of $\O/\pi^n\O$, where $n \ll v(\Delta) \log v(\Delta)$.
  In the remaining cases, we can compute an upper bound~$B$ on~$\eps$ as in
  Remark~\ref{R:eps_DVF}, and we can apply the algorithm with this choice of~$B$.
\end{rk}

%=======================================================================================

\section{Computing $\mu$ at archimedean places} \label{HCinf}

In this section, $k$ is an archimedean local field, so $k = \R$ or $k = \C$.
We assume that the curve~$C$ is given by a Weierstrass equation~$\calC$ with $H = 0$.
In the following, $\log_+ x = \max\{0, \log x\}$.

Let $x \in k^4$ be a set of Kummer coordinates.
Recall that
\[ \tilde\eps(x) = - [k : \R] \left(\log \|\delta(x)\|_\infty - 4 \log \|x\|_\infty\right) \]
and
\[ \tilde{\mu}(x) = \sum_{n=0}^\infty 4^{-n-1} \tilde{\eps}(\delta^{\circ n}(x)) \,. \]
We easily obtain
a lower bound for~$\tilde{\eps}$ using the standard estimate for $\|\delta(x)\|_\infty$.
Since the coefficients of the duplication polynomials~$\delta_j$ are universal polynomials
of degree at most~$4$ in the coefficients of~$F$, this gives
\[ -\tilde{\eps} \ll 1 + \log_+ \|F\|_\infty \,, \]
where $\|F\|_\infty$ is the maximum norm of the coefficient vector of~$F$.
We recall that
the method described in Section~7 of~\cite{StollH1}, leading to equation~(7.1) there,
provides an upper bound~$\tilde\gamma$ for $\tilde\eps$ that can be explicitly computed
for any given Weierstrass equation~$\calC$ of the curve (provided $H=0$). It is given by
\begin{align*}
  \tilde{\gamma} &=
  \log \max_i \left(\sum_{\{S,S'\}} |a_{i,\{S,S'\}}| \sqrt{\sum_{j=1}^4 |b_{\{S,S'\},j}|}\right)^2 \\
     &\le \log 400 + 2 \log \max_{i, \{S,S'\}} |a_{i,\{S,S'\}}|
                   + \log \max_{\{S,S'\}, j} |b_{\{S,S'\},j}|
\end{align*}
with certain numbers $a_{i,\{S,S'\}}$, $b_{\{S,S'\},j}$, where $i,j \in \{1,2,3,4\}$
and $\{S,S'\}$ runs through the ten partitions of the set of roots of~$F$ into two sets
of three. Using the formulas in~\cite{StollH1}*{\S10} and Mignotte's bound
(see for example~\cite{GG}*{Cor.~6.33}), we see that
\[ \log \max_{\{S,S'\}, j} |b_{\{S,S'\},j}| \ll 1 + \log_+ \|F\|_\infty \]
and
\[ \log \max_{i, \{S,S'\}} |a_{i,\{S,S'\}}|
     \ll 1 + \log_+ \|F\|_\infty + \log_+ \max_{\{S,S'\}} |R(S,S')|^{-1} \,,
\]
where $R(S,S')$ is the resultant of the two factors $G$, $G'$ of~$F$ corresponding to the
partition of the roots. Using Mignotte's bound again, we find that
\[ |R(S,S')|^{-1} = \frac{\sqrt{|\disc G|\,|\disc G'|}}{\sqrt{|\disc F|}}
                  \ll \|F\|_\infty^2 |\Delta(\calC)|^{-1/2} \,,
\]
leading finally to the estimate
\[ |\tilde{\eps}| \ll  1 + \log_+ \|F\|_\infty + \log_+ |\Delta(\calC)|^{-1}  \equalscolon s(F)\,. \]
If $|\tilde{\eps}(x)| \le \tilde{\eta}$ for all $x \in \KS_{\A}$, then we have
\[ \left|\sum_{n \ge N} 4^{-n-1} \tilde{\eps}(\delta^{\circ n}(x))\right|
     \le \frac{\tilde{\eta}}{3} 4^{-N} \,,
\]
so we need to sum the first
\[ N = \left\lceil \frac{d}{2} + \frac{\log (\tilde{\eta}/3)}{\log 4} \right\rceil
     \ll d + \log s(F)
\]
terms to obtain an accuracy of~$2^{-d}$. Comparing the largest term in any of
the~$\delta_j$ and the lower bound on~$\|\delta(x)\|_\infty$, we obtain a
bound~$\tilde\theta$ on the loss of relative precision (in terms of bits)
in the computation of~$\delta(x)$; we have
$\tilde\theta \ll s(F)$.
To achieve the desired precision at the end, we therefore need to compute
with an initial precision of
\[ d + 1 + N \tilde\theta \ll (d + \log s(F)) s(F) \]
bits. The time needed for each duplication is then
\[ \ll \Mult\bigl((d + \log s(F)) s(F)\bigr) \,. \]
A logarithm can be computed to $d$~bits of precision in time $\ll (\log d) \Mult(d)$
by one of several quadratically converging algorithms, see for
example~\cite{Borwein}*{Chapter~7}, so we obtain the following result.

\begin{prop} \label{P:arch}
  Given Kummer coordinates~$x$ of a point~$P$ in~$J(k)$ (or~$\KS(k)$) to sufficient
  precision, we can compute $\tilde\mu(P)$ to an accuracy of~$d$ bits in time
  \[ \ll \bigl(d + \log s(F)\bigr) (\log d) \Mult\bigl((d + \log s(F)) s(F)\bigr) \,, \]
  where
  \[ s(F) = 1 + \log_+ \|F\|_\infty + \log_+ |\Delta(\calC)|^{-1} \,. \]
\end{prop}

In the applications $k$ will be the completion of a number field at a real
or complex place. If the number field is $\Q$ and the given equation~$\calC$ of~$C$ is integral,
then $|\Delta(\calC)| \ge 1$ and we have $s(F) = 1 + \log \|F\|_\infty = 1 + h(F)$, where $h(F)$
denotes the (logarithmic) height of the coefficient vector of~$F$ as a point
in affine space. In general, we have the estimate
(denoting the value of~$s(F)$ for a place~$v$ by~$s_v(F)$)
\begin{align*}
  \sum_{v \mid \infty} s_v(F)
     &\le [K : \Q] + \sum_{v \mid \infty} \log_+ \|F\|_v
              + \sum_{v \mid \infty} \log_+ |\Delta(\calC)|_v^{-1} \\
     &\le [K : \Q] + h(F) + h(\Delta(\calC)) \ll h(F)
\end{align*}
for $h(F)$ large.
This implies that we can compute the infinite part of the height correction function
in time
\[ \ll \bigl(d + \log h(F)\bigr) (\log d) \Mult\bigl((d + \log h(F)) h(F)\bigr) \,, \]
which is polynomial in~$d$ and~$h(F)$.

%=======================================================================================

\section{Computing the canonical height of rational points} \label{S:cch}

The first algorithm for computing the canonical height on a genus~2 Jacobian over~$\Q$
was introduced by Flynn and Smart~\cite{FlynnSmart}.
It does not require any integer factorization, but can be impractical even for
simple examples, see the discussion in~\cite{StollH2}*{\S1}.
A more practical algorithm was introduced by the second author in~\cite{StollH2}; here the
local height correction functions are computed separately, so some integer factorization is
required. Uchida~\cite{UchiCano} later introduced a similar algorithm.
De Jong and the first author~\cite{deJongMueller} used division polynomials for a different
approach.

Building on the Arakelov-theoretic Hodge index theorem for arithmetic surfaces due to
Faltings and Hriljac, Holmes~\cite{Holmes} and the first author~\cite{MuellerArak}
independently developed algorithms for the computation of canonical heights of points on
Jacobians of hyperelliptic curves of arbitrary genus over global fields.
While these algorithms can be used to compute canonical heights for
genus as large as~10 (see~\cite{MuellerArak}*{Example~6.2}), they are much slower than the
algorithm from~\cite{StollH2} when the genus is~2.

In this section
we now combine the results of Sections \ref{algo1} and~\ref{HCinf} into an efficient
algorithm for computing the canonical height of a point on the Jacobian of
a curve of genus~2 over a global field~$K$.

When $K$ is a function field, then there are no archimedean places and factorization
is reasonably cheap. So in this case, the best approach seems to be to first
find the places~$v$ of~$K$ such that $\mu_v(P)$ is possibly non-zero (this includes
the places at which the given equation of the curve is non-integral) and then
compute the corrections~$\mu_v(P)$ for each place separately as in the algorithm
of Proposition~\ref{P:fast algo}, if necessary changing first to an integral
model and correcting for the transformation afterwards.
In fact this approach can be used whenever $K$ is a field with a set of absolute values that
satisfy the product formula, because the algorithm before Proposition~\ref{P:fast algo}
is applicable over any discretely valued field, see Remark~\ref{R:DVF}.
This includes function fields such as $\Q(t)$ and $\C(t)$.

If $K$ is a number field, then we compute the contribution from the archimedean
places as described in Section~\ref{HCinf}.
The finite part of our algorithm is analogous to our quasi-linear algorithm for the
computation of the finite part of the canonical height of a point on an elliptic
curve in~\cite{MuellerStollEll}; see Proposition~\ref{P:nofact} below.
For simplicity, we take $K$
to be~$\Q$ in the following. We write $\eps_p$ and~$\mu_p$ for the local height
correction functions
over~$\Q_p$ as given by Definition~\ref{defepsmu} and~$\tilde{\mu}_\infty$ for the
local height correction function over~$\R$ as defined in equation~\eqref{E:tildemu}.

We assume that our curve is given by a model $\calC \colon Y^2 = F(X,Z)$
with $F \in \Z[X,Z]$, and we set $\Delta = \Delta(\calC)$.
Our goal is to devise an algorithm for the computation of~$\hat{h}(P)$
that runs in time polynomial in $\log \|F\|_\infty$, $h(P)$ and the required precision~$d$
(measured in bits after the binary dot). We note that $h(P)$ can be computed in time
\[ {} \ll \log(h(P) + d) \Mult(h(P) + d)\,, \]
since it is just a logarithm. By Proposition~\ref{P:arch},
the height correction function~$\tilde{\mu}_\infty(P)$ can be computed in polynomial time.
So we only have to find an efficient algorithm for the computation of
the `finite part'
$\tilde{\mu}^{\text{f}}(P) \colonequals \sum_p \mu_p(P) \log p$ of the height correction.

Fix $P \in J(\Q)$.
We call a set $x$ of Kummer coordinates for $P$ {\em primitive} if $x \in \Z^4$
and $\gcd(x) = 1$.
We set $g_n = \gcd(\delta(x^{(n)}))$, where $x^{(n)}$ is a
primitive set of Kummer coordinates for~$2^n P$.
Then
\[ \tilde{\mu}^{\text{f}}(P) = \sum_{n=0}^\infty 4^{-n-1} \log g_n \,. \]
We also know by~\cite{StollH1}
that $g_n$ divides $D = |\Delta|/2^4 = 2^4 |\disc(F)|$, which implies
that $\log g_n \le \log D$ for all $n$. To achieve a precision of~$2^{-d}$,
it is therefore enough to take the sum up to
\[ n = m \colonequals \left\lfloor \frac{d}{2} + \log\frac{\log D}{3} \right\rfloor
     \ll d + \log\log D \ll d + \log\log \|F\|_\infty \,.
\]
Since at each duplication step, we have to divide by~$g_n$ to obtain primitive
coordinates again, it suffices to do the computation modulo~$D^{m+2}$. This leads
to the following algorithm.
\begin{enumerate}[1.]
  \item Let $D = |\Delta|/16$ and set
        $m \colonequals \lfloor d/2 + \log\log D - \log 3 \rfloor$.
  \item Let $x$ be primitive Kummer coordinates for~$P$.
  \item Set $\mu \colonequals 0$.
  \item For $n \colonequals 0$ to $m$ do:
        \begin{enumerate}[a.]
          \item Compute $x' \colonequals \delta(x) \bmod D^{m+2}$.
          \item Set $g_n \colonequals \gcd(D, \gcd(x'))$ and $x \colonequals x'/g_n$.
          \item Set $\mu \colonequals \mu + 4^{-n-1} \log g_n$ (to $d$ bits of precision).
        \end{enumerate}
  \item Return $\tilde{\mu}^{\text{f}}(P) \approx \mu$.
\end{enumerate}

\begin{prop}\label{P:mu_approx}
  This algorithm computes $\tilde{\mu}^{\text{\rm f}}(P)$ to $d$~bits of precision in time
  \[ \ll (d + \log\log D)\log (d + \log\log D)\Mult\bigl((d + \log\log D) \log D\bigr) + h(P) \,. \]
\end{prop}

\begin{proof}
  The discussion preceding the algorithm shows that it is correct.
  The duplication in step~4a  can be computed in time
  $ \ll \Mult((m+2) \log D) \ll  \Mult\bigl((d + \log\log D) \log D\bigr)$,
  while the gcd in step~4b can be computed in time
  \[ {} \ll \Mult((m+2) \log D)\log\bigl((m+2) \log D\bigr)
        \ll \log(d + \log \log D) \Mult\bigl((d + \log\log D) \log D\bigr) \,;
  \]
  the division is even faster, since $g_n$ is small.
  The computation of the logarithm takes time $\ll \log(d + \log D) \Mult(d + \log D)$;
  this is dominated by the time for computing the gcd. This gives a time complexity of
  \[ {} \ll (d + \log\log D) \log(d + \log\log D) \Mult\bigl((d + \log\log D) \log D\bigr) + h(P) \,, \]
  where the last term comes from processing the input~$x$.
\end{proof}

Note that $\log D \ll \log \|F\|_\infty$, so this bound is similar to (and even
better by a factor of~$\log d$ than) the complexity for computing $\tilde{\mu}_\infty(P)$.

\begin{rk} \label{R:nf1}
  We note that an alternative way to proceed is to compute
  $x' = \delta^{\circ(m+1)}(x)$ mod~$D^{m+2}$ (without dividing out gcd's in between)
  and then use $\mu = 4^{-m-1} \log \gcd(x')$. The advantage of the algorithm
  above is that we can actually work mod~$D^{m+2-n}$, which makes the computation
  more efficient. The advantage of the alternative is that it can also be used
  when working over a number field with non-trivial class group (replacing
  $\log \gcd(x')$ by the logarithm of the ideal norm of the ideal generated by~$x'$).
  The resulting complexity is similar, with the implied constant depending
  on the base field.
\end{rk}

We now show that we can in fact do quite a bit better than this,
by using the strategy already employed in~\cite{MuellerStollEll}.
Note that $\tilde{\mu}^{\text{f}}(P)$ is a rational linear combination
of logarithms of positive integers. We can compute
such a representation exactly and efficiently by the following algorithm.
We again assume that $x$ is a set of primitive Kummer coordinates for~$P$.

\begin{enumerate}[1.]
  \item Set $x' \colonequals \delta(x)$, $g_0 \colonequals \gcd(x')$ and $x \colonequals x'/g_0$.
  \item Set $D \colonequals \gcd(2^4 \disc(F), g_0^\infty)$ and
        $B \colonequals \lfloor \log D/\log 2 \rfloor$.
  \item If $B \le 1$, return~0. \\
        Otherwise, set $M \colonequals \max\{2,  \lfloor (B+4)^2/3 \rfloor\}$
        and $m \colonequals \lfloor \log(B^3M^2/3)/\log 4 \rfloor$.
  \item For $n \colonequals 1$ to $m$ do:
        \begin{enumerate}[a.]
          \item Compute $x' \colonequals \delta(x) \bmod D^{m+1} g_0$.
          \item Set $g_n \colonequals \gcd(D,\gcd(x'))$ and $x \colonequals x'/g_n$.
        \end{enumerate}
      \item Using the algorithm in~\cite{dcba2} (or in~\cite{dcba}), compute
        a sequence $(q_1, \ldots, q_r)$ of pairwise coprime positive integers
        such that each~$g_n$ (for $n = 0, \ldots, m$) is a product of powers
        of the~$q_i$: $g_n = \prod_{i=1}^r q_i^{e_{i,n}}$.
  \item For $i \colonequals 1$ to $r$ do:
        \begin{enumerate}[a.]
          \item Compute $a \colonequals \sum_{n=0}^{m} 4^{-n-1} e_{i,n}$.
          \item Let $\mu_i$ be the simplest fraction between $a$ and $a+1/(B^2M^2)$.
        \end{enumerate}
  \item Return $\sum_{i=1}^r \mu_i \log q_i$ (a formal linear combination of logarithms).
\end{enumerate}

\begin{prop} \label{P:nofact}
  The preceding algorithm computes $\tilde{\mu}^{\text{\rm f}}(P)$ in time
  \[ {} \ll (\log\log D)^2 \Mult\bigl((\log\log D) (\log D)\bigr) + \Mult(h(P))(\log h(P)) \,. \]
\end{prop}

Note that $D \le |\Delta|/16$ and $\log D \ll \log \|F\|_\infty$.

\begin{proof}
  If $B\le 1$ in step~3, then we either have $g_0=1$ and $\tilde{\mu}^{\text{\rm f}}(P) = 0$, or we
  have $D \in \{2,3\}$.
  In the latter case, $g_0$ is a power of $p = 2$ or~$3$
  and $v_p(\Delta) = 1$, which would imply that $\eps_p(P) = 0$
  by~\cite{StollH2}*{Prop.~5.2}, so $g_0 = 1$, and we get a contradiction.

  If a prime $p$ does not divide $g_0$, then $\eps_p(P) = 0$, implying $\mu_p(P) = 0$.
  Suppose now that $p$ divides $g_0$; then we have $v_p(D) \le B$ and $v_p(\Delta) \le B+4$,
  so $B$, $M$ and~$m$ are suitable values for Lemma~\ref{L:fast algo 2}.
  We have $v_p(g_n) = \eps_p(2^n P)$ for all $n \le m$, because
  $p^{(m+1)v_p(D) + 1}\,|\,D^{m+1} g_0$ (compare the proof of Proposition~\ref{P:fast algo}).
  All the~$g_n$ are power products of the~$q_i$, so there will be exactly
  one~$i = i(p) \in \{1,\ldots,r\}$ such that $p \mid q_{i(p)}$.
  Setting  $b_p = v_p(q_{i(p)})$ and $a  = \sum_{n=0}^m 4^{-n-1} e_{i(p),n}$, we have
  \[ \sum_{n=0}^m 4^{-n-1} \eps_p(2^n P)
       = \sum_{n=0}^m 4^{-n-1} v_p(g_n)
       %= b_p \sum_{n=0}^m 4^{-n-1} e_{i(p),n}
       = b_p a \,,
  \]
  implying
  \[ \mu_p(P) = \sum_{n=0}^\infty 4^{-n-1} \eps_p(2^n P)
              = b_p a + \sum_{n=m+1}^\infty 4^{-n-1} \eps_p(2^n P) \,.
  \]
  Here the last sum is in $[0, 1/(B^2M^2)]$ by the definition of~$m$
  (compare the proof of Lemma~\ref{L:fast algo 2}).
  Therefore
  \[a \le \mu_p(P)/b_p \le a + 1/(b_p B^2M^2) \le a + 1/(B^2M^2)\,.\]
  Since the denominator of~$\mu_p(P)$ is at most~$M$ and since we have $b_p\le v_p(D) \le B$, the
  denominator of~$\mu_p(P)/b_p$ is at most~$BM$.
  Hence $\mu_p(P)/b_p$ is the unique fraction in $[a, a + 1/(B^2M^2)]$
  with denominator bounded by~$BM$, so $\mu_p(P)/b_p = \mu_{i(p)}$ by Step~6b.
  Now
  \[ \sum_p \mu_p(P) \log p
       = \sum_p \mu_{i(p)} b_p \log p
       = \sum_{i=1}^r \mu_i \sum_{p \mid q_i} b_p \log p
       = \sum_{i=1}^r \mu_i \log q_i \,,
  \]
  so the algorithm is correct.

  The complexity analysis is as in the proof of~\cite{MuellerStollEll}*{Prop.~6.1}.
  Namely, the computations in step~1 can be done in time $\ll \Mult(h(P))\log h(P)$.
  The computations in steps 2 and~3 take negligible time.
  Each pass through the loop in step~4 takes time
  $\ll \log\bigl((m+2) \log D\bigr)\Mult\bigl((m+2) \log D\bigr)$,
  so the total time for step~4 is
  \[ {} \ll m \Mult(m \log D) \log(m\log D) \ll (\log\log D)^2 \Mult((\log\log D) (\log D))\,, \]
  because
  $m \ll \log\log D$. The coprime factorization algorithm in~\cite{dcba2} (or in~\cite{dcba})
  computes suitable~$q_i$ for a pair
  $(a, b)$ of positive integers in time $\ll (\log ab)(\log\log ab)^{2}$.
  We iterate this algorithm,
  applying it first to $g_0$ and~$g_1$, then to each of the resulting $q_i$ and~$g_2$,
  and so on. There are always
  $\ll \log D$ terms in the sequence of~$q_i$'s and we have $g_n\le D$ for all $n$.
  Hence step~5 takes time
  $\ll \log D (\log\log D)^3$. Because this is dominated by the time for the loop and
  because the remaining steps take negligible time, the result follows.
\end{proof}

Note that the complexity of the algorithm above is quasi-linear in $\log D$ and~$h(P)$.
In practice, the efficiency of this approach can be improved somewhat:

\begin{enumerate}[$\bullet$]
  \item We can split off the contributions of all sufficiently small primes~$p$ by choosing a
        suitable bound~$T$ and  trial factoring $\Delta$ up to~$T$; the corresponding~$\mu_p$
        can then be computed using the algorithm of Proposition~\ref{P:fast algo}; see also
        Remark~\ref{R:better_bounds}.
        In step~3, we can then set $B \colonequals \lfloor \log D'/\log T \rfloor$,
        where $D'$ is the unfactored part of~$D$, and replace $B+4$ by~$B$
        in the definition of~$M$.
        If the coefficients of $F$ are sufficiently large, then this trial division
        can become quite expensive (even for small values of $T$). So when $h(F)$ is large, it is
        usually preferable to avoid trial division altogether.
  \item We can update the~$q_i$ after each pass through the
        loop in step~4 using the new~$g_n$; we can also do the computation in step~4a modulo suitable
        powers of the~$q_i$ instead of modulo~$D^{m+1} g_0$.
        Moreover, it is possible to use separate values of $B$, $M$ and~$m$ for
        each~$q_i$; these will usually be smaller than the one computed in step 2 and 3.
        In this way, we can integrate steps 4, 5 and~6 into one loop.
\end{enumerate}

\begin{rk}
  Over a more general number field~$K$ in place of~$\Q$ the algorithm as stated does not
  quite work, since we cannot always divide out greatest common divisors.
  In this case we first compute $x^{(1)} = \delta(x)$ and the ideal $g_0$
  generated by~$D$ and the entries of~$x^{(1)}$. Then we compute
  $x^{(2)} = \delta(x^{(1)})$, \dots, $x^{(m+1)} = \delta(x^{(m)})$ modulo
  the ideal $D^{m+1} g_0$. Let $G_j$ be the ideal generated by the entries
  of~$x^{(j)}$ and~$D^{m+1}$ and set
  \[ g_1 = g_0^{-4} G_2, \quad g_2 = G_2^{-4} G_3, \quad g_3 = G_3^{-4} G_4, \quad
     \ldots, \quad g_m = G_m^{-4} G_{m+1} \,.
  \]
  The coprime factorization algorithms in~\cite{dcba2} and~\cite{dcba} also
  work for ideals. In the final result, $\log q_i$ has to be replaced by
  $\log N(q_i)$, where $N(q_i)$ is the norm of the ideal~$q_i$.
  This should result in a complexity similar to that over~$\Q$ (with the
  implied constant depending on~$K$), or at least one that is dominated
  by the complexity of computing the naive height and the contributions
  from the archimedean places. Unfortunately, no complexity
  analysis for standard operations with ideals in number fields seems to be
  available in the literature; this prevents us from making a precise statement.
  Alternatively, we can take the approach described in Remark~\ref{R:nf1}.
\end{rk}

Combining this with the results for archimedean places, we obtain an efficient
algorithm for computing the canonical height~$\hat{h}(P)$ of a point $P \in J(\Q)$.
As mentioned above, we expect a similar result to hold for any number field~$K$
in place of~$\Q$, with the implied constant depending on~$K$.

\begin{thm}\label{T:PolyAlgo}
  Let $C$ be given by the model $Y^2 = F(X,Z)$ with $F \in \Z[X,Z]$
  and let $P \in J(\Q)$ be given by primitive Kummer coordinates~$x$ (i.e.,
  the coordinates are coprime integers).
  We can compute $\hat{h}(P)$ to $d$~bits of precision in time
  \begin{align*}
    &\ll \log(d + h(P)) \Mult(d + h(P))\\
    &\qquad{}  + (d + \log\log \|F\|_\infty) (\log d + \log\log \|F\|_\infty)
                \Mult\bigl((d + \log\log \|F\|_\infty) \log \|F\|_\infty\bigr) \,.
  \end{align*}
\end{thm}

\begin{proof}
  The first term comes from computing~$h(P)$.
  The second term dominates both the complexity bound
  for~$\tilde{\mu}_\infty(P)$ from Proposition~\ref{P:arch}
  and the complexity of computing~$\tilde{\mu}^{\text{f}}(P)$ using the algorithm of
  Proposition~\ref{P:nofact}, since we have $D \le |\Delta|/16$
  and $\log D \ll \log \|F\|_\infty$.
  The time for the numerical evaluation of the logarithms~$\log q_i$
  to $d$~bits of precision is also dominated by this term.
\end{proof}

Note that the complexity is quasi-linear in $\log \|F\|_\infty$ and in~$h(P)$,
and quasi-quadratic in~$d$. The latter is caused by the (only) linear convergence
of the computation of~$\tilde{\mu}_\infty(P)$.
For elliptic curves one can use a quadratically convergent algorithm due to
Bost and Mestre~\cite{BostMestre}, see also~\cite{MuellerStollEll}; such an algorithm in the
genus~2 case would lead to a complexity that is quasi-linear in~$d$ as well.

In Section~\ref{Ex:canht} below we illustrate the efficiency of our algorithm
by applying it to a family of curves and points with the property that the number~$g_0$
above is large, so that the previously known algorithms have problems factoring it.

%=======================================================================================

\section{Examples} \label{Ex:canht}

We have implemented our algorithm using the computer algebra system {\sf
Magma}~\cite{Magma}.
For the factorization into coprimes we have implemented
a simple quadratic algorithm   due to Buchmann
and Lenstra~\cite{BuchmannLenstra}*{Prop.~6.5} instead of the quasi-linear, but more
complicated, algorithms of~\cite{dcba2} or~\cite{dcba}

Since the estimates for the required precision in the computation of the archimedean
contribution as given in Section~\ref{HCinf} are too wasteful in practice,
we instead compute this contribution repeatedly using a geometrically increasing
sequence of digits of precision until the results agree up to the desired
number of bits.

We now compare our implementation with {\sf Magma's} built-in {\tt CanonicalHeight}
(version~2.21-2), which is based on~\cite{FlynnSmart} and the second
author's paper~\cite{StollH2}, for a family of genus~2 curves.
In {\tt CanonicalHeight}, the duplication on the Kummer surface is done using
arithmetic over $\Q$, making the  implementation slow when points with large
coordinates show up during the computation.
No factorization of the discriminant is required.
However, to find a set of primes such that $\mu_p(P) \ne0$ for every prime $p$ not in the set, {\tt CanonicalHeight}
factors the integer $\gcd(\delta(x))$, where $x$ are primitive Kummer coordinates for $P$.

\begin{ex}\label{E:ex1}
  For an integer $a  \ne 0$, consider the curve $C_a$ of genus~2 defined by the integral Weierstrass model
  \[ y^2 = x^5 + a^2x + a^2 \,. \]
  Let $J_a$ denote the Jacobian of $C_a$. Then the point $P =[ ((0,\,a)) - (\infty)] \in J_a(\Q)$
  is non-torsion. A set of primitive Kummer coordinates
  is given by $x = (0,1,0,0)$ and we have $\delta(x) = (4a^2,0,0,a^4)$.
  Hence {\tt CanonicalHeight} needs to factor $a^2$.

  We choose this family of curves because (a) there is an obvious rational point~$P$
  on the Jacobian that is generically non-torsion and (b) $\gcd(\delta(x))$ involves a
  large integer, where $x$ is a set of primitive integral Kummer coordinates for~$P$.
  For a random sextic polynomial in~$\Z[x]$, very likely
  the discriminant will have a large square-free part, and so $\gcd(\delta(x))$ will
  be fairly small. Of course, the advantages of our algorithm show most clearly when
  $\gcd(\delta(x))$ is too large to be factored quickly.

Consider $a =$ {\tiny
5807658604988570942160367122286824505787920190639678196072209904446815339845301407936102
\\37063603282}, with partial factorization
$2\cdot7\cdot643\cdot804743\cdot a'$, where $a'$ has~89 decimal digits, and its smallest
prime factor has~34 decimal digits.
Our implementation computes $\hat{h}(P)$ in 0.51~seconds, whereas {\sf Magma's} {\tt
CanonicalHeight} needs about 15~minutes.

Next, we look at
$a=$ {\tiny 200403772956059488950289789507853617719701760528626768445669337185652379002740\
2225238543540575431528468305556200069359999066088091821746622820780762863572550314577271857779581968920}\,.
\\This factors as
$a = 2^3\cdot5\cdot 17\cdot a'$, where $a'$ has~178 decimal digits and no prime divisor with less
than 50~decimal digits.
Here, our implementation took 1.04~seconds to compute $\hat{h}(P)$, whereas {\tt Magma}
did not terminate in 8 weeks.

For $a = p\cdot q$, where $p$ (respectively, $q$) is the smallest prime larger than $10^{200}$
(respectively, $10^{250}$), the canonical height of $P$ was computed in 5.87~seconds using our
implementation.

For the computations in these examples, we used a single core Xeon CPU E7-8837 having
2.67GHz.
All heights were computed to 30~decimal digits of precision.
\end{ex}

We conclude this part with an example over the rational function field~$\Q(t)$.

\begin{ex}
  Consider the curve $C/\Q(t)$ given by the equation
  \begin{align*}
    y^2 = x^6 &- 2 t (t+1) x^5 + (t+1) (t^3-5t^2+4t-2) x^4
                     + 2 t (t+1)^2 (3t^2+1) x^3 \\
                       &- (t+1) (3t^4-2t^2+4t-1) x^2
                           - 4 t^2 (t+1)^3 (t^2+2t-1) x + 4 t^4 (t+1)^4\,.
  \end{align*}
  It has the points
  \[ P_1 = (1 : 1 : 0), \quad P_2 = \bigl(0, -2 t^2 (t+1)^2\bigr), \quad
     P_3 = \bigl(t+1, 2 t (t-1) (t+1)\bigr)
  \]
  (and also points with $x$-coordinate $t(t+1)$ and a Weierstrass point $(-t-1, 0)$).
  Let $Q = [(P_1) - 2 (P_2) + (P_3)] \in J(\Q(t))$.
  Its image on the Kummer surface has coordinates
  \[ (1 : -t+1 : -2 t^2(t+1) : 0) \,. \]
  Applying the duplication polynomials and looking at the gcd of the result,
  we see that we have to compute the height correction functions at the places
  given by $t = 0$, $t = 1$ and $t = -1$. We also have to consider the place
  at infinity, since our model of~$C$ is not integral there.
  We use the algorithms of Section~\ref{algo1}. Consider the place $t = 0$.
  From the valuations of the Igusa invariants (see Section~\ref{igusa}) we can
  deduce that the reduction type is~$[I_{7-3-2}]$, which gives us $M = 41$
  for the exponent of the component group and a bound $B = 10$ for~$\eps$.
  We follow Lemma~\ref{L:fast algo} and compute
  \[ \mu_0(Q) = \frac{1}{41} \Bigl\lceil 41 \sum_{n=0}^3 4^{-n-1} \eps_0(2^n Q) \Bigr\rceil
              = \frac{1}{41} \Bigl\lceil 41 \Bigl(\frac{8}{4} + \frac{4}{4^2}
                                                  + \frac{7}{4^3} + \frac{6}{4^4}\Bigr) \Bigr\rceil
              = \frac{98}{41} \,.
  \]
  At $t = 1$, the model is not stably minimal. We can deduce from the Igusa
  invariants that there is a stably minimal model over an extension of
  ramification index~$4$, which has reduction type~$[I_{12-2-2}]$. This shows
  that the denominator of~$\mu_1$ is divisible by $4 \cdot 26 = 104$.
  With $M = 104$ and~$B = 9$ we get $m = 4$ in Lemma~\ref{L:fast algo}; we obtain
  \[ \mu_1(Q) = \frac{1}{104} \Bigl\lceil 104 \sum_{n=0}^4 4^{-n-1} \eps_1(2^n Q) \Bigr\rceil
              = \frac{1}{104} \Bigl\lceil 104 \Bigl(\frac{4}{4} + \frac{4}{4^2}
                                                  + \frac{3}{4^3} + \frac{2}{4^4}
                                                  + \frac{2}{4^5}\Bigr) \Bigr\rceil
              = \frac{17}{13} \,.
  \]
  At $t = -1$, the situation is similar. There is a stably minimal model
  over an extension with ramification index~$4$ again, which has reduction
  type~$[I_{20-0-0}]$. This leads to $M = 4 \cdot 20 = 80$ and $B = 20$, so $m = 4$, and
  \[ \mu_{-1}(Q) = \frac{1}{80} \Bigl\lceil 80 \sum_{n=0}^4 4^{-n-1} \eps_{-1}(2^n Q) \Bigr\rceil
                 = \frac{1}{80} \Bigl\lceil 80 \Bigl(\frac{7}{4} + \frac{10}{4^2} + \frac{8}{4^3}
                                                  + \frac{10}{4^4} + \frac{8}{4^5}\Bigr) \Bigr\rceil
                 = \frac{51}{20} \,.
  \]
  Finally, at the infinite place, there is a stably minimal integral model
  over an extension with ramification degree~$2$, which has reduction type~$[I_{8-0-0}]$.
  In a similar way as for $t = -1$ and taking into account a shift of~$-8$
  coming from making the model integral, we obtain $\mu_\infty(Q) = 19/4 - 8 = -13/4$.
  This results in
  \[ \hat{h}(Q) = h(Q) - \mu_0(Q) - \mu_1(Q) - \mu_{-1}(Q) - \mu_\infty(Q)
                = 3 - \frac{98}{41} - \frac{17}{13} - \frac{51}{20} + \frac{13}{4}
                = \frac{11}{5330} \,.
  \]
  To our best knowledge, the point~$Q$ is the point of smallest known nonzero canonical
  height on the Jacobian of a curve of genus~$2$ over~$\Q(t)$.
  The curve was found by Andreas K\"uhn (a student of the second author) in the
  course of a systematic search for curves with many points mapping into a
  subgroup of rank~$1$ in the Jacobian.
\end{ex}

%=======================================================================================

\vfill\pagebreak

\section*{\large Part~IV: Efficient Search for Points With Bounded Canonical Height}

\section{Bounding the height difference at archimedean places} \label{S:htdiffarch}

We now describe two approaches for getting a better upper bound~$\tilde{\beta}$ on~$\tilde{\mu}$
than the one coming from the bound on~$\tilde{\eps}$ given in~\cite{StollH1}*{Equation~(7.1)},
when $k$ is an archimedean local field and $C/k$ is a smooth projective curve of genus~2,
given by a Weierstrass equation $Y^2 = F(X,Z)$ in $\BP_K(1,3,1)$.

We write $\|x\|_\infty = \max\{|x_1|, |x_2|, |x_3|, |x_4|\}$ for the maximum norm.

%--------------------------------------------------------------------------

\subsection{Bounding $\tilde{\eps}$ closely} \label{S:htdiffarch1} \strut

For the first approach we assume that $k = \R$. We describe how to approximate
$\max \{\tilde\eps(P) : P \in J(\R)\}$ to any desired accuracy, which gives us
an essentially optimal bound~$\tilde{\gamma}$. Recall that
\[ \tilde\eps(P)
     = -\log\frac{\max\{|\delta_1(x_1,x_2,x_3,x_4)|, \ldots, |\delta_4(x_1,x_2,x_3,x_4)|\}}%
                 {\max\{|x_1|,|x_2|,|x_3|,|x_4|\}^4} \,,
\]
where $(x_1 : x_2 : x_3 : x_4)$ is the image of $P \in J(\R)$ on the Kummer surface.
We can normalize the Kummer coordinates in such a way that $\|x\|_\infty = 1$ and
one of the coordinates is~$1$. We then have to minimize $\max\{|\delta_1|,\ldots,|\delta_4|\}$
over four three-dimensional unit cubes, restricted to the points on the Kummer surface
that are in the image of~$J(\R)$. This means that the relevant points satisfy the
equation defining the Kummer surface and in addition the value of (at least) one of
four further auxiliary polynomials is positive. (In general, the values of these polynomials
are squares if the point comes from the Jacobian, and the converse holds for any one
of the polynomials when its value is non-zero. One can choose four such polynomials in such
a way that they do not vanish simultaneously on the Kummer surface.)

The idea is now to successively subdivide the given cubes. For each small cube, we
check if it may contain points in the image of~$J(\R)$, by evaluating the various
polynomials at the center of the cube and bounding the gradient on the cube. If it
can be shown that the defining equation cannot vanish on the cube or that one of the
auxiliary polynomials takes only negative values on the cube, then the cube can be
discarded. Otherwise, we find upper and lower estimates for $\max\{|\delta_1|,\ldots,|\delta_4|\}$
in a similar way. If the lower bound is larger than our current best upper bound for the
minimum, the cube can also be discarded. (At the beginning, we have a trivial upper
bound of~$1$ for the minimum, coming from the origin.) Otherwise, we keep it and subdivide
it further. We continue until the difference of the upper and lower bounds for $\tilde\eps$
on the cube with the smallest lower bound for $\max\{|\delta_1|,\ldots,|\delta_4|\}$
becomes smaller than a specified tolerance. The upper bound for~$\tilde\eps$
on that cube is then our bound~$\tilde\gamma$, and we take (as before)
$\tilde\beta = \tilde\gamma/3$.

We have implemented this approach in {\sf Magma}~\cite{Magma}. After a considerable amount of fine-tuning,
our implementation usually takes a few seconds to produce the required bound. In many
cases the new bound, which is essentially optimal as a bound on~$\tilde\eps$,
is considerably better than the bound of~\cite{StollH1}*{(7.1)}, but there are also cases
for which it turns out that the old bound is actually pretty good.

We used the following tricks to get the implementation reasonably fast.
\begin{enumerate}[$\bullet$]
  \item We keep the polynomials shifted and rescaled so that the cube under consideration
        is $[-1,1]^3$.
  \item The shifting and scaling is done using linear algebra (working with vectors of coefficients
        and matrices) and not using polynomial arithmetic.
  \item The coordinates of the centers and vertices of all cubes are dyadic fractions.
        We scale everything (by $2^4 = 16$ at each subdivision step --- note that the
        polynomials involved are of degree~$4$) so that we can compute with integers instead.
\end{enumerate}

%--------------------------------------------------------------------------

\subsection{Iterating Stoll's bound} \label{S:htdiffarch2} \strut

We now describe a different approach that also works for complex places.
Instead of trying to get an optimal bound
on~$\tilde{\eps}$, we aim at a bound on~$\tilde{\mu}$ by iterating
the bound obtained from equation~(7.1) in~\cite{StollH1}. We recall how this bound
was obtained. There is an elementary abelian group scheme~$G$ of order~$32$ that
maps onto~$J[2]$ and acts on the space of quadratic forms in the coordinates of
the~$\BP^3$ containing the Kummer surface. This representation splits into a
direct sum of ten one-dimensional representations that correspond to the ten
partitions $\{S, S'\}$ of the set of ramification points of the double cover $C \to \BP^1$
into two sets of three. We write $y_{\{S,S'\}}$ for suitably normalized generators
of these eigenspaces (\cite{StollH1} gives explicit formulas in the case $H = 0$).
We can then express the squares $x_i^2$ as linear combinations of these quadratic forms:
\[ x_i^2 = \sum_{\{S,S'\}} a_{i,\{S,S'\}} y_{\{S,S'\}}(x) \]
for certain complex numbers $a_{i,\{S,S'\}}$ that can be explicitly determined.
On the other hand, $y_{\{S,S'\}}^2$ is a quartic form invariant under the action
of~$J[2]$ (the representation of~$G$ on quartic forms descends to a representation of~$J[2]$)
and is therefore a linear combination of the duplication polynomials~$\delta_j$
and the quartic defining the Kummer surface. So there are complex numbers $b_{\{S,S'\},j}$
that can also be explicitly determined such that
\[ y_{\{S,S'\}}(x)^2 = \sum_{j=1}^4 b_{\{S,S'\},j} \delta_j(x) \]
if $x$ is a set of Kummer coordinates.
Taking absolute values and using the triangle inequality, we obtain
  \[
  |x_i|^4
    \le \left(\sum_{\{S,S'\}} |a_{i,\{S,S'\}}| |y_{\{S,S'\}}(x)|\right)^2
    \le \left(\sum_{\{S,S'\}} |a_{i,\{S,S'\}}|
                               \sqrt{\sum_{j=1}^4 |b_{\{S,S'\},j}| |\delta_j(x)|}\right)^2
                             \]
for all $(x_1 : x_2 : x_3 : x_4) \in \KS(\C)$. This gives a bound
for~$\tilde{\eps}$ in terms of the $a_{i,\{S,S'\}}$ and~$b_{\{S,S'\},j}$
as in equation~(7.1) of~\cite{StollH1}.

We refine this as follows. Define a function
\[ \varphi \colon \R_{\ge 0}^4 \To \R_{\ge 0}^4, \quad
    (d_1,d_2,d_3,d_4) \longmapsto
       \left(\sqrt{\sum_{\{S,S'\}} |a_{i,\{S,S'\}}|
                        \sqrt{\sum_{j=1}^4 |b_{\{S,S'\},j}| d_j}}\right)_{\!1 \le i \le 4} \!.
\]

\begin{lemma} \label{L:arch1}
  Define a sequence $(b_n)_n$ in~$\R_{\ge 0}^4$ by
  \[ b_0 = (1, 1, 1, 1) \qquad\text{and}\qquad b_{n+1} = \varphi(b_n) \,. \]
  Then $(b_n)$ converges to a limit~$b$ and we have
  \[ \tilde{\mu}(P) \le \frac{4^N}{4^N-1} \log \|b_{N}\|_\infty  \]
  for all $N \ge 1$ and all $P \in J(\C)$. In particular,
  $\sup \tilde{\mu}(J(\C)) \le \log \|b\|_\infty$.
\end{lemma}

\begin{proof}
  By our previous considerations, it is clear that $|\delta_j(x)| \le d_j$
  for all~$j$ implies $|x_i| \le \varphi_i(d_1, d_2, d_3, d_4)$ for all~$i$.
  We deduce by induction on~$N$ that
  \[ \log \|x\|_\infty
       \le \log \|b_{N}\|_\infty + 4^{-N} \log \|\delta^{\circ N}(x)\|_\infty
  \]
  for all~$N \ge 1$. Writing
  \[ \tilde{\mu}(P)
        = -\sum_{m=0}^\infty 4^{-mN}
              \bigl(\log \|\kappa(2^{mN} P)\|_\infty
                    - 4^{-N} \log \|\delta^{\circ N}(\kappa(2^{mN} P))\|_\infty\bigr) \,,
  \]
  we obtain an upper bound of $\log \|b_N\|_\infty$
  for each of the terms in parentheses, which gives the desired bound.

  To see that $(b_n)$ converges, we consider
  $\Phi(x) = \bigl(\log \varphi_i(\exp(x_1), \ldots, \exp(x_4))\bigr)_{1 \le i \le 4}$.
  It is easy to see that the partial derivatives $\frac{\partial \Phi_i}{\partial x_j}$
  are positive and that for each~$i$, summing them over~$j$ gives $\frac{1}{4}$.
  (This comes from the fact that $\varphi_i$ is homogeneous of degree~$\frac{1}{4}$.)
  This implies that $\|\Phi(x') - \Phi(x)\|_\infty \le \frac{1}{4} \|x' - x\|_\infty$,
  so that $\Phi$ is contracting with contraction factor $\le \frac{1}{4}$.
  The Banach Fixed Point Theorem then guarantees the existence of a unique fixed point of~$\Phi$,
  which every iteration sequence converges to. This implies the corresponding statement
  for~$\varphi$.
\end{proof}

If we are dealing with a real place, then we may gain a little bit more
by making use of the fact that the~$\delta_j(x)$ are real, while some
of the coefficients~$b_{\{S,'S\},j}$ may be genuinely complex. This can lead
to a better bound on~$|y_{\{S,S'\}}|$.

For example, considering the curve with the record number of known rational points,
we get an improvement from~$7.726$ to~$0.973$ for the upper bound on~$-\tilde{\mu}$
using Lemma~\ref{L:arch1}. See Section~\ref{Ex:hdiff} for more details.
In practice it appears that this second approach is at the same time more
efficient and leads to better bounds than the approach described in Section~\ref{S:htdiffarch1}
above.

The approach described here can also be applied in the context of heights on
genus~3 hyperelliptic Jacobians, see~\cite{Stollg3}.

%=====================================================================================

\section{Optimizing the naive height} \label{S:varnaive}

We now consider an arbitrary local field~$k$, with absolute value~$|{\cdot}|$.
Let $C$ be given by an equation
\[ Y^2 = F(X,Z)\, , \]
and let $W$ be the canonical class on $C$.
The first three coordinates of the image of a point
$P = [(X_1 : Y_1 : Z_1) + (X_2 : Y_2 : Z_2)] - W \in J$ on the Kummer surface
are given by $Z_1 Z_2$, $X_1 Z_2 + Z_1 X_2$, $X_1 X_2$, whereas the fourth coordinate is homogeneous
of degree~$1$ in the coefficients~$f_j$ of~$F$ (if we consider $Y_1$ and~$Y_2$
to be of degree~$1/2$).
This has the effect that the fourth coordinate usually differs by a factor of
about~$\|F\| \colonequals \max\{|f_0|, |f_1|, \ldots, |f_6|\}$ from the other three,
which gives this last coordinate
a much larger (when $\|F\|$ is large; this is usually the case when $k$ is archimedean)
or smaller (this may occur when $k$ is non-archimedean) influence on the local contribution
to the naive height when $k=K_v$ and $K$ is a global field.
This imbalance tends to increase the difference $h_\std - \hat{h}$ between naive and
canonical height. This observation suggests to modify the naive height in the following way,
so as to give all coordinates roughly the same weight. Compare Section~\ref{S:genhts}
for the general set-up.
Let $x$ be a set of Kummer coordinates over a global field~$K$ and set
\[ h'(x) \colonequals \sum_{v\in M_K} \log \max\bigl\{|x_1|_v, |x_2|_v, |x_3|_v,
|x_4|_v/\|F\|_v\bigr\} \,. \]
This is a height as in Example~\ref{Ex:nonstdht}.

We state the following simple result, which will help us use this modified height.

\begin{lemma} \label{L:delta-twist}
  Let $F_0 \in k[X,Z]$ be squarefree and homogeneous of degree~$6$.
  For $c \in k^\times$, let $C^{(c)}$ denote the curve $Y^2 = c F_0(X,Z)$.
  The Kummer surfaces $\KS^{(1)}$ of~$C^{(1)}$ and $\KS^{(c)}$ of~$C^{(c)}$ are isomorphic via
  \[ \iota \colon \KS^{(1)} \To \KS^{(c)}, \qquad
                  (x_1 : x_2 : x_3 : x_4) \longmapsto (x_1 : x_2 : x_3 : c x_4) \,.
  \]
  We abuse notation and write $\iota$ also for the linear map
  $(x_1,x_2,x_3,x_4) \mapsto (x_1,x_2,x_3,cx_4)$.
  Write $\delta^{(c)}$ for the duplication polynomials on~$\KS^{(c)}$. Then
  \[ \delta^{(c)}(\iota(x)) = c^3 \iota(\delta^{(1)}(x)) \qquad\text{for each $x \in \KS^{(1)}_{\A}$.} \]
\end{lemma}

\begin{proof}
  This can be checked by an easy calculation.
\end{proof}

If $k$ is non-archimedean and we use the modified local height given by
\[ h'_v(x) =  \log \max\bigl\{|x_1|_v, |x_2|_v, |x_3|_v, |x_4|_v/\|F\|_v\bigr\} \,, \]
then we need to change the definition of~$\eps$ accordingly (compare Lemma~\ref{L:telescope}):
\begin{align*}
  \eps(x) &= \min\{v(\delta_1(x)), v(\delta_2(x)), v(\delta_3(x)), v(\delta_4(x)) - v(F)\} \\
          &\qquad{} - 4 \min\{v(x_1), v(x_2), v(x_3), v(x_4) - v(F)\} \,,
\end{align*}
where $v(F) = v(\{f_0,\ldots,f_6\})$.
By Lemma~\ref{L:delta-twist} with $c = \pi^{v(F)}$, where $\pi$ is a uniformizer
of~$k$, and $F_0 = c^{-1} F$, we then have, denoting the objects associated to~$F_0$
by $\delta_0$, $\eps_0$ and~$\mu_0$,
\[ \eps(x) = v\bigl(\iota^{-1}(\delta(x))\bigr) - 4 v(\iota^{-1}(x))
           = v\bigl(c^3 \delta_0(\iota^{-1}(x))\bigr) - 4 v(\iota^{-1}(x))
           = 3 v(F) + \eps_0(\iota^{-1}(x)) \,.
\]
This implies $\mu(x) = v(F) + \mu_0(\iota^{-1}(x))$.
Let $C_0$ be the curve given by $Y^2 = F_0(X,Z)$. We then get that
\[ \beta(C) \le v(F) + \bar{\beta}(C_0) \,. \]
Note that the Jacobians of $C$ and~$C_0$ are in general only isomorphic
over the ramified quadratic extension $k(\sqrt{\pi})$, so we cannot
necessarily use $\beta(C_0)$ here. If $v(F)$ is even, however, then
the isomorphism is defined over~$k$, and we have $\beta(C) = v(F) + \beta(C_0)$.

So, except for the correction term~$v(F)$, the effect is that we
use the Kummer surface associated to the quadratic twist~$C_0$ of~$C$,
which has a primitive polynomial on the right hand side of its equation.
Note in addition that this also allows us to deal with non-integral equations;
in this case, we again implicitly scale to make the polynomial on the right
integral and primitive.

When $k = K_v \cong \Q_2$ (say) and we can write $F = 4 F_1 + H^2$ with binary forms $F_1$
and~$H$ with integral coefficients, then $C$ is isomorphic to the curve~$C'$
given by the Weierstrass equation
\[ Y^2 + H(X,Z) Y = F_1(X,Z) \,, \]
and we can use the Kummer surface of the
latter to define the local contribution to the naive height. The isomorphism
between the Kummer surfaces is given by (see~\cite{MuellerKummer}*{p.~53};
note that this is the inverse of the map given there)
\[ (x_1 : x_2 : x_3 : x_4) \longmapsto
      \bigl(x_1 : x_2 : x_3 : \tfrac{1}{4} x_4 + \tfrac{1}{2} (h_0 h_2 x_1 + h_0 h_3 x_2 +
      h_1 h_2 x_3)\bigr) \,.
\]
The scaling factor this induces for the $\delta$ polynomials is $2^6$ in this case.
So defining the local component at~$v$ of~$h'(x)$ to be
\[ \log \max \bigl\{|x_1|_v, |x_2|_v, |x_3|_v,
                \bigl|\tfrac{1}{4}x_4 + \tfrac{1}{2}(h_0 h_2 x_1 + h_0 h_3 x_2 + h_1 h_2
              x_3)\bigr|_v\bigr\}\,,
\]
we can replace the bound for~$\mu_v$ by the bound we get on~$C'$ plus~$2$.
If we use this at the places above~$2$ where it applies (instead of, or combined with,
the scaling described above), we still obtain a height as in Example~\ref{Ex:nonstdht}.

If $v$ is an archimedean place, then the approach described in
Section~\ref{S:htdiffarch2} above can easily be adapted to the modified naive height.
We just have to replace $b_{\{S,S'\},4} = 1$ by $\|F\|_v$ and $a_{4,\{S,S'\}}$
by $a_{4,\{S,S'\}}/\|F\|_v^2$. This will usually lead to a \emph{negative} upper
bound for~$\tilde{\mu}_v$, which is fairly close to $-\log \|F\|_v$,
at least when $F$ is reduced in the sense of~\cite{CremonaStoll} and its roots
are not too close together. This is because the scaled $a_{i,\{S,S'\}}$ are now
all of size $\approx \|F\|_\infty^{-2}$ and the scaled $b_{\{S,S'\},j}$ are all
of size $\approx \|F\|_\infty$, so $\Phi$ as in the proof of Lemma~\ref{L:arch1}  roughly satisfies
$\|\Phi(x)\|_\infty \approx -\frac{3}{4} \log \|F\|_\infty + \frac{1}{4} \|x\|_\infty$,
which has $-\log \|F\|_\infty$ as its fixed point.

Note that for a point $(0 : 0 : 0 : 1) \neq P = (x_1 : x_2 : x_3 : x_4) \in \KS(K)$
we have, for all versions $h'$ of the modified height,
\[ h_\std\bigl((x_1 : x_2 : x_3)\bigr) \le h'(P)\,. \]
We will therefore find all points~$P$
with $h'(P) \le B$, if we can enumerate all $P$ with $h_\std((x_1 : x_2 : x_3)) \le B$.
This can be done (over~$\Q$) by using the \texttt{-a}~option of the second author's
program~\texttt{j-points}, which is available at~\cite{j-points}.
(This option is also available in Magma version 2.22 or later.)
In this way, enumerating all points as above with $B$ up to roughly $\log 50\,000$
is feasible. See the discussion in Section~\ref{S:enum} below.

Note that it is quite possible that we end up with a bound
\[ h_\std\bigl((x_1 : x_2 : x_3)\bigr) \le  h'(P) \le \hat{h}(P) + \tilde{\beta} \qquad
   \text{for all $P \in J(\Q) \setminus \{O\}$}
\]
with $\tilde{\beta} < 0$. In this case $-\tilde{\beta}$ is a lower bound on the canonical height
of any nontrivial point in~$J(\Q)$; in particular, the torsion subgroup of~$J(\Q)$
must be trivial. To give an indication of when we can expect $\tilde{\beta}$
to be close to zero or negative, write $|2^4 \disc(F)| = D D'$ with $D$ and~$D'$ coprime
and $D'$ squarefree and odd. Then the contribution of the finite places to~$\tilde{\beta}$
can be bounded by $\frac{1}{4} \log D$, and we get
$\tilde{\beta} \approx -\log \|F\|_\infty + \frac{1}{4} \log D$.
So if $D \ll \|F\|_\infty^4$, we are in good shape. Note that $|\disc(F)| \ll \|F\|_\infty^{10}$,
so this means that $60\%$ or more of~$\log |\disc(F)|$ comes from primes~$p$ dividing
the discriminant exactly once. For curves that are not very special this is very likely
to be the case.

In Section~\ref{Ex:chnaive} we show how this approach can be used to get a
very small bound for the height difference even for a curve with ten-digit coefficients.

%=======================================================================================

\section{Efficient enumeration of points of bounded canonical height} \label{S:enum}

Let $C \colon y^2 = f(x)$ be a curve of genus~$2$ over~$\Q$ with Jacobian~$J$.
In this section we describe the algorithm for enumerating all points $P \in J(\Q)$
with $\hat{h}(P) \le B$ that follows from the considerations above.
We assume that $f \in \Z[x]$ and proceed as follows.

\begin{enumerate}[1.]\addtolength{\itemsep}{1mm}
  \item Compute the complex roots of~$f$ numerically.
  \item Compute the coefficients $a_{i,\{S,S'\}}$ and~$b_{\{S,S'\},j}$
        from the roots and the leading coefficient of~$f$ according
        to the formulas given in~\cite{StollH1}*{Section~10}.
  \item Multiply all $a_{4,\{S,S'\}}$ by $\|f\|_\infty^{-2}$ and
        multiply all $b_{\{S,S'\},4}$ by $\|f\|_\infty$.
  \item Iterate the function~$\varphi$ from Section~\ref{S:varnaive}
        (but using the modified coefficients)
        a number of times, starting at~$(1,1,1,1)$, until there is little change;
        let $\tilde{\beta}_\infty$ be the upper bound for~$\tilde{\mu}_\infty$
        as in Lemma~\ref{L:arch1}.
  \item Factor the discriminant of~$f$. \\
        Let $g$ be the $\gcd$ of the coefficients of~$f$. % and set $f_1(x) = f(x)/g$.
  \item For each prime divisor~$p$ of $2 \disc(f)$, do the following.
        \begin{enumerate}[a.]
          \item Let $e_p$ be the $p$-adic valuation of~$g$ and set $f_1 = p^{-e_p} f$.
          \item If $p = 2$ and $f_1 = h^2 + 4 f_2$ for polynomials $f_2, h \in \Z[x]$,
                set $C_1 \colon y^2 + h(x) y = f_2(x)$ and replace $g$ by $4g$;
                otherwise set $C_1 \colon y^2 = f_1(x)$. Let $J_1$ be the Jacobian of~$C_1$.
          \item If $e_p$ is even, let $\beta_p$ be the bound for~$\mu_p$ on~$J_1(\Q_p)$ as
                obtained in Part~II.
                Otherwise, let $\beta_p$ be the bound for~$\mu_p$ on~$J_1(\bar{\Q}_p)$.
        \end{enumerate}
  \item Set $\tilde{\beta} = \tilde{\beta}_\infty + \sum_p \beta_p \log p + \log g$.
  \item Use \texttt{j-points} with the \texttt{-a} option to enumerate all points
        $O \neq P \in J(\Q)$ such that
        $h_\std\bigl((\kappa_1(P) : \kappa_2(P) : \kappa_3(P))\bigr) \le B + \tilde{\beta}$.
  \item Add $O$ to this set and return it.
\end{enumerate}

Note that $\log g$ is the sum of the correction terms $v_p(f) \log p$.
% Adding $\log g$ in Step~7 has the same effect as working with $f/g$ at the real place.
% So alternatively, we can work with~$f_1$ also at the real place; then
% the addition of $\log g$ in Step~7 can be skipped (except when $g$ has been
% modified in Step~6b; then $\log 4$ has still to be added).

It follows from the discussion in the previous sections that the set returned by this
algorithm contains all points with canonical height at most~$B$. If necessary, one can
compute the actual canonical heights using the algorithm from Part~III and discard the
points whose height is too large.

The actual enumeration is done by running through all points $(x_1 : x_2 : x_3) \in \BP^2$
of (standard) height at most $B + \tilde{\beta}$ and checking whether there are rational numbers~$x_4$
such that $(x_1 : x_2 : x_3 : x_4)$ is on the Kummer surface. For each of these points
on the Kummer surface, we then check if it lifts to the Jacobian. Both these conditions
are equivalent to some expression in the coordinates (and the coefficients of~$f$)
being a square. \texttt{j-points}
tries to do this efficiently by using information modulo a number of primes to filter
out triples that do not lift to rational points on~$J$.
Let $N = \lfloor \exp(B + \tilde{\beta}) \rfloor$.
Then \texttt{j-points} usually takes a couple of seconds when $N = 1000$, a few minutes
when $N = 5\,000$ and a few days when $N = 50\,000$. The running time scales with~$N^3$,
but the scaling factor depends on how effective the sieving mod~$p$ is. For Jacobians
of high rank, the program tends to take longer than for `random' Jacobians.

Since the running time depends exponentially on~$B + \tilde{\beta}$, it is very
important to obtain a small bound~$\tilde{\beta}$ for the difference between naive
and canonical height. The improvement at the infinite place that we can achieve by
considering a modified naive height is crucial for making the enumeration feasible
also in cases when the defining polynomial has large coefficients. This is demonstrated
by the example in Section~\ref{Ex:hdiff} below.

If the discriminant of~$f$ is too large to be factored, then one can use
\[ \tilde{\beta} = \tilde{\beta}_\infty + \frac{1}{4} \log |\disc(f_1)| + \log g \]
(or use information from small prime divisors as in the algorithm above and
$\frac{1}{4} \log D$ for the remaining primes, where $D$ is the unfactored part
of the discriminant). But note that it is usually a great advantage to know the
bad primes, since we can take $\beta_p = 0$ for primes~$p$ such that $v_p(\disc(f)) = 1$.
In most cases, this leads to a much smaller bound~$\tilde{\beta}$.

One of the most important applications of this enumeration algorithm is its
use in saturating a given finite-index subgroup of~$J(\Q)$, which gives (generators of)
the full group~$J(\Q)$. This is a necessary ingredient for the method for obtaining
all integral points on~$C$ developed in~\cite{BMSST08}, for example, and for
computing the regulator of~$J(\Q)$.

There are essentially two ways of performing the saturation. Let $G \subset J(\Q)$ denote
the known subgroup.
\begin{enumerate}[(i)]\addtolength{\itemsep}{2mm}
  \item Let $\rho$ be (an upper bound for) the covering radius of the lattice
        $\Lambda = (G/G_{\tors}, \hat{h})$. Then $J(\Q)$ is generated by~$G$ together with
        all points~$P \in J(\Q)$ that satisfy \hbox{$\hat{h}(P) \le \rho^2$},
        see~\cite{StollH2}*{Prop.~7.1}. This approach is feasible
        when $\tilde{\beta} + \rho^2$ is sufficiently small.
  \item Let $I = (J(\Q) : G)$ denote the index; we assume that $J(\Q)_{\tors} \subset G$.
        If $m_1, \ldots, m_r$ are the successive minima of~$\Lambda$
        and there are no points $P \in J(\Q) \setminus G$ with $\hat{h}(P) < B$, then
        \[ I \le \sqrt{\frac{R \cdot\gamma_r^r}{\prod_{j=1}^r \min\{m_j, B\}}} \,; \]
        see~\cite{FlynnSmart}*{Section~7}. Here $\gamma_r$ is (an upper bound for)
        the Hermite constant for lattices of rank~$r$ and $R$ is the regulator of~$G$
        (i.e., the determinant of the Gram matrix of any basis of~$\Lambda$).
        This can be used to get a
        bound on~$I$ whenever $B$ is strictly positive, so for the enumeration
        we only need $\tilde{\beta}$ to be sufficiently small. (If $\tilde{\beta} < 0$,
        then we can do entirely without enumeration to get an index bound.)
        In a second step, one then has to check that $G$ is $p$-saturated in~$J(\Q)$
        (or find the largest group $G \subset G' \subset J(\Q)$ with $(G' : G)$
        a power of~$p$) for all primes~$p$ up to the index bound. This can be done
        by considering the intersection of the kernels of the maps
        $J(\Q)/p J(\Q) \to J(\F_q)/p J(\F_q)$ for a set of good primes~$q$ (such that
        the group on the right is nontrivial). If this intersection is trivial,
        then $G$ is $p$-saturated; otherwise it tells us where to look for points
        that are potentially divisible by~$p$.
        Since the index bound gets smaller with increasing~$B$ (as long as $B < m_r$),
        it makes sense to pick $B$ in such a way as to balance the time spent in
        the two steps of this approach.
\end{enumerate}

%=======================================================================================

\section{Example} \label{Ex:hdiff} \label{Ex:chnaive}

As an example that demonstrates the use of our nearly optimal upper bound
for the difference $h - \hat{h}$ between naive and canonical height (which is
based on the optimal bounds for the~$\mu_p$ obtained in Sections \ref{formulas},
\ref{formulas2} and~\ref{UpperBeta} and the variation of the naive height discussed
in Section~\ref{S:varnaive}), we consider the curve
\begin{align*}
  C \colon y^2 &= 82342800 x^6 - 470135160 x^5 + 52485681 x^4 \\
               &\qquad{} + 2396040466 x^3 + 567207969 x^2 - 985905640 x + 247747600 \,.
\end{align*}
This curve is of interest, since it holds the current record for the largest
number of known rational points (which is~$642$ for this curve), see~\cite{record}.
A 2-descent on its Jacobian~$J$ (assuming GRH) as described in~\cite{Stoll2desc}
and implemented in Magma gives an upper bound of~$22$ for the
rank of~$J(\Q)$, and the differences of the known rational points generate a group
of rank~$22$. The latter statement can be checked by computing the determinant~$R$
of the height pairing matrix of the $22$~points in~$J(\Q)$ listed in
Table~\ref{TableGens}, which is
fairly fast using the algorithm for computing canonical heights described in
Section~\ref{S:cch}. The points are given in Mumford representation $(a(x), b(x))$,
which stands for $[(\theta_1, b(\theta_1)) + (\theta_2, b(\theta_2))] - W$,
where $\theta_1$, $\theta_2$ are the two roots of~$a(x)$ and $W$ is the canonical class.
Not all of these points are differences of rational points, but they are linear
combinations of such differences.

We can easily check that $J(\Q)$ has trivial torsion subgroup
by computing the order of~$J(\F_p)$ for a few good primes~$p$.

\begin{table}[htb]
\hrulefill
\begin{align*}
  (x^2 + x, 18868 x + 15740), &\quad (x^2 - \tfrac{1}{3} x, \tfrac{216800}{3} x - 15740), \\
  (x^2 + \tfrac{2}{3} x - \tfrac{1}{3}, \tfrac{11747}{3} x + \tfrac{21131}{3}), &\quad
  (x^2 + 5 x + 4, 276256 x + 273128), \\
  (x^2 + \tfrac{4}{3} x - \tfrac{5}{9}, 16315 x + \tfrac{26195}{9}), &\quad
  (x^2 + \tfrac{53}{12} x + \tfrac{5}{3}, \tfrac{1433669}{6} x + \tfrac{371650}{3}), \\
  (x^2 - 3 x - 4, 34104 x + 30976), &\quad (x^2 - 4 x - 5, 65987 x + 69115), \\
  (x^2 + \tfrac{8}{5} x + \tfrac{3}{5}, 67671 x + 64543), &\quad
  (x^2 - 5 x - 6, \tfrac{883626}{7} x + \tfrac{905522}{7}), \\
  (x^2 - \tfrac{3}{4} x - \tfrac{7}{4}, 31875 x + 35003), &\quad
  (x^2 + \tfrac{5}{7} x - \tfrac{2}{7}, \tfrac{432898}{49} x + \tfrac{279626}{49}), \\
  (x^2 + \tfrac{29}{6} x - \tfrac{178}{9}, \tfrac{3014179}{6} x - \tfrac{10824742}{9}), &\quad
  (x^2 + \tfrac{19}{84} x - \tfrac{65}{84}, \tfrac{4287373}{294} x + \tfrac{5207005}{294}), \\
  (x^2 + \tfrac{97}{42} x - \tfrac{37}{42}, \tfrac{23742013}{294} x - \tfrac{5459431}{294}), &\quad
  (x^2 - \tfrac{5}{11} x, \tfrac{1089388}{121} x - 15740), \\
  (x^2 + \tfrac{325}{84} x - \tfrac{11}{21}, \tfrac{30014567}{147} x - \tfrac{2230444}{147}), &\quad
  (x^2 - \tfrac{683}{140} x - \tfrac{279}{140}, \tfrac{45519013}{490} x + \tfrac{5478709}{490}), \\
  (x^2 - \tfrac{91}{769} x - \tfrac{584}{769},
   \tfrac{6911886712}{591361} x + \tfrac{16665656516}{591361}), &\quad
  (x^2 - \tfrac{259}{96} x + \tfrac{163}{72}, \tfrac{52305719}{768} x - \tfrac{13101271}{576}), \\
  (x^2 - \tfrac{3073}{2307} x - \tfrac{1252}{769},
   \tfrac{54505985456}{1774083} x + \tfrac{25990632928}{591361}), &\quad
  (x^2 - \tfrac{137}{51} x + \tfrac{40}{51}, \tfrac{47131040}{867} x - \tfrac{8471860}{867})
\end{align*}
\hrulefill
\caption{Generators of the known part of $J(\Q)$.} \label{TableGens}
\end{table}

The discriminant of~$C$ factors as
\begin{align*}
  \Delta &=  2^{47} \cdot 3^5 \cdot 5^9 \cdot 11^2 \cdot 13^2 \cdot 17^6 \cdot 19^4 \cdot 23^2
              \cdot 41^4 \cdot 73^3 \\
         &\qquad {}\cdot 2707 \cdot 43579 \cdot 108217976921 \cdot 8723283517315751077 \,.
\end{align*}
The results of~\cites{StollH1,StollH2} lead to a bound of
\begin{align*}
  \frac{1}{3} &\bigl(43 \log 2 + 3 \log 3 + 9 \log 5 + 2 \log 11 + 2 \log 13 \\
              &\qquad{} + 6 \log 17 + 4 \log 19 + 2 \log 23 + 4 \log 41 + 3 \log 73\bigr)
      \approx 40.1
\end{align*}
for the contribution of the finite places to the height difference bound. When trying
to get a better bound (for~$\gamma_p$) by essentially doing an exhaustive search over
the $p$-adic points of the Kummer surface, Magma gets stuck at~$p=2$ for a long while,
but eventually finishes with a contribution of~$26.434$ from the finite places
and a total bound of~$34.163$. This contribution turns out to be $(\gamma_p/3) \log p$
in all cases except for $p = 73$, where it is $\frac{2}{3} \log 73$ instead of~$\frac{1}{3} \log 73$.
Our new results from this paper give bounds on the local contributions
as shown in Table~\ref{TableBeta}.
$\Phi_p$ is the component group ($\eps$ and~$\mu$ factor through it in all cases)
and `gain' gives the gain in the bound on the height difference obtained by
using the optimal bound on~$\mu$ versus the
bound~$\gamma/3$, where $\gamma$ is the maximum of the values of~$\eps$.

\begin{table}[htb]
\[ \renewcommand{\arraystretch}{1.2}
   \begin{array}{|r|c|c|c|c|c|} \hline
      p & \text{reduction type} & \Phi_p & \beta_p        & \gamma_p/3  & \text{gain} \\\hline
      2 & [I_{10-9-8}]          & \Z/242\Z & 2 + 1145/242 & 26/3        & 1.341 \\
      3 & [I_0-IV-0]            & \Z/3\Z & 2/3            & 2/3         & 0.000 \\
      5 & [I_{4-3-2}]           & \Z/26\Z & 22/13         & 2           & 0.495 \\
     11 & [I_{2-0-0}]           & \Z/2\Z & 1/2            & 2/3         & 0.400 \\
     13 & [I_{2-0-0}]           & \Z/2\Z & 1/2            & 2/3         & 0.427 \\
     17 & [I_{2-2-2}]           & \Z/2\Z \times \Z/6\Z & 1 & 4/3        & 0.944 \\
     19 & [I_{2-1-1}]           & \Z/5\Z & 3/5            & 2/3         & 0.196 \\
     23 & [I_{2-0-0}]           & \Z/2\Z & 1/2            & 2/3         & 0.523 \\
     41 & [I_{2-1-1}]           & \Z/5\Z & 3/5            & 2/3         & 0.248 \\
     73 & [I_{1-1-1}]           & \Z/3\Z & 1/3            & 1/3         & 0.000 \\\hline
    \end{array}
\]
\caption{Bounds for $\beta_p$.} \label{TableBeta}
\end{table}

This now gives a bound of $\approx 20.429$ for the contribution of the finite
places. The optimization of the naive height does not give any improvement
at the odd finite places, since the polynomial~$f$ defining the curve is primitive.
On the other hand, we note that $f$ is congruent
to a square mod~$4$, so we could use the Kummer surface of the curve
$y^2 + (x^2 + x) y = f_1(x)$ (where $f(x) = 4 f_1(x) + (x^2 + x)^2$)
for the local height at~$2$, but this results in no improvement,
since we have already used a minimal model to get our bound.

Now we consider the contribution of the infinite place. The bound obtained
from~\cite{StollH1}*{(7.1)} is~$7.726$. Using Lemma~\ref{L:arch1} with~$N = 10$
improves this to~$0.973$; increasing~$N$ further gives no significant improvement.
However, modifying the local height at the infinite place by scaling
the contribution of the fourth coordinate by~$\|f\|_\infty^{-1}$ reduces this
bound drastically to $\tilde{\mu}_\infty \le -19.25654$
(compare this to $-\log \|f\|_\infty \approx -21.59708$).
This finally gives
\[ h'(P) \le \hat{h}(P) + 1.17273 \]
for our modified naive height~$h'$.

So if we enumerate all points $P \in J(\Q)$ with $h'(P) \le \log N$ and do not find
points that are not in the known subgroup~$G$, then we obtain a bound for
the index $I = (J(\Q) : G)$ as follows (see the discussion at the end
of Section~\ref{S:enum}).
\[ I \le \sqrt{\frac{R \cdot\gamma_{22}^{22}}%
                    {\prod_{j=1}^{22} \min\{m_j, \log N  - 1.17273\}}} \,, \]
where $R$ is the regulator of~$G$ and $m_1, m_2, \ldots, m_{22}$ are the
successive minima of the lattice~$(G, \hat{h})$, which are
\begin{gather*}
  8.5276, \; 8.5668, \; 8.5956, \; 8.8594, \; 9.0256, \; 9.0776, \; 9.1426, \; 9.1753, \\
  9.4456, \; 9.7428, \; 9.7747, \; 9.9047, \; 9.9465, \; 9.9611, \; 9.9704, \; 10.1408, \\
  10.3472, \; 10.3784, \; 10.5284, \; 10.5356, \; 10.6318, \; 10.9287 \,.
\end{gather*}
With $N = 10\,000$ we obtain $I \le 6842$, with $N = 20\,000$ we get $I \le 2835$
and with $N \ge {178\,245}$ we obtain the best possible bound $I \le 900$.
We checked that there are no unknown points~$P$ with $\kappa(P) = (x_1 : x_2 : x_3 : x_4)$
such that $h_\std((x_1 : x_2 : x_3)) \le \log 20\,000$ and verified that the index is
not divisible by any prime $p \le {2835}$. The first computation took about two days
on a single core, the second less than half a day. This implies the following.

\begin{prop} \label{P:record}
  Assume the Generalized Riemann Hypothesis. Let
  \begin{align*}
    C \colon y^2 &= 82342800 x^6 - 470135160 x^5 + 52485681 x^4 \\
                &\qquad{} + 2396040466 x^3 + 567207969 x^2 - 985905640 x + 247747600
  \end{align*}
  and denote by~$J$ the Jacobian of~$C$. Then $J(\Q)$ is a free abelian group
  of rank~$22$, freely generated by the points listed in Table~\ref{TableGens}.
  In particular, $J(\Q)$ is generated by the differences of rational points on~$C$.
\end{prop}

%%%%%%%%%%%%%%%%%%%%%%%%%%%%%%%%%%%%%%%%%%%%%%%%%%%%%%%%%%%%%%%%%%%%%%%%%%%%%%%%%%%

\end{document}